\documentclass[11pt]{aart_big}

\usepackage{titlesec}
\titleformat{\subsubsection}[runin]{\bfseries}{\thesubsubsection.}{.3em}{}[.]\titlespacing{\subsubsection}{0pt}{1ex plus .1ex minus .2ex}{.5em}

\usepackage[labelfont=bf,font=small,labelsep=period]{caption}
\usepackage[letterpaper, hmargin=1in, top=1.1in, bottom=1.3in, footskip=0.7in]{geometry}

\usepackage{amsmath} 
\usepackage{amssymb}
\usepackage{amsfonts}
\usepackage{latexsym}
\usepackage{amsthm}
\usepackage{amsxtra}
\usepackage{amscd}
\usepackage{bbm}
\usepackage{mathrsfs}
\usepackage{bm}
\usepackage{subfigure}

\usepackage{graphicx, color}

\definecolor{darkred}{rgb}{0.9,0,0.3}
\definecolor{darkblue}{rgb}{0,0.3,0.9}

\usepackage{ifthen}
\def\comment#1{\ifthenelse{\isodd{\value{page}}}{\marginpar{\raggedright\scriptsize{\textcolor{darkred}{#1}}}}{\marginpar{\raggedleft\scriptsize{\textcolor{darkred}{#1}}}}}  

\definecolor{vdarkred}{rgb}{0.6,0,0.2}
\definecolor{vdarkblue}{rgb}{0,0.2,0.6}
\usepackage[pdftex, colorlinks, linkcolor=vdarkblue,citecolor=vdarkred]{hyperref}

\graphicspath{{./Figures/}}

\usepackage[nottoc,notlof,notlot]{tocbibind} 

\numberwithin{equation}{section}
\numberwithin{figure}{section}
\newcommand{\f}[1]{\boldsymbol{\mathrm{#1}}} 
 \newcommand{\rr}{\mathrm} 
\renewcommand{\cal}{\mathcal} 
 
\newcommand{\fra}{\mathfrak} 
\newcommand{\ol}[1]{\overline{#1} \!\,} 
\newcommand{\wh}{\widehat}
\newcommand{\wt}{\widetilde}
\newcommand{\txt}[1]{\text{\rm{#1}}}

\renewcommand{\P}{\mathbb{P}}
\newcommand{\E}{\mathbb{E}}
\newcommand{\R}{\mathbb{R}}
\newcommand{\C}{\mathbb{C}}
\newcommand{\N}{\mathbb{N}}
\newcommand{\Z}{\mathbb{Z}}

\newcommand{\me}{\mathrm{e}}
\newcommand{\ii}{\mathrm{i}}
\newcommand{\dd}{\mathrm{d}}
\newcommand{\col}{\mathrel{\vcenter{\baselineskip0.75ex \lineskiplimit0pt \hbox{.}\hbox{.}}}}
\newcommand*{\deq}{\mathrel{\vcenter{\baselineskip0.65ex \lineskiplimit0pt \hbox{.}\hbox{.}}}=}

\newcommand{\todist}{\overset{\mathrm{d}}{\longrightarrow}}
\renewcommand{\leq}{\leqslant}
\renewcommand{\geq}{\geqslant}
\renewcommand{\le}{\leqslant}
\renewcommand{\ge}{\geqslant}
\renewcommand{\epsilon}{\varepsilon}

\newcommand{\qq}[1]{[\![{#1}]\!]}

\newcommand{\ind}[1]{\f 1 (#1)}
\newcommand{\indb}[1]{\f 1 \pb{#1}}

\newcommand{\pb}[1]{\bigl({#1}\bigr)}
\newcommand{\pB}[1]{\Bigl({#1}\Bigr)}
\newcommand{\pbb}[1]{\biggl({#1}\biggr)}
\newcommand{\pBB}[1]{\Biggl({#1}\Biggr)}

\newcommand{\qb}[1]{\bigl[{#1}\bigr]}
\newcommand{\qB}[1]{\Bigl[{#1}\Bigr]}
\newcommand{\qbb}[1]{\biggl[{#1}\biggr]}
\newcommand{\qBB}[1]{\Biggl[{#1}\Biggr]}

\newcommand{\h}[1]{\{{#1}\}}
\newcommand{\hb}[1]{\bigl\{{#1}\bigr\}}

\newcommand{\hbb}[1]{\biggl\{{#1}\biggr\}}

\newcommand{\abs}[1]{\lvert #1 \rvert}
\newcommand{\absb}[1]{\bigl\lvert #1 \bigr\rvert}
\newcommand{\absB}[1]{\Bigl\lvert #1 \Bigr\rvert}
\newcommand{\absbb}[1]{\biggl\lvert #1 \biggr\rvert}
\newcommand{\absBB}[1]{\Biggl\lvert #1 \Biggr\rvert}

\newcommand{\norm}[1]{\lVert #1 \rVert}

\newcommand{\normbb}[1]{\biggl\lVert #1 \biggr\rVert}

\newcommand{\scalar}[2]{\langle{#1} \mspace{2mu}, {#2}\rangle}

\DeclareMathOperator{\tr}{Tr}
\DeclareMathOperator{\var}{Var}
\DeclareMathOperator{\supp}{supp}
\DeclareMathOperator{\real}{Re}
\DeclareMathOperator{\im}{Im}

\DeclareMathOperator{\dist}{dist}

\DeclareMathOperator{\spec}{spec}

\newcommand{\ga}{\gamma} 
\newcommand{\la}{\label}
\newcommand{\eqre}{\eqref}
\newcommand{\re}{\ref}
\newcommand{\ld}{\ldots}
\newcommand{\beg}{\begin}
\newcommand{\en}{\end}
\newcommand{\trm}{\textrm}
\newcommand{\bgt}{\begin{itemize}}
\newcommand{\ent}{\end{itemize}}
\newcommand{\ite}{\item} 
\newcommand{\op}{\operatorname}

\newcommand{\si}{\sigma}
\newcommand{\ds}{\displaystyle}

\newcommand{\Tr}{\operatorname{Tr}}
\newcommand{\p}{\mathbb{P}}

\newcommand{\ud}{\mathrm{d}}
\newcommand{\blem}{\begin{lem}}
\newcommand{\elem}{\end{lem}}
\newcommand{\pro}{probability }
\newcommand{\ff}{\frac{1}}
\newcommand{\lf}{\left}
\newcommand{\ri}{\right}
\newcommand{\st}{such that }
 \newcommand{\lam}{\lambda}
\newcommand{\La}{\Lambda}
\newcommand{\ti}{\times}

\newcommand{\ste}{\, ;\, }
\newcommand{\mc}{\mathcal }
\newcommand{\eps}{\varepsilon}
\newcommand{\al}{\alpha}
\newcommand{\tta}{\theta}
\newcommand{\Tta}{\Theta} 

\newcommand{\bbm}{\begin{bmatrix}}
\newcommand{\ebm}{\end{bmatrix}}
\newcommand{\bes}{\begin{equation*}}
\newcommand{\ees}{\end{equation*}}
\newcommand{\be}{\begin{equation}}
\newcommand{\ee}{\end{equation}}
\newcommand{\beqy}{\begin{eqnarray}}
\newcommand{\eeqy}{\end{eqnarray}}
\newcommand{\beq}{\begin{eqnarray*}}
\newcommand{\eeq}{\end{eqnarray*}}

\newcommand{\lto}{\longrightarrow}

\newcommand{\bpm}{\begin{pmatrix}}
\newcommand{\epm}{\end{pmatrix}}
\newcommand{\cd}{\cdots}
\newcommand{\wH}{\widetilde{H}}
\newcommand{\bpr}{\beg{proof}}
\newcommand{\epr}{\en{proof}}
\newcommand{\bet}{\beta}
\newcommand{\del}{\delta}

\newcommand{\ka}{\kappa}
\newcommand{\lan}{\langle}
\newcommand{\rang}{\rangle}

\newtheorem{Th}{Theorem}[section]

\newtheorem{propo}[Th]{Proposition}
\newtheorem{proposition}[Th]{Proposition} 
 
\newtheorem{lem}[Th]{Lemma}

\newtheorem{ex}[Th]{Example}
\newtheorem{Def}[Th]{Definition}

\theoremstyle{definition}
\newtheorem{rmk}[Th]{Remark}
\long\def\symbolfootnote[#1]#2{\begingroup
\def\thefootnote{\fnsymbol{footnote}}\footnote[#1]{#2}\endgroup}

\title{Lectures on the local semicircle law for Wigner matrices}
\author{Florent Benaych-Georges\footnote{Universit\'e Paris Descartes, MAP5. Email: {\tt florent.benaych-georges@parisdescartes.fr}.} \and Antti Knowles\footnote{ETH Z\"urich, Departement Mathematik. Email: {\tt knowles@math.ethz.ch}.}}

\begin{document}
\maketitle

\begin{abstract}
These notes provide an introduction to the local semicircle law from random matrix theory, as well as some of its applications. We focus on \emph{Wigner matrices}, Hermitian random matrices with independent upper-triangular entries with zero expectation and constant variance. We state and prove the local semicircle law, which says that the eigenvalue distribution of a Wigner matrix is close to Wigner's semicircle distribution, down to spectral scales containing slightly more than one eigenvalue. This local semicircle law is formulated using the Green function, whose individual entries are controlled by large deviation bounds.

We then discuss three applications of the local semicircle law: first, complete delocalization of the eigenvectors, stating that with high probability the eigenvectors are approximately flat; second, rigidity of the eigenvalues, giving large deviation bounds on the locations of the individual eigenvalues; third, a comparison argument for the local eigenvalue statistics in the bulk spectrum, showing that the local eigenvalue statistics of two Wigner matrices coincide provided the first four moments of their entries coincide. We also sketch further applications to eigenvalues near the spectral edge, and to the distribution of eigenvectors.
\end{abstract}

\newpage

\tableofcontents

\newpage
  
\section{Introduction}
 
These notes are based on lectures given by Antti Knowles at the conference \emph{\'Etats de la recherche en matrices al\'eatoires} at the Institut Henri Poincar\'e in December 2014. In them, we state and prove the local semicircle law and give several applications to the distribution of eigenvalues and eigenvectors. We favour simplicity and transparency of arguments over generality of results. In particular, we focus on one of the simplest ensembles of random matrix theory: Wigner matrices.

A Wigner matrix is a Hermitian random matrix whose entries are independent up to the symmetry constraint, and have zero expectation and constant variance. This definition goes back to the seminal work of Wigner \cite{Wig}, where he also proved the semicircle law, which states that the asymptotic eigenvalue distribution of a Wigner matrix is given with high probability by the semicircle distribution. Wigner matrices occupy a central place in random matrix theory, as a simple yet nontrivial ensemble on which many of the fundamental features of random matrices, such as universality of the local eigenvalue statistics and eigenvector delocalization, may be analysed. Moreover, many of the techniques developed for Wigner matrices, such as the ones presented in these notes, extend to other random matrix ensembles or serve as a starting point for more sophisticated methods.

The local semicircle law is a far-reaching generalization of Wigner's original semicircle law, and constitutes the key tool for analysing the local distribution of eigenvalues and eigenvectors of Wigner matrices, and in particular in the study of the universality of Wigner matrices. Roughly, it states that the eigenvalue distribution is well approximated by the semicircle distribution down to scales containing only slightly more than one eigenvalue. The first instance of the local semicircle law down to optimal spectral scales was obtained by Erd\H{o}s, Schlein, and Yau in \cite{ESY3}, following several previous results on larger spectral scales \cite{ESY2,Khor}. Since then, the local semicircle law has been improved and generalized in a series of works \cite{ESY4, ESRY,  MR2981427, EYY2, EYYrigi, EKYY1, EKYY4, CMS}. The proof presented in these notes is modelled on the argument of Erd\H{o}s, Knowles, Yau, and Yin from \cite{EKYY4}, which builds on the works \cite{ESY2, ESY3, ESY4,  MR2981427, EYY2, EYYrigi} of Erd\H{o}s, Schlein, Yau, and Yin.

In order to keep these notes focused, for the applications we restrict ourselves to relatively simple consequences of the semicircle law: eigenvalue rigidity, eigenvector delocalization, and a four-moment comparison theorem for the local eigenvalue statistics. Further topics and applications, such as distribution of eigenvalues near the spectral edges, Dyson Brownian motion and its local relaxation time, distribution of eigenvectors, and spectral statistics of deformed matrix ensembles, are not covered here. For further reading on Wigner matrices, we recommend the books \cite{Mehta, agz}. Additional applications of the local semicircle law, in particular in the analysis of the local relaxation time of Dyson Brownian motion, are given in the survey \cite{Erd1}. In Section \ref{sec:GFC_edge_vec}, we briefly outline applications of the local semicircle law to the analysis of eigenvectors and eigenvalues near the spectral edge.

\subsubsection*{Outline}
In Section \ref{sec:LL} we define Wigner matrices and state the local semicircle law, or \emph{local law} for short. We also give two simple consequences of the local law: \emph{eigenvalue rigidity} and \emph{complete eigenvector delocalization}. Sections \ref{sec:prelim}--\ref{sec:averaging} are devoted to the proof of the local law. Section \ref{sec:prelim} collects some basic tools from linear algebra and probability theory that are used throughout the proof. Section \ref{Sec:local_law_proof_abstract} gives a detailed outline of the proof. In Section \ref{sec:weak_law} we perform the first of two major steps of the proof: the weak local law, which yields control down to optimal scales but with non-optimal error bounds. In Section \ref{sec:opt_bounds} we perform the second major step of the proof, which yields optimal error bounds and concludes the proof of the local law. The key estimate used to obtain the optimal error bounds is a fluctuation averaging result, whose proof is given in Section \ref{sec:averaging}.

Having concluded the proof of the local law, in Sections \ref{sec:local_law_small_scales}--\ref{sec:extension} we draw some simple consequences. In Section \ref{sec:local_law_small_scales} we prove the semicircle law on small scales, which provides large deviation bounds on the number of eigenvalues in small intervals. In Section \ref{sec:rig} we prove eigenvalue rigidity, which provides large deviation bounds on the locations of individual eigenvalues. In Section \ref{sec:extension} we extend the estimates from the local law to arbitrary scales and distances from the spectrum.

In Section \ref{sec:comparison} we illustrate how to use the local law to obtain a four-moment comparison theorem for the local eigenvalue statistics, using the Green function comparison method from \cite{MR2981427}. We also sketch how to extend such comparison methods to the edge of the spectrum and to eigenvectors. Finally, in Section \ref{sec:outlook} we discuss further generalizations of the local law, and also other random matrix models for which local laws have been established.

The appendices contain some basic tools from linear algebra, spectral theory, and probability that are used throughout the notes, as well as some standard results on the semicircle distribution and the norm of Wigner matrices.

Pedagogical aspirations aside, in these notes we also give a coherent summary of the different guises of the local semicircle law that have proved useful in random matrix theory. They have all appeared, at least implicitly, previously in the literature; we take this opportunity to summarize them explicitly in Sections \ref{sec:LL} and \ref{sec:extension}.

\subsubsection*{Conventions} 
We use $C$ to denote a generic large positive constant, which may depend on some fixed parameters and whose value may change from one expression to the next. Similarly, we use $c$ to denote a generic small positive constant. For two positive quantities $A_N$ and $B_N$ depending on $N$ we use the notation $A_N \asymp B_N$ to mean $C^{-1} A_N \leq B_N \leq C A_N$ for some positive constant $C$. In statements of assumptions, we use the special constant $\tau > 0$, which should be thought of as a fixed arbitrarily small number. A smaller $\tau$ always leads to a weaker assumption. We always use the Euclidean norm $\abs{\f v}$ for vectors $\f v$, and denote by $\norm{A}$ the associated operator norm of a matrix $A$. Finally, in some heuristic discussions we use the symbols $\ll$, $\gg$, and $\approx$ to mean ``much less than'', ``much greater than'', and ``approximately equal to'', respectively.

 \section{The local law} \label{sec:LL}
Let $H = H^* = (H_{ij} \col 1 \leq i,j \leq N) \in \C^{N \times N}$ be an $N \ti N$ random Hermitian matrix. We normalize $H$ so that its eigenvalues are typically of order one, $\norm{H} \asymp 1$. A central goal of random matrix theory is to understand the distribution of the eigenvalues $\lam_1 \ge \lambda_2 \geq \cd \ge \lam_N$ and the normalized associated eigenvectors $\f u_1, \f u_2, \ld, \f u_N$ of $H$.

Since most quantities that we are interested in depend on $N$, we shall almost always omit the explicit argument $N$ from our notation. Hence, every quantity that is not explicitly a constant is in fact a sequence indexed by $N \in \N$.

\subsection{The Green function} \label{sec:Gfunct}
The main tool in the study of the eigenvalues and eigenvectors of $H$ is the \emph{Green function} or \emph{resolvent}
$$
\qquad G(z) \;\deq\; (H - z)^{-1}\,,
$$
where $z \in \C \setminus \{\lambda_1, \dots, \lambda_N\}$ is the \emph{spectral parameter}. Here, and throughout the following, we identify the matrix $z I$ with the scalar $z$.

Since we are ultimately interested in the eigenvalues and eigenvectors, it might seem that the Green function is an unnecessary distraction. In fact, however, the Green function is a much simpler and more stable object than the eigenvalues and eigenvectors. Using the Green function as the central object instead of eigenvalues and eigenvectors allows one to establish results whose direct proof (i.e.\ working directly on the eigenvalues and eigenvectors) is not feasible or at least much more complicated. We give an informal summary of the key features of the Green function that make it such a powerful tool.

\begin{enumerate}
\item
The Green function contains the complete information about the eigenvalues and eigenvectors. Indeed, by spectral decomposition we have
\begin{equation} \label{spec_decomp}
G(z) \;=\; \sum_{i=1}^N \frac{\f u_i\f u_i^*}{\lam_i-z}\,.
\end{equation}
Hence, $G$ is a matrix-valued meromorphic function in the complex plane, whose poles are the eigenvalues $\lambda_i$ of $H$ and whose residues the spectral projections $\f u_i \f u_i^*$ of $H$.  We may therefore use complex analysis to extract the full information about $\lambda_i$ and $\f u_i$ from $G(z)$. Note that the spectral projection $\f u_i \f u_i^*$ is a more natural object than the eigenvector $\f u_i$, since it discards the ambiguous phase of $\f u_i$. More details are given in Section \ref{sec:comparison}.
\item
Green functions form a basis of a powerful \emph{functional calculus}: every ``reasonable'' function $f(H)$ of $H$ may be expressed as a superposition of Green functions. For instance, if $f$ is holomorphic in an open domain containing the spectrum of $H$, we have
\begin{equation} \label{cont_rep}
f(H) \;=\; - \frac{1}{2 \pi \ii} \oint_{\Gamma} f(z) G(z) \, \dd z\,,
\end{equation}
where $\Gamma$ is a contour encircling the spectrum of $H$.

A more general and powerful functional calculus is provided by the \emph{Helffer-Sj\"ostrand functional calculus} given in Appendix \ref{sec:HS}.
\item
The Green function is trivial to differentiate in $H$. For instance,
\begin{equation*}
\frac{\partial G_{ij}(z)}{\partial H_{kl}} \;=\; - G_{ik}(z) G_{lj}(z)\,,
\end{equation*}
and higher order derivatives have analogous simple expressions. This is in stark contrast to the eigenvalues and eigenvectors, whose derivatives in the entries of $H$, while explicit, are notoriously unwieldy in higher orders.

A related observation is that the perturbation theory of Green functions is trivial. Let $\wt H = H + \Delta$ be a perturbation of $H$, and write $\wt G(z) \deq (\wt H - z)^{-1}$. Then by iterating the resolvent identity $\wt G = G - G \Delta \wt G$ we get the \emph{resolvent expansion}
\begin{equation} \label{res_exp}
\wt G(z) \;=\; \sum_{k = 0}^{n - 1} G(z) (-\Delta G(z))^k + \wt G(z) (-\Delta G(z))^n\,.
\end{equation}
Analogous perturbation series for the eigenvalues and eigenvectors are thoroughly unpleasant.
\item
The Green function satisfies a host of very useful identities. One example is the resolvent expansion from \eqref{res_exp}. More sophisticated identities can be derived from Schur's complement formula; see Lemma \ref{lem8} below.
\item
The Green function is analytically stable. More precisely, it is better suited for doing estimates and analysing the eigenvalue density on small scales than other families of functions. Common alternative families are the moments $H^k$, $k \in \N$, and the unitary one-parameter group $\me^{\ii t H}$, $t \in \R$. Each of these families serves as a basis of a functional calculus. For example, we have the Fourier representation $f(H) = \int \dd t \, \hat f(t) \, \me^{-\ii t H}$, where $\hat f$ is the Fourier transform of $f$. The advantage of Green functions is that integrals of the form \eqref{cont_rep} are far less oscillatory than e.g.\ the Fourier representation. For many arguments, this means that reaching small spectral scales using Green functions is much easier than using moments or the unitary one-parameter group. In fact, in these notes we show how to obtain control on the optimal scales using Green functions. Control of analogous precision using moments or the unitary one-parameter group has until now not been achieved.
\end{enumerate}

We conclude this subsection with a cultural discussion on the term \emph{Green function} (sometimes also called \emph{Green's function}). Many mathematical problems can be formulated in terms of a linear operator $L \col f \to Lf$ acting on functions $f = f(x)$ defined on some set $\cal X$. A prominent example is the Dirichlet problem on an open domain $\cal X \subset \R^d$, where the linear operator $L = -\Delta$ is the Laplacian with Dirichlet boundary conditions. For such a problem, the \emph{Green function} $G(x,y)$ is a function of two variables, $x,y \in \cal X$, defined as the integral kernel of the resolvent $(L - z)^{-1}$ with respect to some natural measure on $\cal X$, often Lebesgue measure if $\cal X$ is an open subset of $\R^d$ or the counting measure if $\cal X$ is discrete. Here the spectral parameter $z$ is often taken to be $0$. We discuss a few examples.
\begin{enumerate}
\item
In the above example of the Dirichlet problem with $\cal X \subset \R^d$ open and $L = -\Delta$ with Dirichlet boundary conditions, the Green function $L^{-1}(x,y)$ is the integral kernel of $L^{-1}$ with respect to Lebesgue measure. It gives the solution of the problem $L f = \rho$ via the integral formula $f(x) = \int_{\cal X} \dd y \, G(x,y) \rho(y)$.
\item
For a discrete-time random walk on $\Z^d$, $L$ is the transition matrix of the walk, and the Green function $G(x,y) = (1 - L)^{-1}(x,y)$ is the integral kernel of $(1 - L)^{-1}$ with respect to the counting measure on $\Z^d$. The quantity $G(x,y)$ has the probabilistic interpretation of the expected time spent at $y$ starting from $x$.
\item
For a continuous-time $\R^d$-valued Markov process $(X_t)_{t \geq 0}$ with generator $L$, the Green function $G(x,y) = (-L)^{-1}(x,y)$ is the integral kernel of $(-L)^{-1}$ with respect to Lebesgue measure. The quantity $G(x,y) \, \dd y$ has the probabilistic interpretation of the expected time spent in $\dd y$ starting from $x$. Indeed, for a test function $f$ we find $\int \dd y \, G(x,y) f(y) = ((-L)^{-1} f)(x) = \int_0^\infty \dd t \,  (\me^{L t} f)(x) = \int_0^\infty \dd t \, \E^x [f(X_t)] = \E^x \qb{\int_0^\infty \dd t \, f(X_t)}$, where in the third step we used that the semigroup $S(t)$ defined by $S(t) f(x) \deq \E^x [f(X_t)]$ is given as $S(t) = \me^{L t}$ in terms of the generator $L$.
\item
A Green function can be associated with much more general differential operators than the Laplacian from (i), and can for instance also be used for parabolic and hyperbolic partial differential equations. For example, let $\cal X = \R \times \R^d$ and $L = \partial_0^2 - \partial_1^2 - \cdots - \partial_d^2$ be the d'Alambert operator. A Green function $G(x,y)$ gives a solution of the wave equation $L f = \rho$ using the integral formula $f(x) =  \int \dd y \, G(x,y) \rho(y)$. Note that, since the homogeneous wave equation $L f = 0$ has nontrivial solutions, $L$ is not invertible and the Green function is not unique. Nevertheless, Green functions can be easily constructed.
\item
Let $\cal X = \{1, \dots, d\}$, and $L$ be a positive matrix on $\R^d$. Define the Gaussian probability measure $\mu(\dd \xi) \deq \frac{1}{Z} \me^{-\frac{1}{2} \scalar{\xi}{L\xi}} \dd \xi$. Then the \emph{two-point correlation function} $G(x,y) \deq \int \mu(\dd \xi) \, \xi_x \, \xi_y$ is simply equal to the matrix entry $(L^{-1})(x,y)$, i.e.\ the integral kernel of $L^{-1}$ with respect to the counting measure on $\cal X$. Generally, in statistical mechanics correlation functions of free fields are often referred to as Green functions. This convention is natural for classical fields, where a free field is a Gaussian measure, but it is also commonly used in quantum field theory.
\item
Finally, in the context of random matrix theory we take $\cal X = \{1, \dots, N\}$, $L = H$, and write $G_{ij}$ instead of $G(x,y)$. The spectral parameter $z$ plays a central role and is therefore often kept in the notation. Moreover, we do not distinguish the resolvent $(H - z)^{-1}$ from its entries $G_{ij} = (H - z)^{-1}_{ij}$, and refer to both as the Green function.
\end{enumerate}

\subsection{Local and global laws} 
We define the \emph{empirical eigenvalue distribution}
\begin{equation} \label{def_mu}
\mu \;\deq\; \ff{N}\sum_{i=1}^N \del_{\lam_i}\,.
\end{equation}
For many matrix models $H$ one has $\mu \to \varrho$ as $N \to \infty$ (in some sense to be made precise), where $\varrho$ is a deterministic probability measure that does not depend on $N$. (For instance, for the Wigner matrices defined in Definition \ref{def:Wigner} below, $\varrho$ is the semicircle distribution \eqref{defscl}.) This convergence is best formulated in terms of \emph{Stieltjes transforms}.
 
We define the Stieltjes transforms
\begin{equation} \label{def_s}
s(z) \;\deq\; \int \frac{\mu(\ud x)}{x-z} \;=\; \ff{N}\Tr G(z)
\end{equation}
and
\begin{equation} \label{def_m}
m(z) \;\deq\; \int \frac{\varrho(\ud x)}{x-z}
\end{equation}
of the probability measures $\mu$ and $\varrho$.
Note that $s(z)$ is random while $m(z)$ is deterministic.

From now on, we use the notation
\begin{equation*}
z \;=\; E + \ii \eta
\end{equation*}
for the real and imaginary parts of $z$. We always assume that $z$ lies in the open upper-half plane $\C_+$, i.e.\ that $\eta > 0$.
 
Note that \be\la{6121414h44}\ff{\pi}\im s(z) \;=\; \ff{N}\sum_{i=1}^N \frac{\eta/\pi}{(\lam_i-E)^2+\eta^2} \;=\; (\mu *\theta_\eta)(E)\,,\ee
where we defined
\begin{equation} \label{def_theta}
\theta_\eta(x) \;\deq\; \frac{1}{\pi} \im \frac{1}{x - \ii \eta} \;=\; \frac{\eta/\pi}{x^2+\eta^2}\,.
\end{equation}
The function $\theta_\eta$ is an approximate delta function as $\eta \downarrow 0$, with breadth $\eta$. We conclude that control of $\pi^{-1} \im s(z)$ is tantamount to control of the empirical distribution $\mu$ smoothed out on the scale $\eta$. Hence,   $\eta$ is called the \emph{spectral resolution}.

Stieltjes transforms in particular provide a convenient way of studying convergence of (random) probability measures in distribution.
The following lemma is proved in Appendix \ref{sec:PGL}.
 \beg{lem} \label{thm:global_law}
Let $\mu \equiv \mu_N$ be a random probability measure, and $\varrho$ a deterministic probability measure. Let $s \equiv s_N$ and $m$ denote the Stieltjes transforms of $\mu$ and $\varrho$ respectively. Then
 \be\la{cvscl} \mu \;\todist\; \varrho \qquad \Longleftrightarrow \qquad  s(z) \;\lto\; m(z) \txt{ for all fixed $z\in \C_+$}\,,\ee
where both convergences are as $N\to\infty$ and with respect to convergence in probability in the randomness of $\mu$ or $s$, and $\todist$ denotes convergence in distribution of probability measures\footnote{Explicitly, ``$\mu \todist \varrho$ in probability'' means: for all continuous bounded functions $f : \R \to \R$ and all $\epsilon > 0$ we have $\P(\abs{\int f \,\dd (\mu -  \varrho)} > \epsilon) \to 0$ as $N \to \infty$.}.
 \en{lem}

 Note that in Lemma \ref{thm:global_law} the spectral parameter $z$ does not depend on $N$. Hence, recalling the discussion following \eqref{6121414h44}, we see that the convergence $s(z) \to m(z)$ has a spectral resolution of order one. A result of the form
\begin{equation*}
s(z) \;\lto\; m(z) \txt{ for all $z\in \C_+$}
\end{equation*}
(for instance in probability) is therefore called a \emph{global law}.

A \emph{local law} is a result that controls the error $s(z) - m(z)$ for all $z$ satisfying $\eta \gg N^{-1}$. In other words, a local law admits arguments $z$ that depend on $N$, so that the spectral resolution $\eta$ may be much smaller than the global scale $1$. We always require the spectral resolution to be bigger than $N^{-1}$. This restriction is necessary and its origin easy to understand. Since we assumed that $\norm{H} \asymp 1$ and $H$ has $N$ eigenvalues, the typical separation of eigenvalues is of order $N^{-1}$. Individual eigenvalues are expected to fluctuate about their mean locations, so that the random quantity $s(z)$ can only be expected to be close to the deterministic quantity $m(z)$ if some averaging mechanism ensures that $s(z)$ depends strongly on a large number of eigenvalues. This means that the spectral resolution $\eta$ has to be larger than the typical eigenvalue spacing $N^{-1}$. See Figure \ref{Fig:Spectral_Resolution} for an illustration of the spectral scale $\eta$ and the approximation from \eqref{6121414h44}.

\begin{figure}[!ht]
\centering
\subfigure[Spectral resolution $\eta=N^{-0.8}$.]{
\includegraphics[scale=.35]{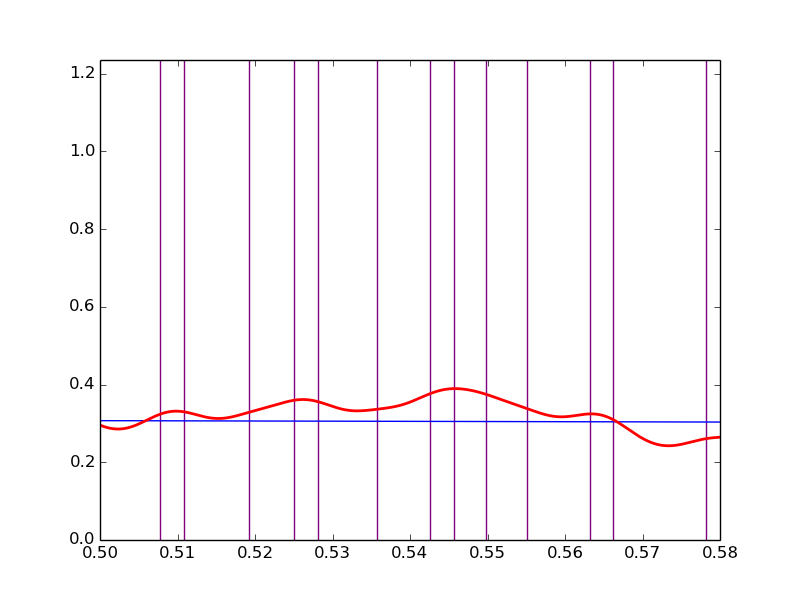}} \qquad 
\subfigure[Spectral resolution $\eta=N^{-1.2}$.]
{\includegraphics[scale=.35]{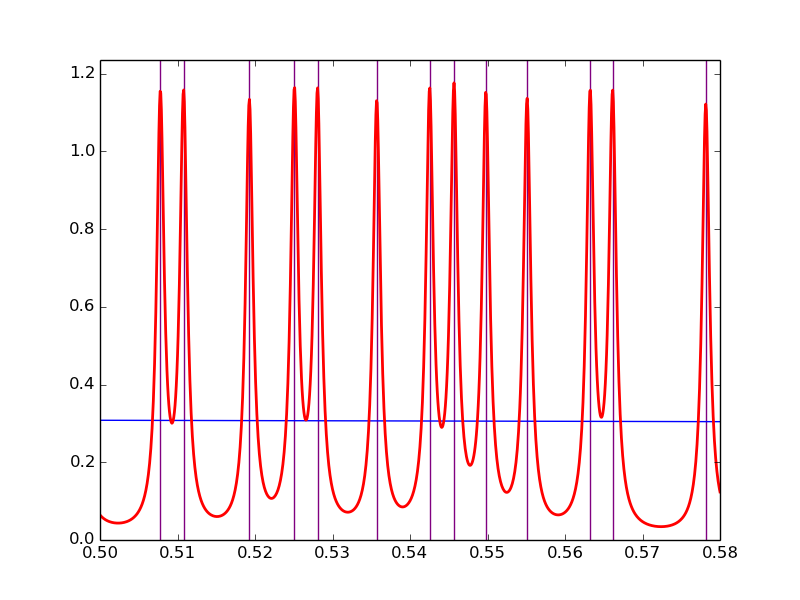}}
\caption{\emph{Role of the spectral resolution $\eta$.} In both plots we draw the eigenvalues of a $500\times 500$ GUE matrix (vertical lines), $\im s(z)/\pi$ as a function of $E$ for a fixed $\eta$  (red curve), and   $\im m(z)/\pi$  as a function of $E$  for the same fixed $\eta$  (blue curve). Both plots are drawn using the same realization of the underlying matrix and the same $x$- and $y$-ranges.}\la{Fig:Spectral_Resolution}
\end{figure}

For most applications to the distribution of the eigenvalues and eigenvectors, control of just $s(z) = N^{-1} \Tr G(z)$ is not enough, and one needs control of the Green function $G(z)$ regarded as a matrix. The need for such a stronger control is obvious if we are interested in the eigenvectors of $H$ (recall the spectral decomposition from \eqref{spec_decomp}). However, perhaps surprisingly, to understand the distribution of individual eigenvalues we also need to control the individual entries of $G$; see Section \ref{sec:comparison} for more details.

We shall see that, at least for Wigner matrices, the matrix $G(z)$ is close (in some sense to be made precise) to $m(z)$, a multiple of the identity matrix.

\subsection{The local law for Wigner matrices}

From now we focus on the case where $H$ is a Wigner matrix.

\beg{Def}[Wigner matrix] \label{def:Wigner}
A \emph{Wigner matrix}, or \emph{Wigner ensemble}, is a Hermitian $N\times N$ matrix $H = H^*$ whose entries $H_{ij}$ satisfy the following conditions.
\begin{enumerate}
\item
The upper-triangular entries $(H_{ij} \col 1\le i\le j\le N)$ are independent.
\item
For all $i,j$ we have $\E H_{ij}=0$ and $\E |H_{ij}|^2=N^{-1} (1 + O(\delta_{ij}))$.
\item
The random variables $\sqrt{N}H_{ij}$ are  bounded in any $L^p$ space, uniformly in $N,i,j$.
\end{enumerate}
\en{Def}

The conditions (i) and (ii) are standard. The factor $1 + O(\delta_{ij})$ in (ii) admits more general variances for the diagonal entries of $H$; as we shall see, many properties of the model are insensitive to the variances of the diagonal entries. Note that we do not assume the entries of $H$ to be identically distributed, and impose no constraint on $\E (H_{ij})^2$ associated with some symmetry class.

The condition (iii) is technical and made for convenience. Explicitly, it states that for each $p \in \N$ there exists a constant $C_p$ such that $\norm{\sqrt{N} H_{ij}}_p \leq C_p$ for all $N,i,j$, where we use the notation
\begin{equation*}
\norm{X}_p \;\deq\; \pb{\E \abs{X}^p}^{1/p}\,.
\end{equation*}
In the literature, various other conditions on the tails of the entries $\sqrt{N} H_{ij}$ have been used, ranging from sub-Gaussian tails to merely the existence of the second moment in (ii). The assumption (iii) can be relaxed (for instance by truncation), but we shall not pursue this direction in these notes.

This choice of the normalization $N^{-1}$ in (ii) ensures that $\norm{H} \asymp 1$ as required above. A simple heuristic way to convince ourselves that this is the right normalization is to compute the average square distance of an eigenvalue from the origin:
\begin{equation*}
\E \frac{1}{N} \sum_{i = 1}^N \lambda_i^2 \;=\; \frac{1}{N} \E \tr H^2 \;=\; \frac{1}{N} \sum_{i,j = 1}^N \E \abs{H_{ij}}^2 \;=\; 1 + O(N^{-1})\,,
\end{equation*}
as desired.

\begin{ex}[Gaussian ensembles] \label{def:Gaussian_W}
Let $X$ be an $N \times N$ matrix with i.i.d.\ real standard normal entries. Then the \emph{Gaussian orthogonal ensemble} (GOE) is defined as
\begin{equation*}
\frac{X + X^*}{\sqrt{2 N}}\,.
\end{equation*}
Similarly, let $Y$ be an $N \times N$ matrix with i.i.d.\ complex standard normal entries. (This means that $\real Y_{ij}$ and $\im Y_{ij}$ are independent standard normals.) Then the \emph{Gaussian unitary ensemble} (GUE) is defined as
\begin{equation*}
\frac{Y + Y^*}{\sqrt{4 N}}\,.
\end{equation*}
The entries of the GOE and the GUE are centred Gaussian random variables. Moreover, it is easy to check that if $H$ is the GOE then $N \E \abs{H_{ij}}^2 = 1 + \delta_{ij}$, and if $H$ is the GUE then $N \E \abs{H_{ij}}^2 = 1$. The terminology orthogonal/unitary stems from the fact that the GOE is invariant under conjugation $H \mapsto O H O^*$ by an arbitrary deterministic orthogonal matrix $O$, and the GUE is invariant under conjugation $H \mapsto U H U^*$ by an arbitrary deterministic unitary matrix $U$.
\end{ex}

Next, we define the \emph{semicircle distribution}
\be\la{defscl}\varrho(\ud x) \;\deq\; \ff{2\pi}\sqrt{(4-x^2)_+} \, \ud x \,.\ee
It is not hard to check (see Lemma \ref{lem:stieltjessclformula})
that the Stieltjes transform \eqref{def_m} of the semicircle distribution $\varrho$ is
\begin{equation} \label{m_qsolution}
m(z) \;=\; \int \frac{\varrho(\ud x)}{x-z} \;=\; \frac{-z+\sqrt{z^2-4}}{2}\,,
\end{equation}
where the square root $\sqrt{z^4 - 4}$ is chosen with a branch cut in the segment $[-2,2]$ so that $\sqrt{z^2 - 4} \sim z$ as $z \to \infty$. This ensures that $m(z) \in \C_+$ for $z \in \C_+$.

Clearly, we have
\begin{equation} \label{m_id}
m(z)+\ff{m(z)}+z\;=\;0\,.
\end{equation}
In fact, it is easy to show, for $z \in \C_+$, that  $m(z)$ is characterized as the unique solution of \eqref{m_id} in the upper half-plane $\C_+$. This latter characterization is often more convenient than the form \eqref{m_qsolution}.

The following global law, illustrated by Figure \ref{Fig:Global_Law},  is well known (see e.g.\ \cite{bai-silver-book,agz}). (It will also follow from the local law, Theorem \ref{Th3} below.)
\beg{Th}[Global law]\la{Th:Global_Law}
Let $H$ be a Wigner matrix and $\varrho$ defined by \eqre{defscl}. Then for all fixed $z \in \C_+$ we have $s(z) \to m(z)$ in probability as $N\to\infty$.
\en{Th}

\begin{figure}[!ht]
\centering
\subfigure[Histogram of the spectrum (in red) and density of the semicircle distribution (in blue).]{
\includegraphics[scale=.35]{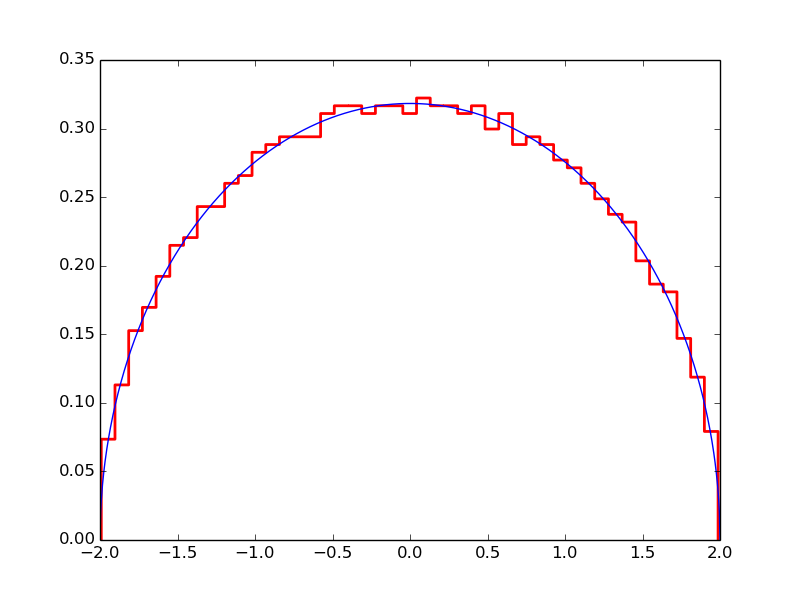}} \qquad 
\subfigure[As functions of $E$, for fixed spectral resolution $\eta=0.02$: $\im s(z)/\pi$ (in red) and   $\im m(z)/\pi$ (in blue).]
{\includegraphics[scale=.35]{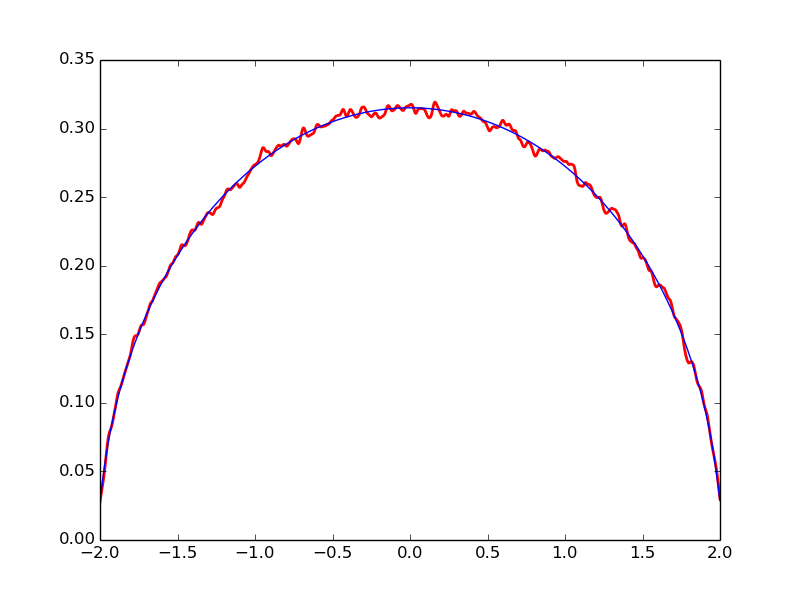}}
\caption{\emph{Global semicircle law.} Both plots, made with the same realization of a GUE matrix with $N=2 \cdot 10^3$, give slightly different (yet equivalent, by Lemma \ref{thm:global_law}) illustrations of the global law from Theorem \ref{Th:Global_Law}. As observed in \eqre{6121414h44}, $\im m(z)/\pi$ is close to the density of the semicircle distribution when $\eta$ is close (but not too close; see Figure \ref{Fig:Spectral_Resolution}) to zero.}\la{Fig:Global_Law}
\end{figure}

In order to state the local law, we use the following notion of high-probability bounds that was introduced in \cite{EKYfluc}. It provides a simple way of systematizing and making precise statements of the form ``$X$ is bounded with high probability by $Y$ up to small powers of $N$''. As we shall see, it is a very convenient way of wrapping various details of convergence in high probability, such as the low-probability exceptional events, into an object that very rarely needs to be unwrapped.

\begin{Def}[Stochastic domination] \label{def:stochdom}
Let
\begin{equation*}
X \;=\; \pb{X^{(N)}(u) \col N \in \N, u \in U^{(N)}} \,, \qquad
Y \;=\; \pb{Y^{(N)}(u) \col N \in \N, u \in U^{(N)}}
\end{equation*}
be two families of nonnegative random variables, where $U^{(N)}$ is a possibly $N$-dependent parameter set. 

We say that $X$ is \emph{stochastically dominated by $Y$, uniformly in $u$,} if for all (small) $\epsilon > 0$ and (large) $D > 0$ we have
\begin{equation*}
\sup_{u \in U^{(N)}} \P \qB{X^{(N)}(u) > N^\epsilon Y^{(N)}(u)} \;\leq\; N^{-D}
\end{equation*}
for large enough $N\geq N_0(\epsilon, D)$.
The stochastic domination is always uniform in all parameters (such as matrix indices and spectral parameters $z$) that are not explicitly fixed.

If $X$ is stochastically dominated by $Y$, uniformly in $u$, we use the notation $X \prec Y$. Moreover, if for some complex family $X$ we have $\abs{X} \prec Y$ we also write $X = O_\prec(Y)$.

Finally, we say that an event $\Xi=\Xi^{(N)}$ holds \emph{with high probability} if $1 - \ind{\Xi} \prec 0$, i.e.\ if for any $D>0$ there is $N_0(D)$ \st for all $N\ge N_0(D)$ we have $\P(\Xi^{(N)}) \ge 1-N^{-D}$.
\end{Def}

For example, it is easy to check that  $|H_{ij}|\prec N^{-1/2}$. Note that this notation implicitly means that $\prec$ is uniform in the indices $i,j$.

We may now state the main result of these notes.

\beg{Th}[Local law for Wigner matrices]\la{Th3}
Let $H$ be a Wigner matrix. Fix $\tau>0$ and define the domain 
\begin{equation} \label{def_S}
\f S \;\equiv\; \f S_N(\tau) \;\deq\; \hb{E+\mathrm{i}\eta\col |E|\le \tau^{-1}\,,\, N^{-1+\tau}\le \eta\le \tau^{-1}}\,.
\end{equation}
Then we have
\be\label{121414h53}
s(z) \;=\; m(z)+O_\prec \pbb{\ff{N\eta}}\ee
and
\be\label{121414h54}G_{ij}(z) \;=\; m(z) \del_{ij}+O_\prec (\Psi(z))\ee
uniformly for $i,j = 1, \dots, N$ and $z\in \f S$, where we defined the deterministic error parameter
\begin{equation} \label{def_Psi}
\Psi(z) \;\deq\; \sqrt{\frac{\im m(z)}{N\eta}}+\ff{N\eta}\,.
\end{equation}
\en{Th}

For $\abs{E} \leq 2$, both estimates \eqref{121414h53} and \eqref{121414h54}, as encapsulated by the respective error parameters $(N \eta)^{-1}$ and $\Psi$, are optimal up to the details in the definition of $\prec$. For $\abs{E} > 2$ the optimal bounds are better than those of Theorem \ref{Th3}.   Remarkably, these optimal bounds for $\abs{E} > 2$ are essentially a consequence of the weaker ones from Theorem \ref{Th3}; see Theorems \ref{thm:ext1} and \ref{thm:ext2} below for the precise statements.  

To help the interpretation of the error parameter $\Psi$, we consider the cases where $E$ is in the \emph{bulk} $(-2,2)$ of the spectrum and at the \emph{edge} $\{-2,2\}$ of the spectrum. Suppose first that $E \in (-2,2)$ is fixed in the bulk. Then we easily find $\im m(z) \asymp 1$ for $\eta \in [N^{-1},1]$. Hence, the first term of \eqref{def_Psi} dominates and we have $\Psi(z) \asymp (N \eta)^{-1/2}$, which is much smaller than $\im m(z)$ for $\eta \gg N^{-1}$.   Note that the scale $N^{-1}$ is the typical separation of the eigenvalues in the bulk, as can be seen for instance by choosing $i \in [cN, (1 -c)N]$ in \eqref{asymp_gamma} below, for some constant $c > 0$. 

On the other hand, if $E = 2$ is at the edge, we find $\im m(z) \asymp \sqrt{\eta}$. We conclude that the first term of \eqref{def_Psi} dominates over the second if $\eta \gg N^{-2/3}$ and the second over the first if $\eta \ll N^{-2/3}$.   Note that the threshold $N^{-2/3}$ is precisely the typical separation of the eigenvalues near the edge, as can be seen for instance by choosing $i  \leq C$ in \eqref{asymp_gamma} below, for some constant $C > 0$.   Hence, we conclude that at the edge $\Psi(z)$ is much smaller than $\im m(z)$ provided that $\eta \gg N^{-2/3}$. See Figure \ref{Fig:Local_Law}.
\begin{figure}[!ht]
\centering
\subfigure[$\im s(z)$ (in red),   $\im m(z)$ (in blue) and the error bounds $\im m(z)\pm 1/(N\eta)$ (green dashed lines).]{
\includegraphics[scale=.35]{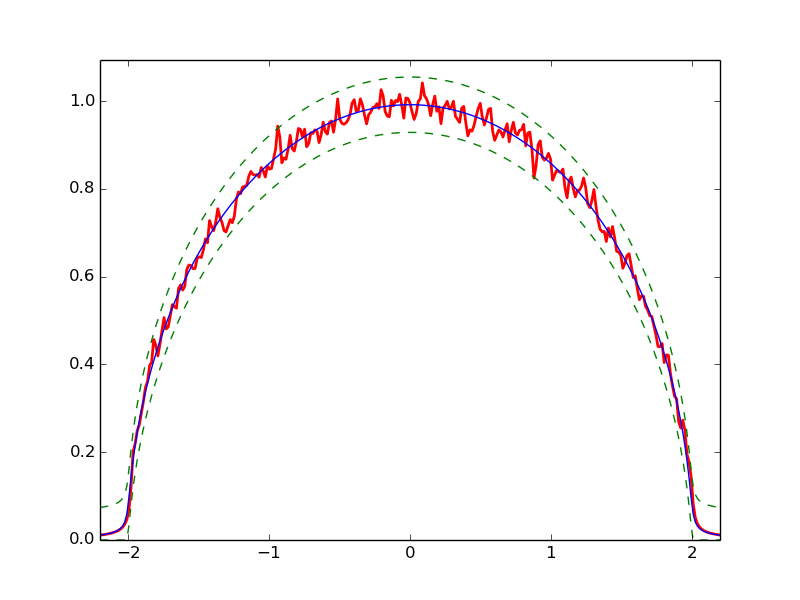}} \qquad 
\subfigure[$\im G_{11}(z)$ (in red),   $\im m(z)$ (in blue) and the error bounds $\im m(z)\pm \Psi(z)$ (green dashed lines).]
{\includegraphics[scale=.35]{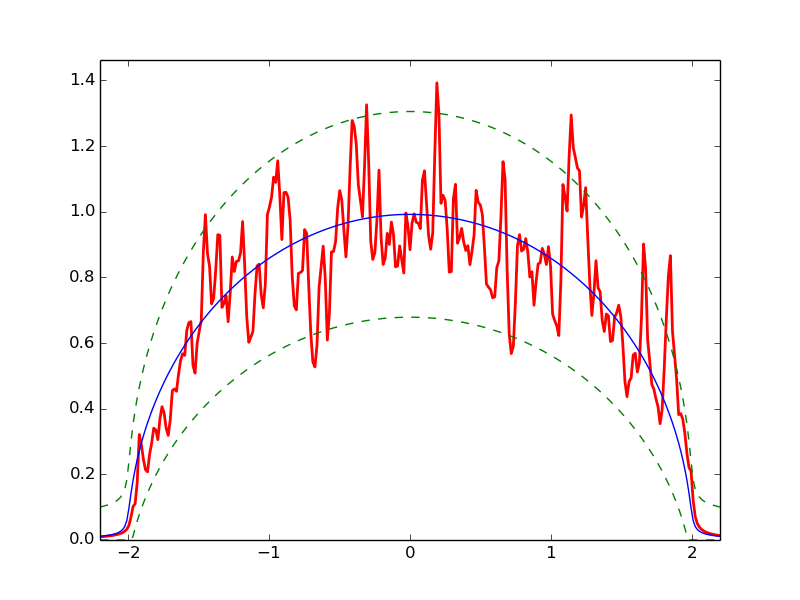}}
\\
\subfigure[$\im G_{12}(z)$ (in red)    and the error bounds $\pm \Psi(z)$ (green dashed lines).]
{\includegraphics[scale=.35]{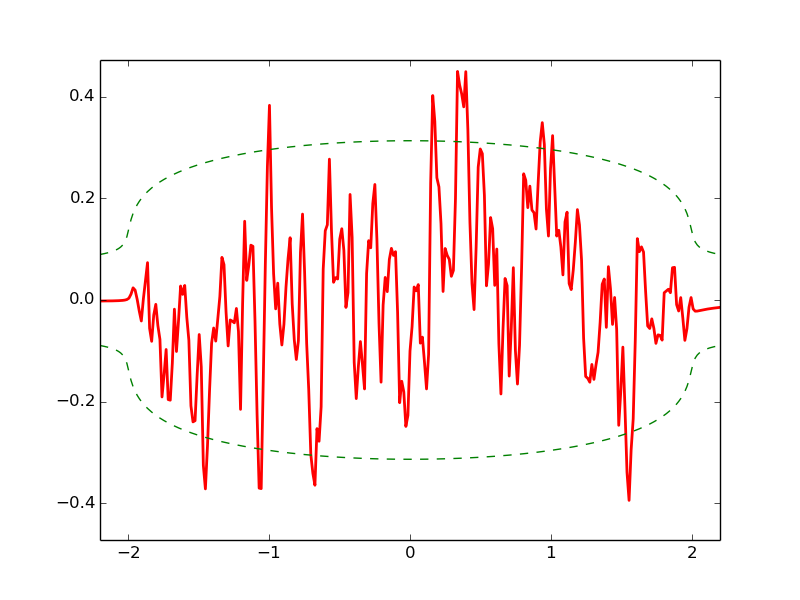}}
\caption{\emph{Local semicircle law.} All plots were obtained from the same realization of a GUE matrix with $N=10^3$. We draw several functions of $E$ at a fixed spectral resolution $\eta=N^{-0.6}$. As expected from Theorem \ref{Th3}, the fluctuations and the error bounds are larger for the matrix entries  than for the normalized trace of the Green function.}\la{Fig:Local_Law}
\end{figure}

\beg{rmk} \label{rem:Lipschitz}
The estimate \eqref{121414h54} holds uniformly in $i,j = 1, \dots, N$ and $z\in \f S$. In general, such an estimate does of course not imply a simultaneous high-probability bound for all $z \in \f S$. However, thanks to the Lipschitz-continuity of all quantities in \eqref{121414h54}, this implication does hold. More precisely, we have
\begin{equation} \label{simult_bound}
\sup_{z \in \f S} \max_{i,j} \frac{\abs{G_{ij}(z) - m(z) \delta_{ij}}}{\Psi(z)} \;\prec\; 1\,.
\end{equation}
To obtain \eqref{simult_bound} from \eqref{121414h54}, we define the $N^{-3}$-net $\wh {\f S} \deq (N^{-3} \Z^2) \cap \f S$. Hence, $\abs{\wh{\f S}} \leq C N^6$ and for any $z \in \f S$ there exists a $w \in \wh {\f S}$ such that $\abs{s - w} \leq 2 N^{-3}$. By a simple union bound, we find from \eqref{121414h54} that \eqref{simult_bound} holds with $\f S$ replaced with $\wh {\f S}$. Using that $G_{ij}(z)$, $m(z)$, and $\Psi(z)$ are all $N^2$-Lipschitz on $\f S$ (as follows immediately from their definitions) and that $\Psi(z) \geq 1/N$ on $\f S$, we easily deduce \eqref{simult_bound}.
\en{rmk}

\subsection{Applications of Theorem \ref{Th3}}
We now state some important consequences of Theorem \ref{Th3}. The first one gives the semicircle law on \emph{small scales}.
In Lemma \ref{thm:global_law} we saw that control of $s(z) - m(z)$ for fixed $z$ yields control of $\mu - \varrho$ on large scales. Using the strong local control on $s(z) - m(z)$ from Theorem \ref{Th3}, we may correspondingly obtain strong bounds on the local convergence of $\mu$ to $\varrho$.
\begin{Th}[Semicircle law on small scales]\la{thm:llsc}
For a Wigner matrix we have
\begin{equation*}
\mu(I) \;=\; \varrho(I) + O_\prec(N^{-1})
\end{equation*}
uniformly for all intervals $I \subset \R$.
\end{Th}
\begin{figure}[ht]
\centering 
\subfigure[$N=50$, GUE matrix]{\includegraphics[scale=.4]{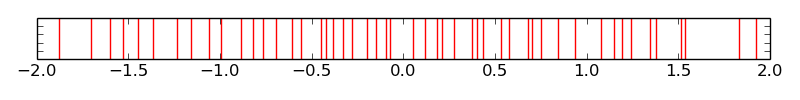}}
\subfigure[$N=10^3$, GUE matrix (zoomed in)]{\includegraphics[scale=.4]{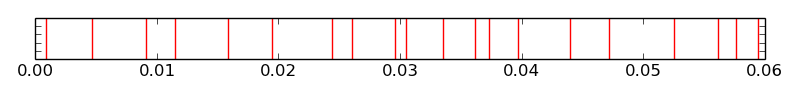}}
\subfigure[$N=50$, i.i.d.\ variables with semicircle distribution]{\includegraphics[scale=.4]{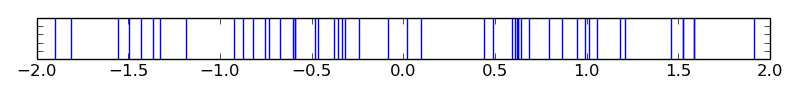}}
\subfigure[$N=10^3$, i.i.d.\ variables with semicircle distribution (zoomed in)]{\includegraphics[scale=.4]{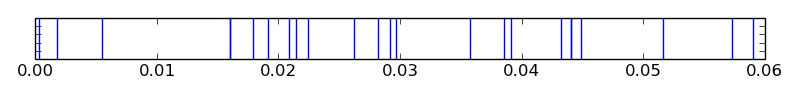}}
\caption{\emph{Local eigenvalue density.}  We compare the eigenvalues of an $N\ti N$ GUE matrix (abscissas of the red vertical lines) with a sample of $N$ i.i.d.\ variables with semicircle distribution (abscissas of the blue vertical lines). The GUE eigenvalue spacings are far more regular than those of the i.i.d.\ variables; the former are typically of order $N^{-1}$ (by Theorem \ref{thm:llsc} or Theorem \ref{thm:rig}), while the latter are typically of order $N^{-1/2}$.
}
\la{Fig:llsc}
\end{figure}

We postpone the proof of Theorem \ref{thm:llsc} to Section \ref{sec:local_law_small_scales}.
The second consequence of Theorem \ref{Th3} is an \emph{eigenvalue rigidity} estimate, which gives large deviation bounds on the locations of the eigenvalues. Eigenvalue rigidity for Wigner matrices was first established in \cite{EYYrigi}, using the optimal error bound from \eqref{121414h53} that was first obtained there.

For $i = 1, \dots, N$ we define the \emph{typical location of $\lam_i$} as the quantile $\ga_i$ satisfying
\begin{equation} \label{def_gamma}
N \int_{\ga_i}^{2}\varrho(\ud x) \;=\; i - 1/2\,.
\end{equation}
It is easy to see that the typical eigenvalue locations $\gamma_i$ satisfy
\begin{equation} \label{asymp_gamma}
2 - \gamma_i \;\asymp\; \pbb{\frac{i}{N}}^{2/3} \qquad (i \leq N/2)\,,
\end{equation}
and by symmetry a similar estimate holds for $i \geq N/2$. Note that the right-hand side of \eqref{asymp_gamma} is characteristic of the \emph{square root decay} of the density of $\varrho$ near the boundary of its support.
  
\beg{Th}[Eigenvalue rigidity] \label{thm:rig}
For a Wigner matrix we have
$$|\lam_i-\ga_i| \;\prec\; N^{-2/3}\pb{i\wedge (N+1-i)}^{-1/3}$$
uniformly for $i = 1, \dots, N$.
\en{Th}
In particular, the extreme eigenvalues $\lambda_1$ and $\lambda_N$ are with high probability located at a distance of order at most $N^{-2/3}$ from their typical locations. Moreover, the bulk eigenvalues, $\lambda_i$ satisfying $\abs{i} \asymp \abs{N - i} \asymp N$, are with high probability located at a distance of order at most $N^{-1}$ from their typical locations. The eigenvalue rigidity is a manifestation of a strong repulsion between the eigenvalues, which tend to form a rather rigid ``jelly'' where neighbouring eigenvalues avoid getting too close and are therefore pinned down by the influence of their neighbours. In contrast, if $\lambda_1, \dots, \lambda_N$ were i.i.d.\ random variables distributed according to the semicircle distribution, the global semicircle law for $s - m$ would remain true, but a standard exercise in order statistics shows that in this case $\lambda_i$ would typically fluctuate on the scale $N^{-1/2}$. This is much larger than the scale $N^{-1}$ from Theorem \ref{thm:rig}. Correspondingly, for this i.i.d.\ model we would expect to find gaps of order $N^{-1/2}$, and the error term in Theorem \ref{thm:llsc} would have to be replaced with $O_\prec(N^{-1/2})$. See Figures \ref{Fig:llsc} and \ref{Fig:Rigidity} for an illustration of this rigidity.
\begin{figure}[!ht]
\centering
\subfigure[$N=50$]{\includegraphics[scale=.35]{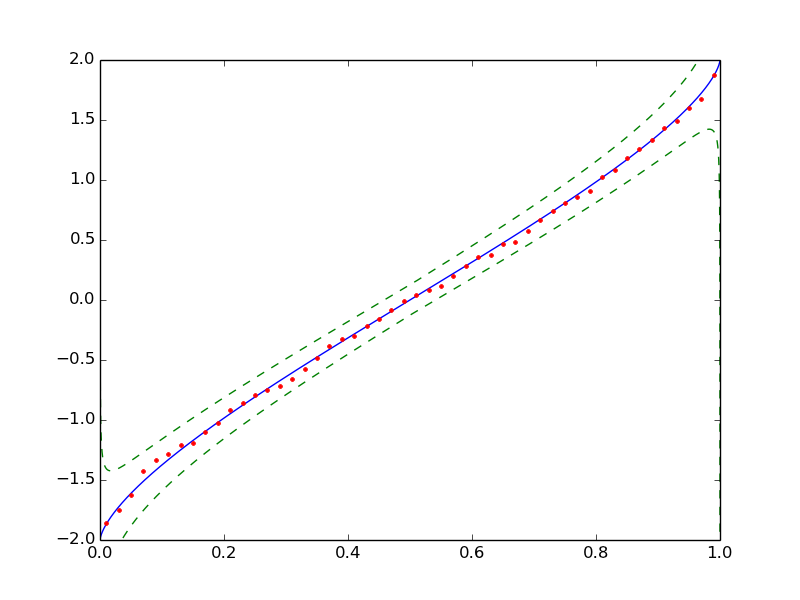}}\subfigure[$N=2.10^3$ (zoomed in)]{\includegraphics[scale=.35]{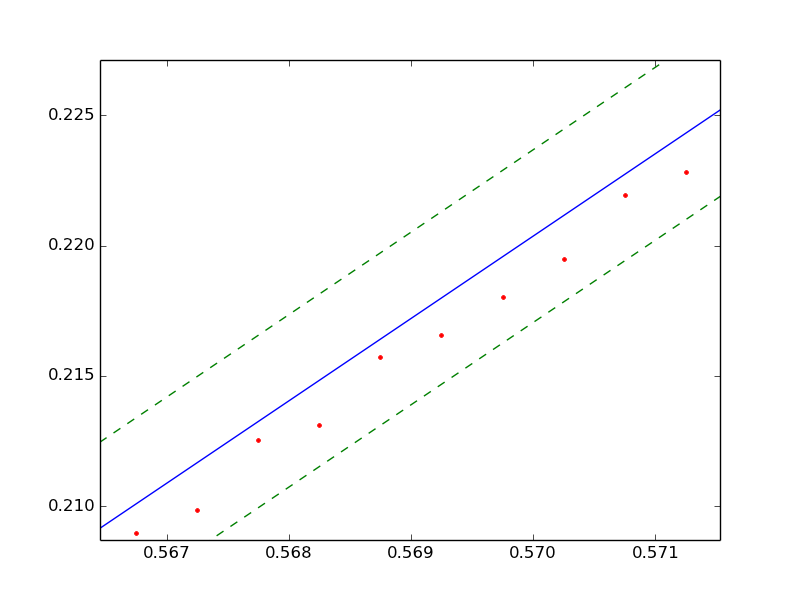}}
\caption{\emph{Eigenvalue rigidity.} The quantile function $x \in [0,1] \mapsto Q(x)$ of the semicircle distribution (in blue), the error bounds $x\in [0,1]\mapsto Q(x)\pm 5N^{-1}(x\wedge  (1-x))^{-1/3}$ (green dashed lines), and a scatter plot of the function $(i-1/2) /N \mapsto \lambda_i$ where $\lambda_i$ is the $i$-th eigenvalue of  an $N\ti N$ GUE matrix (red points).}\la{Fig:Rigidity}
\end{figure}

We postpone the proof of Theorem \ref{thm:rig} to Section \ref{sec:rig}.
A very easy corollary of Theorems \ref{Th3} and \ref{thm:rig} is the \emph{complete delocalization of eigenvectors}, illustrated in Figure \ref{Fig:Eigenvectors}. The complete delocalization of eigenvectors of Wigner matrices was first obtained in \cite{ESY2} as a corollary of the local law on optimal scales proved there.

We use the notation $\f u_i = (u_i(k))_{k = 1}^N$ for the components of the eigenvectors.  

\beg{Th}[Complete eigenvector delocalization] \label{thm:deloc}
For a Wigner matrix we have $$|u_i(k)|^2 \;\prec\; \ff{N}\,.$$
uniformly for $i,k = 1, \dots, N$.
\en{Th}

\bpr
By Theorem \ref{thm:rig} we know that $\abs{\lambda_i} \leq 3$ with high probability. We choose a random spectral parameter by choosing $z\deq \lam_i+\mathrm{i} \eta$ with $\eta \deq N^{-1 + \tau}$. Hence for small enough $\tau > 0$ we have $z \in \f S$ with high probability. Moreover, by Remark \ref{rem:Lipschitz} we may invoke Theorem \ref{Th3} with the random error parameter $z$. Since $|m(z)|\le 1$ (as can be easily checked using \eqref{m_qsolution}) and $\Psi(z)\le C$, we get from \eqre{121414h54} that
$$
1 \;\succ\; \im G_{kk}(z) \;=\; \sum_{j=1}^N \frac{\eta}{(\lam_j-\lam_i)^2+\eta^2}|u_j(k)|^2 \;\ge\; \ff{\eta}|u_i(k)|^2 \;=\; N^{1-\tau}|u_i(k)|^2\,.
$$
Since $\tau>0$ was arbitrary, the conclusion follows.
\epr

\begin{figure}[!ht]
\centering
\includegraphics[width=15cm,height=7cm]{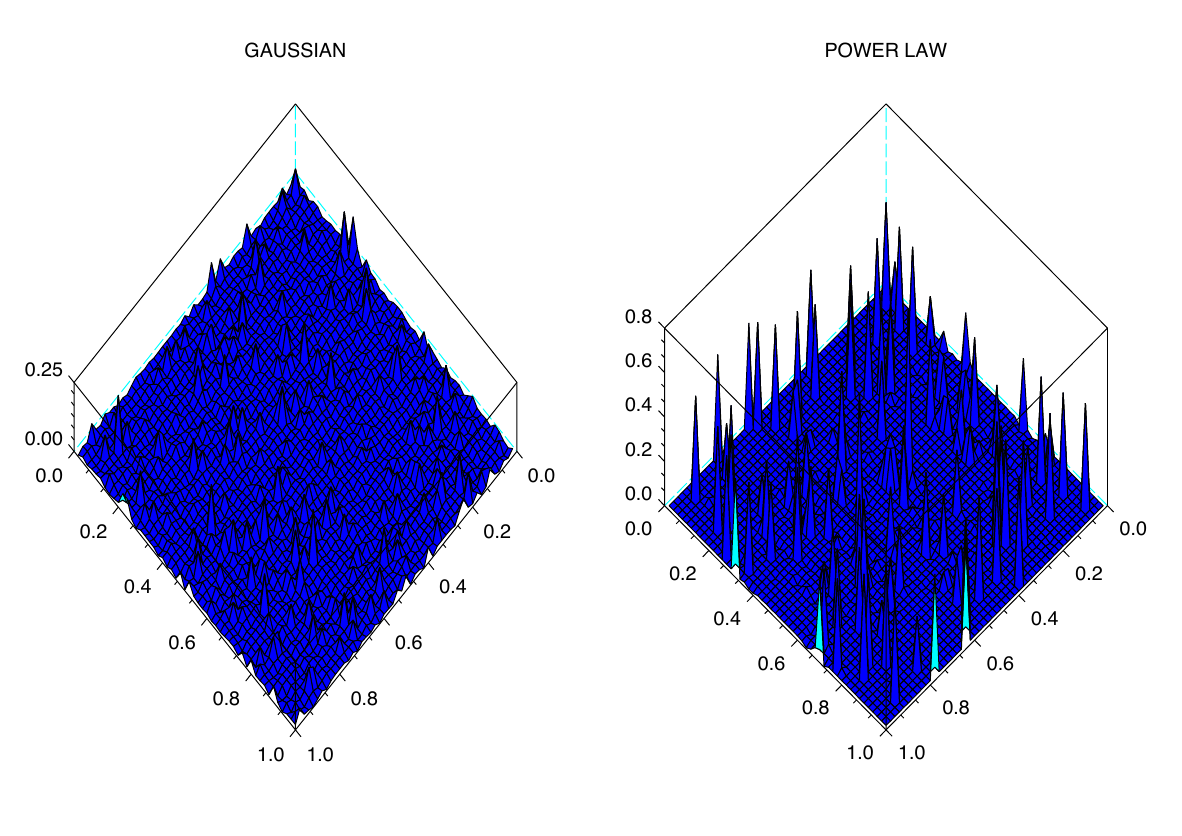}
\caption{\emph{Complete eigenvector delocalization.}  We illustrate two plots of the function $(i/N,k/N)\mapsto ||u_i(k)|^2-N^{-1}|$, where $u_i(k)$ denotes the $k$-th component of the $i$-th eigenvector of a Wigner matrix, and $N = 50$. The Wigner matrix of the left-hand plot has Gaussian entries (GOE), while that of the right-hand plot has heavy-tailed entries of the form $S U^{-(2+\eps)}$ for $S,U$ independent \st $S=\pm 1$ is symmetric, $U$ is uniformly distributed  on $[0,1]$ and $\eps=10^{-3}$.
}\la{Fig:Eigenvectors}
\end{figure}

To understand the meaning of Theorem \ref{thm:deloc}, observe that for each fixed $i$ the map $k \mapsto p_i(k) \deq \abs{u_i(k)}^2$ is a (random) probability measure on $\{1, \dots, N\}$. A fundamental question about eigenvectors is whether $p_i$ is mainly concentrated at a few points or whether it is evenly spread out over $\{1, \dots, N\}$. Complete eigenvector delocalization states that, for each eigenvector $\f u_i$ and with high probability, the latter possibility always holds. Note that if $G$ is a GOE/GUE matrix, the law of $H$ is invariant under orthogonal/unitary conjugations, and its eigenvector matrix therefore uniformly distributed on the orthogonal/unitary group. In particular, each eigenvector $\f u_i$ is uniformly distributed on the unit sphere, and we easily find that $\abs{u_i(k)}^2 \leq (\log N)^a N^{-1}$ for all $i,k$ with high probability, provided that $a > 1/2$. Up to the technical details in the definition of $\prec$ (see Remark \ref{rem:logN} below), this simple argument therefore recovers, for the special cases of GOE and GUE, the result of Theorem \ref{thm:deloc} (which holds of course for all Wigner matrices).

\begin{rmk} \label{rem:logN}
The notion of stochastic domination from Definition \ref{def:stochdom} greatly simplifies many statements and arguments. On the other hand, the specific notion of high-probability bounds that it yields is not as sharp as possible: the factors of $N^\epsilon$ can be improved to powers of $\log N$ and the polynomial error probabilities $N^{-D}$ can be improved to exponential error probabilities. Such extensions may in fact be obtained by a routine but tedious extension of the arguments presented in these notes. We do not pursue this further.
\end{rmk}

The local law has many further applications, for instance to the universality of the local eigenvalue statistics and the distribution of eigenvectors; see Section \ref{sec:comparison} below for more details.

\subsection{Sketch of the proof of Theorem \ref{Th3}}
We conclude this section with a sketch of the proof of Theorem \ref{Th3}, which is the main content of these notes. A more detailed presentation of the main steps of the proof is given in Section \ref{Sec:local_law_proof_abstract}.

\bgt\ite[(a)] The starting point of the proof is \emph{Schur's complement formula}
\be\la{SPSFintroShort} \ff{G_{ii}} \;=\; H_{ii}-z-\sum_{k,l \neq i} H_{ik}G^{(i)}_{kl}H_{l i}\,,\ee
where $G^{(i)}$ denotes the Green function of the matrix obtained from $H$ by removing the $i$-th row and column. The coefficients $G_{kl}^{(i)}$ are independent of the family $(H_{ik})_{k = 1}^N$, so that we may apply large deviation estimates to the sum, finding that it is in fact close with high probability to its expectation with respect to the randomness in the $i$-th column, which is in turn close to $\frac{1}{N} \sum_k G_{kk} = s$.
\ite[(b)] We split the right-hand side of \eqref{SPSFintroShort} into a leading term and a random error term, which yields after multiplication of both sides by $G_{ii}$
\be\la{3121417h38Short}1 \;=\; -zG_{ii}-sG_{ii}+(\trm{error term})\,.\ee
Taking $\eta \geq 1$, averaging over $i$, and estimating the random error term, we get \be\la{3121417h382Short}1+s^2+sz \;=\; O_\prec(N^{-1/2})\ee
for $\eta\ge 1$.
\ite[(c)] As $m(z)$ is a solution of \eqre{3121417h382Short} with the right-hand side set to zero, a stability analysis of \eqref{3121417h382Short} yields an estimate of $s(z)-m(z)$ for $\eta\ge 1$. This is the global semicircle law. Going back to \eqre{SPSFintroShort} and using further analogous formulas for the off-diagonal entries, we obtain weak control of the entries $G_{ij} - m \delta_{ij}$ on the global scale $\eta\ge 1$.
\ite[(d)] The heart of the proof of the local law is a \emph{boostrapping} from large scales to smaller scales. Roughly, we prove the local semicircle law on a certain scale $\eta$, and then use the Lipschitz continuity of $G_{ij} - m \delta_{ij}$ to deduce an a priori bound on the slightly smaller scale $\eta' = \eta - N^{-3}$. Using this a priori bound as input, we use the formula \eqref{SPSFintroShort} and similar formulas for the off-diagonal entries to obtain an improved bound $G_{ij} - m \delta_{ij}$ on the same scale $\eta'$. Thanks to this self-improving estimation, we may iterate this procedure an order $N^3$ times to reach the optimal scale $\eta \approx N^{-1}$. At each step of this iteration, we lose some probability. This loss has to be carefully tracked to ensure that the combined losses in probability from $N^3$ iterations remain much smaller than $1$. This procedure yields control of $G_{ij} - m \delta_{ij}$ for all $z \in \f S$, albeit with an error bound that is larger than $\Psi$ from \eqref{def_Psi}.

\ite[(e)]
In order to obtain \eqref{121414h54} with the optimal error bound $\Psi$, we use a \emph{fluctuation averaging argument}. Roughly, it says that the random variables $1/G_{ii}$, when centred with respect to the randomness of the $i$-th column of $H$, are not too strongly dependent, and hence their average is typically much smaller than any one of them.\ent

\section{Proof of Theorem \ref{Th3} (a): preliminaries} \label{sec:prelim}
 
We now move on to the proof of Theorem \ref{Th3}. For the proof we often omit the spectral parameter $z \in \f S$ from our notation, bearing in mind that we are typically dealing with random sequences (indexed by $N \in \N$) of functions (of $z \in \f S$).

In this preliminary section we collect the main tools of the proof.
\begin{Def}[Minors] \label{def: minors}
We consider general matrices whose indices lie in subsets of $\{1, \dots, N\}$. For $T \subset \{1, \dots, N\}$ we define $H^{(T)}$ as the $(N - \abs{T}) \times (N - \abs{T})$ matrix
\begin{equation*}
H^{(T)} \;=\; (H_{ij})_{i,j \in \{1, \dots, N\} \setminus T}\,.
\end{equation*}
Moreover, we define the Green function of $H^{(T)}$ through
\begin{equation*}
G^{(T)}(z) \;\deq\;  \pb{H^{(T)} - z}^{-1}\,.
\end{equation*}
We also set
\begin{equation*}
\sum_i^{( T)} \;\deq\; \sum_{i : i \notin  T}\,.
\end{equation*}
When $ T = \{a\}$, we abbreviate $(\{a\})$ by $(a)$ in the above definitions; similarly, we write $(ab)$ instead of $(\{a,b\})$.
\end{Def}

Note that in the minor $H^{(T)}$ it is important to keep the original values of the matrix indices, and not to identify $\{1, \dots, N\} \setminus T$ with $\{1, \dots, N - \abs{T}\}$.

\begin{Def}[Conditional expectation] \label{definition: P Q}
Let $X \equiv X(H)$ be a random variable. For $i \in \{1, \dots, N\}$ we define the operations $P_i$ and $Q_i$ through
\begin{equation*}
P_i X \;\deq\; \E(X | H^{(i)}) \,, \qquad Q_i X \;\deq\; X - P_i X\,.
\end{equation*}
Note that $P_i$ and $Q_i$ form a family of commuting orthogonal projections on $L^2(\P)$.
\end{Def}

The following lemma collects basic bounds on $m$. In order to state it, we define the \emph{distance to the spectral edge}
\begin{equation} \label{def_kappa}
\kappa \;\equiv\; \kappa(E) \;\deq\; \absb{\abs{E} - 2}\,.
\end{equation}
\begin{lem}[Basic properties of $m$] \label{lem7}
Fix $\tau > 0$. Then for $z \in \f S$ we have
\begin{equation} \label{mupperlowerbounded}
\abs{m(z)} \;\asymp\; 1
\end{equation}
and
\begin{equation} \label{Immmupperlowerbounded}
\im m(z) \;\asymp\;
\begin{cases}
\sqrt{\kappa + \eta} & \text{if $\abs{E} \leq 2$}
\\
\frac{\eta}{\sqrt{\kappa + \eta}} & \text{if $\abs{E} \geq 2$}\,.
\end{cases}
\end{equation}
\end{lem}
\begin{proof}
The proof is an elementary exercise using \eqref{m_qsolution} and \eqref{m_id}.
\end{proof}

The following lemma collects properties of stochastic domination $\prec$. Roughly, it states that $\prec$ satisfies the usual arithmetic
 properties of order relations. We shall use it tacitly throughout the following.

\begin{lem}[Basic properties of $\prec$] \label{lemma: basic properties of prec}
\begin{enumerate}
\item
Suppose that $X(u,v) \prec Y(u,v)$ uniformly in $u \in U$ and $v \in V$. If $\abs{V} \leq N^C$ for some constant $C$ then
\begin{equation*}
\sum_{v \in V} X(u,v) \;\prec\; \sum_{v \in V} Y(u,v)
\end{equation*}
uniformly in $u$.
\item
Suppose that $X_{1}(u) \prec Y_{1}(u)$ uniformly in $u$ and $X_{2}(u) \prec Y_{2}(u)$ uniformly in $u$. Then
$X_{1}(u) X_{2}(u) \prec Y_{1}(u) Y_{2}(u)$
uniformly in $u$.
\end{enumerate}
\end{lem}
\begin{proof}
This is a simple union bound.
\end{proof}

The following resolvent identities form the backbone of all of our calculations. The idea behind them is that a Green function entry $G_{ij}$
depends strongly on the $i$-th and $j$-th columns of $H$, but weakly on all other columns. The first identity determines how to make
a Green function entry $G_{ij}$ independent of the randomness in the $k$-th row, where $k \neq i,j$.
The second identity expresses the dependence of a Green function entry $G_{ij}$ on the matrix elements in the $i$-th or in the $j$-th column of $H$. 

\begin{lem}[Resolvent identities] \label{lem8}
For any Hermitian matrix $H$ and $T \subset \{1, \dots, N\}$ the following identities hold.
If $i,j,k \notin T$ and $i,j \neq k$ then
\begin{equation} \label{RI1}
G_{ij}^{(T)} \;=\; G_{ij}^{(Tk)} + \frac{G_{ik}^{(T)} G_{kj}^{(T)}}{G_{kk}^{(T)}}\,.
\end{equation}
If $i,j \notin T$ satisfy $i \neq j$ then
\begin{equation} \label{RI2}
G_{ij}^{(T)} \;=\; - G_{ii}^{(T)} \sum_{k}^{(Ti)} H_{ik} G_{kj}^{(Ti)} \;=\; - G_{jj}^{(T)} \sum_k^{(Tj)} G_{ik}^{(Tj)} H_{k j}\,.
\end{equation}
\end{lem}
\begin{proof}
See Appendix \ref{sec:schur}.
\end{proof}

We also need the \emph{Ward identity}
\begin{equation} \label{Ward}
\sum_j |G_{ij}|^2 \;=\; \frac{\im G_{ii}}{\eta} \,.
\end{equation}
It may be proved by applying the resolvent identity to $G - G^{*}$, or, alternatively, by spectral decomposition:
\begin{equation*}
\sum_j \abs{G_{ij}}^2 \;=\; (G G^*)_{ii} \;=\; \sum_{j} \frac{\abs{u_j(i)}^2}{|\lam_j-z|^2}\;=\; \sum_j \ff{\eta}\im\frac{\abs{u_j(i)}^2}{\lam_j-z}\;=\; \ff{\eta} \im G_{ii}\,.
\end{equation*}
Assuming $\eta \gg \frac{1}{N}$, \eqref{Ward} shows that the squared $\ell^2$-norm $\frac{1}{N} \sum_j |G_{ij}|^2$ is smaller by the factor $\frac{1}{N\eta} \ll 1$
than the diagonal element $|G_{ii}|$. This fact plays a crucial role in our proof: it is precisely this gain of $\eta^{-1} \ll N$ as compared to $N$ that allows us to perform the bootstrapping underlying the proof. This identity is also the source of the ubiquitous factors of $\eta$ in denominators.
It was first used systematically in the proof the local semicircle law for random matrices in \cite{MR2981427}.

Finally, we record the following large deviation bounds.

\begin{lem}[Large deviation bounds] \label{lem9}
Let $\pb{X_i^{(N)}}$, $\pb{Y_i^{(N)}}$, $\pb{a_{ij}^{(N)}}$, and $\pb{b_i^{(N)}}$ be independent families of random variables, where $N \in \N$ and $i,j = 1, \dots, N$. Suppose that all entries $X_i^{(N)}$ and $Y_i^{(N)}$ are independent and satisfy the conditions
\begin{equation} \label{cond on X}
\E X \;=\; 0\,, \qquad \norm{X}_p \;\leq\; \mu_p
\end{equation}
for all $p$ with some constants $\mu_p$.
\begin{enumerate}
\item
Suppose that $\pb{\sum_i \abs{b_i}^2}^{1/2} \prec \Psi$. Then $\sum_i b_i X_i \prec \Psi$.
\item
Suppose that $\pb{\sum_{i \neq j} \abs{a_{ij}}^2}^{1/2} \prec \Psi$. Then $\sum_{i \neq j} a_{ij} X_i X_j \prec \Psi$.
\item
Suppose that $\pb{\sum_{i,j} \abs{a_{ij}}^2}^{1/2} \prec \Psi$. Then $\sum_{i,j} a_{ij} X_i Y_j \prec \Psi$.
\end{enumerate}
If all of the above random variables depend on an index $u$ and the hypotheses of (i) -- (iii) are uniform in $u$, then so are the conclusions.
\end{lem}

\begin{proof}
See Appendix \ref{sec:LDE}
\end{proof}

\section{Outline of the proof of Theorem \ref{Th3}}\la{Sec:local_law_proof_abstract}

\bgt

\ite[(a)] The diagonal entries of $G$ can be written using Schur's complement formula as
\be\la{SPSFintro} \ff{G_{ii}} \;=\; H_{ii}-z-\sum_{k,l}^{(i)}H_{ik}G^{(i)}_{kl}H_{l i}.\ee
Since the coefficients $G^{(i)}_{kl}$ in \eqref{SPSFintro} are independent of the entries $(H_{ik})_{k = 1}^N$, we can condition on $H^{(i)}$ and apply Lemma \ref{lem9} to find that the sum on the right-hand side of \eqref{SPSFintro} is with high probability close to its $P_i$-expectation, $\frac{1}{N} \sum_k^{(i)} G_{kk}^{(i)}$. Using Lemma \ref{lem8} we can get rid of the upper index $(i)$ up to a small error term, and find that the sum on the right-hand side of \eqref{SPSFintro} is with high probability close to $s$.

In order to quantify these errors, we introduce the random $z$-dependent error parameter $$\La \;\deq\; \max_{i,j}|G_{ij}-m\del_{ij}|\,.$$
Then using large deviation estimates as outlined above as well as the Ward identity \eqref{Ward}, we find that if either $\eta\ge 1$ or $\La \le N^{-\tau/10}$, then   \eqre{SPSFintro} reads \be\la{SPSFintrobis} \ff{G_{ii}}=-z-s+ O_\prec(\Psi_\Theta)\,, \qquad \Psi_\Tta \;\deq\; \sqrt{\frac{\im m+|s-m|}{N\eta}}\,.\ee
Combined with analogous arguments applied to the off-diagonal entries starting from \eqref{RI2}, we obtain the estimate
\begin{equation} \label{Gij_sketch}
\max_{i, j} \abs{G_{ij} - \delta_{ij} s} \;\prec\; \Psi_\Theta
\end{equation}
on the individual entries of $G$. Note that the diagonal entries are compared to the empirical quantity $s$ instead of the deterministic quantity $m$.
\ite[(b)]
In the next step, we compare $s$ to $m$.
Rewriting \eqre{SPSFintrobis} as \be\la{3121417h38}1=-zG_{ii}-sG_{ii}+O_\prec(\Psi_\Theta)\ee and averaging over $i$, we obtain
\begin{equation} \label{s_sketch_eq}
1 + zs + s^2 \;=\; O_\prec(\Psi_\Theta)\,.
\end{equation}
Recall from \eqref{m_id} that $m$ solves the same equation as \eqref{s_sketch_eq} without the error term on the right-hand side. Thus, concluding estimating $s - m$ from \eqref{s_sketch_eq} in terms of the error $O_\prec(\Psi_\Theta)$ involves a stability analysis of the quadratic equation \eqref{s_sketch_eq}.

\item[(c)]
If $\eta \geq 1$ we have $\Psi_\Theta \leq C N^{-1/2}$, and it is then not hard to deduce that $\abs{s - m} \prec N^{-1/2}$, so that \eqref{Gij_sketch} implies $\Lambda \prec N^{-1/2}$. Here the stability analysis of \eqref{s_sketch_eq} is simple because the two solutions of \eqref{s_sketch_eq} are well separated.
\ite[(d)]
In order to estimate $s - m$ for $\eta \leq 1$, we use a \emph{bootstrapping} from large to small scales, a self-consistent multiscale induction where estimates on a given scale are used as a priori information for obtaining estimates on the subsequent, smaller, scale. This is the core of the proof. The bootstrapping is performed at fixed energy $E$ along the sequence $z_k\deq E+ \ii (1-kN^{-3})$, where $k = 0,1, \dots, N^3 - N^{2 + \tau}$. At each step of the bootstrapping we establish the estimate $\Lambda(z_k) \prec (N \eta)^{-1/4}$.

The bootstrapping is started at $z_0$ for which the simple analysis of (c) applies.

To explain the bootstrapping, suppose that $\Lambda(z_k) \prec (N \eta)^{-1/4}$ holds at $z_k$. Then we use the trivial Lipschitz continuity of $\Lambda$ with Lipschitz constant $N^2$ and the estimate $\abs{z_{k+1} - z_k} = N^{-3}$ to deduce that with high probability $\Lambda \leq N^{-\tau/10}$ at $z_{k+1}$. This is the a priori assumption needed to obtain the estimates \eqref{SPSFintrobis} and \eqref{Gij_sketch} at $z_{k+1}$. Hence, we also obtain \eqref{s_sketch_eq} at $z_{k+1}$. In order to obtain a bound on $\abs{s - m}$ from \eqref{s_sketch_eq} and hence complete the induction step, we need to perform a stability analysis of \eqref{s_sketch_eq}. For small $\eta$ and $E$ near the spectral edge, this stability analysis requires some care, because the two solutions of \eqref{s_sketch_eq} can be close. This completes the induction step.

\ite[(e)]
From (c)--(d) we obtain $\Lambda \prec (N \eta)^{-1/4}$ for all $z_k$ at some fixed energy $E$. By a union bound, we obtain a simultaneous high-probability estimate for all $z$ in the $N^{-3}$-net $\f S \cap N^{-3} \Z^2 \subset \f S$. By Lipschitz continuity of $\Lambda$ with Lipschitz constant $N^2$, we extend this estimate to a simultaneous high-probability bound for all $z \in \f S$.
\ite[(f)]
In (e) we obtained the \emph{weak local law} $\Lambda \prec (N \eta)^{-1/4}$ for all $z \in \f S$. Although the error bound $(N \eta)^{-1/4}$ is far from optimal, this is a highly nontrivial and very useful estimate, because it prove that $\Lambda$ is small, and in particular bounded, for all $z \in \f S$. This estimate serves as the key input in the proof of the local law, where the optimal error bound is obtained.

The key tool behind the proof of the optimal error bound is a \emph{fluctuation averaging argument}, which states that the error term in \eqref{3121417h38} becomes much smaller after averaging over $i$. Thus, the optimal error in \eqref{s_m_outside} is in fact much smaller than in \eqref{3121417h38}. The main work behind this improvement is to estimate averages of the form $\sum_{i = 1}^N X_i$ where $X_i \deq Q_i \frac{1}{G_{ii}}$ has (by definition) expectation zero. Clearly, if the variables $X_i$ were all equal then the averaging would not change anything, and if they were all independent we would gain a factor $N^{-1/2}$ from the averaging. In fact, with the choice $X_i = Q_i \frac{1}{G_{ii}}$ neither is true, as the variables $X_i$ are obviously not equal but they are nevertheless strongly correlated. As it turns out, the averaging yields a gain of order $(N \eta)^{-1/2}$, so that the variables $X_i$ are almost uncorrelated for $\eta = 1$ and almost equal for $\eta \approx N^{-1}$.

\ent

\section{Proof of Theorem \ref{Th3} (b): weak local law} \label{sec:weak_law}

The conclusion of Theorem \ref{Th3} may be regarded as consisting of two separate achievements on the quantities $G_{ij}-m\del_{ij}$ and $s - m$: first, control for small values of $\eta$; second, optimal error bounds. These two achievements in fact entail separate difficulties in the proof, and we correspondingly separate the proof into two parts. The weak local law presented in this section pertains to the first achievement.

The goal of the weak local law is to establish the smallness of $G_{ij} - m \delta_{ij}$ for all $z \in \f S$, without however trying to get optimal error bounds. Once the weak law has been established, we will use it to improve the error estimates.

In this section we state and prove the weak law. The central quantity is the random $z$-dependent error parameter
$$\La \;\deq\; \max_{i,j}|G_{ij}-m\del_{ij}|\,.$$
\beg{propo}[Weak local law]\la{prop1}
We have $$\La \;\prec\; (N\eta)^{-1/4}$$ for $z\in \f S$.
\en{propo}

The rest of this section is devoted to the proof of Proposition \ref{prop1}. In addition to $\Lambda$, we shall need the additional error parameters
$$\La_* \;\deq\; \max_{i\ne j}|G_{ij}|\,,\qquad
\Tta \;\deq\; |s-m|\,,\qquad \Psi_\Tta \;\deq\; \sqrt{\frac{\im m+\Tta}{N\eta}}\,.
$$
Our starting point is Schur's complement formula (see Lemma \ref{lem:schur}), which we write as
\be\la{SPSF} \ff{G_{ii}} \;=\; H_{ii}-z-\sum_{k,l}^{(i)}H_{ik}G^{(i)}_{kl}H_{l i}\,.\ee
We have the following simple identity.
\beg{lem}\la{lem11}
We have $$\ff{G_{ii}} \;=\; -z-s+Y_i\,,
$$ where we defined
\begin{equation} \label{def_Y_A_Z}
Y_i \;\deq\; H_{ii}+A_i-Z_i\,, \qquad A_i \;\deq\; \ff{N}\sum_k\frac{G_{ki}G_{ik}}{G_{ii}}\,, \qquad Z_i \;\deq\; Q_i\sum_{k,l}^{(i)}H_{ik}G^{(i)}_{kl}H_{l i}\,.
\end{equation}
\en{lem}
\bpr 
Using \eqref{RI1} we get
$$P_i\sum_{k,l}^{(i)}H_{ik}G^{(i)}_{kl}H_{l i} \;=\; \ff{N} \sum_{k}^{(i)} G_{kk}^{(i)} \;=\;  
\ff{N}\sum_{k}G_{kk} 
-\ff{N}\sum_k\frac{G_{ki}G_{ik}}{G_{ii}} \;=\; s-\ff{N}\sum_k\frac{G_{ki}G_{ik}}{G_{ii}}\,.$$
Hence,
$$\sum_{k,l}^{(i)}H_{ik}G^{(i)}_{kl}H_{l i} \;=\; s-\ff{N}\sum_k\frac{G_{ki}G_{ik}}{G_{ii}}+Q_i\sum_{k,l}^{(i)}H_{ik}G^{(i)}_{kl}H_{l i} \;=\; s-A_i+Z_i\,.$$
The claim now follows using \eqre{SPSF}.
\epr

The strategy behind the proof of Proposition \ref{prop1} is a self-consistent multiscale approach, which involves a bootstrapping from the large scale $\eta \geq 1$ down to the small scale $\eta = N^{-1 + \tau}$. The bootstrapping hypothesis will be the event $\phi = 1$, where we defined the $z$-dependent indicator function
\be\la{6}\phi \;\deq\; \indb{\La\le N^{-\tau/10}}\,.
\ee
The following result gives trivial bounds on entries of $G$ on the event $\phi = 1$.
\begin{lem} \label{rmklem7}
Let $p \in \N$ be fixed. Then we have
\begin{equation*}
\phi \abs{G_{ij}^{(T)}} + \phi \absbb{\frac{1}{G_{ii}^{(T)}}} \;\prec\; 1
\end{equation*}
uniformly for $i,j \notin T \subset \{1, \dots, N\}$ satisfying $\abs{T} \leq p$.
\end{lem}
\begin{proof}
The proof follows using \eqref{mupperlowerbounded} and a repeated application of \eqref{RI1}, whose details we leave to the reader.
\end{proof}

The following lemma contains the key a priori estimates used to perform the bootstrapping. Note that the assumptions for the estimates are either $\eta \geq 1$ (start of bootstrapping) or $\phi = 1$ (bootstrapping assumption used for iteration).

\beg{lem}[Main estimates]\la{lem12} For $z\in \f S$ we have
\be\la{7}\pb{\phi+\ind{\eta\ge 1}}\pb{\La_*+|A_i|+|Z_i|} \;\prec\; \Psi_\Tta\ee
and
\be\la{8}\pb{\phi+\ind{\eta\ge 1}}|G_{ii}-s| \;\prec\; \Psi_\Tta\,.\ee
\en{lem}

\bpr
We begin with $Z_i$ under the bootstrapping assumption $\phi = 1$. 
We split
\begin{equation} \label{Z_i_est_split}
\phi \abs{Z_i} \;\le\; \phi \absBB{\sum_k^{(i)}\pbb{|H_{ik}|^2-\ff{N}}G_{kk}^{(i)}}+\phi\absBB{\sum_{k\ne l}^{(i)}H_{ik}G^{(i)}_{kl}H_{l i}}\,.
\end{equation}
Using Lemmas  \ref{lem9} and \ref{rmklem7} we estimate the first term of \eqref{Z_i_est_split} as
$$ \phi \absBB{\sum_k^{(i)}\pbb{|H_{ik}|^2-\ff{N}}G_{kk}^{(i)}} \;\prec\; \phi \pBB{\ff{N^2}\sum_{k}^{(i)}|G_{kk}^{(i)}|^2}^{1/2} \;\prec\; \phi N^{-1/2} \;\le\; C\phi \Psi_\Tta\,,$$
where in the last step we used \eqref{Immmupperlowerbounded} to estimate $\Psi_\Theta \geq c N^{-1/2}$.

Next, using Lemma \ref{lem9} we estimate the second term of \eqref{Z_i_est_split} as
$$\phi\absBB{\sum_{k\ne l}^{(i)}H_{ik}G^{(i)}_{kl}H_{l i}} \;\prec\; \phi\pBB{\sum_{k,l}^{(i)}\ff{N^2}|G_{kl}^{(i)}|^2}^{1/2} \;=\; \phi\pBB{\sum_k^{(i)}\ff{N^2\eta}\im G_{kk}^{(i)}}^{1/2}\,.$$
Here the second step follows from the Ward identity \eqref{Ward} applied to the matrix $H^{(i)}$. Using \eqref{RI1} and Lemma \ref{rmklem7} we therefore conclude
$$\phi\absBB{\sum_{k\ne l}^{(i)}H_{ik}G^{(i)}_{kl}H_{l i}} \;\prec\; \phi\pBB{\sum_k\frac{\im G_{kk}+\La_*^2}{N^2\eta}}^{1/2} \;\le\; \phi\Psi_\Tta+\phi\La_*\,.$$ 
We conclude from \eqref{Z_i_est_split} that
\begin{equation} \label{Zi_est}
\phi \abs{Z_i} \;\prec\; \phi \Psi_\Theta + \phi \Lambda_*\,.
\end{equation}

Next, from Lemma \ref{rmklem7} we get the trivial bound
\begin{equation} \label{Ai_est}
\phi|A_i| \;\prec\; \phi \La_*^2\;\le\; \phi\La_*\,.
\end{equation}

In order to estimate $\La_*$, we take $i\ne j$ and use \eqre{RI2} twice to get
\be\la{EqAuxiliere}G_{ij} \;=\; -G_{ii}\sum_{k}^{(i)}H_{ik}G^{(i)}_{kj}\;=\;-G_{ii}G^{(i)}_{jj}\pBB{H_{ij}-\sum_{k,l}^{(ij)}H_{ik}G_{kl}^{(ij)}H_{l j}}\,.\ee
We estimate the second term using Lemma \ref{lem9} and the identity \eqref{Ward} (applied to $H^{(ij)}$) according to
\begin{multline*}
\phi\absBB{\sum_{k,l}^{(ij)}H_{ik}G_{kl}^{(ij)}H_{l j}} \;\prec\; \phi\pBB{\frac{1}{N^2}\sum_{k,l}^{(ij)}|G_{kl}^{(ij)}|^2}^{1/2} \;=\; \phi\pBB{\frac{1}{N^2 \eta} \sum_{k}^{(ij)}\im G_{kk}^{(ij)}}^{1/2}
\\
\le\; \phi\pBB{\sum_k\frac{\im G_{kk}+C \La_*^2}{N^2\eta}}^{1/2}\le \phi\Psi_\Tta+ C \phi(N\eta)^{-1/2}\La_*\,,
\end{multline*}
where the third step follows using a repeated application of \eqref{RI1}. Using $\abs{H_{ij}} \prec N^{-1/2}$, Lemma \ref{rmklem7}, and the estimate $(N\eta)^{-1/2}\le CN^{-\tau/4}$ valid for $z \in \f S$, we therefore conclude that
$$\phi\La_* \;\prec\; \phi\Psi_\Tta+\phi N^{-\tau/4}\La_*\,.$$
Using the definition of $\prec$, we deduce that
\begin{equation} \label{est_Lambda_star}
\phi\La_* \;\prec\; \phi\Psi_\Tta\,.
\end{equation}
Plugging \eqref{est_Lambda_star} into \eqref{Zi_est} and \eqref{Ai_est} concludes the proof of the term $\phi$ in \eqref{7}.

Moreover, the proof of the term $\ind{\eta \geq 1}$ is almost the same, using the trivial bound $|G_{ij}^{(T)}| \le 1$ instead of Lemma \ref{rmklem7}.

What remains is the estimate \eqre{8}.
From Lemma \ref{lem11} and \eqref{7} we get
$$(\phi+\ind{\eta\ge 1})\absBB{\ff{G_{ii}}-\ff{G_{jj}}} \;\le\; (\phi+\ind{\eta\ge 1})(|Y_i|+|Y_j|) \;\prec\; \phi \Psi_\Tta\,,$$ which implies, by Lemma \ref{rmklem7} that $$\pb{\phi+\ind{\eta\ge 1}}\abs{{G_{ii}}-{G_{jj}}} \;\prec\; \Psi_\Tta\,.$$ Averaging over $j$ yields \eqref{8}.
\epr

We emphasize here the essential role played by the identity \eqref{Ward} in the proof. It is the source of the factors $(N \eta)^{-1}$ in our error bounds. Using it, we can estimate the sum over $N$ Green function entries by $\eta^{-1}$ times a Green function entry. Since $\eta^{-1} \ll N$, this is a nontrivial gain. This gain is responsible not only for the optimal error bounds for $\Lambda$ but also for the self-improving mechanism in the estimation of $\Lambda$ that allows us to bootstrap in $\eta$.

The following elementary lemma is used to estimate $s - m$. It may be regarded as a quantitative stability result for the equation \eqref{m_id}.

\beg{lem}\la{lem13}Let $z\in \f S$. The equation $$u^2+zu+1=0$$ has two solutions, $m\in \C_+$ and $\tilde{m}\in \C_-$. They satisfy $$|m-\tilde{m}| \;\asymp\; \sqrt{\ka+\eta}\,.$$ Moreover, if $u$ solves
\begin{equation} \label{u_eq}
u^2+zu+1+r \;=\; 0 \qquad (|r|\le 1)
\end{equation}
 then
 \begin{equation} \label{u-m_eq}
 \min\{|u-m|,|u-\tilde{m}|\} \;\le\; \frac{C|r|}{\sqrt{\ka+\eta+|r|}}
 \end{equation}
for some constant $C$.
\en{lem}

\begin{proof}
See Appendix \ref{sec:stability}
\end{proof}

The following lemma starts the bootstrapping at $\eta \geq 1$.

\beg{lem}\la{lem14}For $z\in \f S$, we have $$\ind{\eta\ge 1}\La \;\prec\; N^{-1/2}$$
\en{lem}

\bpr
Let $\eta \geq 1$.
From \eqref{7} an Lemma \ref{lem7} we get $\Lambda_* \prec N^{-1/2}$. What remains therefore is the estimate of $G_{ii} - m$. By Lemmas \ref{lem11} and \ref{lem12}, we have $$1+sG_{ii}+zG_{ii} \;=\; O_\prec(N^{-1/2})\,.$$
Averaging over $i$ yields
\begin{equation} \label{avg_s_eq}
1+s^2+sz \;=\; O_\prec(N^{-1/2})\,.
\end{equation}
Lemma \ref{lem13} therefore gives
$$\min\{|s-m|, |s-\tilde{m}|\} \;\prec\; N^{-1/2}\,.$$
We now need to make sure that $s$ is close to $m$ and not $\tilde m$.
Since $\eta \geq 1$, it is easy to see that $\im \tilde m \leq -c$ for some positive constant $c$. Since $\im s > 0$, we therefore deduce that
$$\Tta \;=\; |s-m| \;\prec\; N^{-1/2}\,.$$
Using \eqref{8}, we deduce that $$|G_{ii}-m| \;\le\; |G_{ii}-s|+\Tta \;\prec\; N^{-1/2}\,,$$ 
as claimed.
\epr

We may now come to heart of the proof of Proposition \ref{prop1}: the self-consistent multiscale bootstrapping from large to small $\eta$. If all error estimates were deterministic, we could do a standard continuity argument by choosing a continuous path $\eta(t) = 1 - t$. However, owing to the stochastic nature of our estimates, the bootstrapping has to be done in a finite number of steps. At each step we lose some probability, which imposes an upper bound on the number of allowed steps. On the other hand, the steps have to be small enough to be able to carry over information from one step to the next using continuity. At each step, the deterioration of the error probabilities has to be tracked carefully. Roughly, at each step we lose an additive factor $N^{-D}$ in the error probabilities, which means that we can perform an order $N^C$ steps for any fixed constant $C$. Using the Lipschitz continuity of all error parameters, it turns out to be enough to perform steps of size $N^{-3}$.   Hence, for the \emph{stochastic continuity argument} that underlies the self-consistent multiscale approach, we replace the purely topological notions of a continuous path and continuous error parameters with the quantitative notions of a lattice and Lipschitz continuous error parameters.  

Thus, we define the discrete subset $\wh {\f S} \deq \f S \cap (N^{-3}\Z^2)$. Since $s$, $m$, $G_{ij}$, and $\Psi$ are all $N^2$-Lipschitz on $\f S$, it is enough to prove Proposition \ref{prop1} for $z \in \wh {\f S}$. This is the content of the following lemma.

\blem\la{lem15}  For $z \in \wh {\f S}$ we have $\La\prec (N\eta)^{-1/4}$.
\elem

\bpr The case $\eta\ge 1$ follows from Lemma \ref{lem14}. For the other $z \in \wh {\f S}$ we construct a lattice of points $z_k \in \wh {\f S}$ as follows. Let $z_0=E+\mathrm{i}\in \wh {\f S}$ define $z_k \deq E+\mathrm{i}\eta_k$ with $\eta_k\deq 1-kN^{-3}$.

Fix $\eps\in (0, \tau/16)$ and $D> 0$. Set $\del_k\deq (N\eta_k)^{-1/2}$ and define the events $$\Omega_k \;\deq\; \hbb{\Tta(z_k)\le\frac{N^\eps\del_k}{\sqrt{\ka+\eta_k+\del_k}}}\,, \qquad \Xi_k \;\deq\; \hb{\La(z_k)\le N^{2\eps}\sqrt{\del_k}}\,.$$
By Lemma \ref{lem14}, $$\p((\Omega_0\cap\Xi_0)^c) \;\le\; N^{-D}\,.$$
In the bootstrapping step we prove that, conditioned on the event $\Omega_{k-1} \cap \Xi_{k-1}$, the event $\Omega_k \cap \Xi_k$ has high probability. The estimate on $\Theta$ is obtained from Lemma \ref{lem13} by noting that $s$ solves an equation of the form \eqref{u_eq}. Then we use the estimate \eqref{u-m_eq}. The key difficulty is in making sure that $s$ is close to $m$ and not $\tilde m$. Depending on the relative sizes of $\abs{m - \tilde m}$ and $\abs{r}$, we need to consider two cases. In the first case $\abs{m - \tilde m}$ is large enough for us to deduce by continuity from the previous step that $u$ cannot be close to $\tilde m$. In the second case $\abs{m - \tilde m}$ is so small that it does not matter whether $u$ is close to $m$ or $\tilde m$. The crossover from the first to the second case happens at the index
$$K \;\deq\; \min\{k\col \del_k\ge N^{-2\eps}(\ka+\eta_k)\}\,.$$
Note that, as $k$ increases, $\eta_k$ decreases and $\del_k$ increases. Hence, $\del_k<N^{-2\eps}(\ka+\eta_k)$ for all $k<K$.

\subsubsection*{Case 1: $1\le k < K$} Since $\La$ is $N^2$-Lipschitz on $\f S$, we have $\phi(z_k)=1$ on $\Xi_{k-1}$. Using Lemma \ref{lem12}, we therefore deduce that $$\ind{\Xi_{k-1}}\pb{|Y_i(z_k)|+\La_*(z_k)} \;\prec\; \del_k\,,$$
where we also used that $\im m + \phi \Theta \leq C$.
Plugging this into Lemma \ref{lem11}, averaging over $i$, and invoking Lemma  \ref{lem13} yields
\begin{equation} \label{P_Xi_bs}
\p\pBB{\Xi_{k-1}\cap \pBB{\min\hb{|s(z_k)-m(z_k)| \,,|s(z_k)-\tilde{m}(z_k)|}>\frac{N^\eps\del_k}{\sqrt{\ka+\eta_k+\del_k}}}} \;\le\; N^{-D}\,.
\end{equation}
On the other hand, by Lipschitz continuity, on $\Omega_{k-1}$ we have $$|s(z_k)-m(z_k)| \;\le\; \frac{2N^\eps\del_k}{\sqrt{\ka+\eta_k+\del_k}}\,.$$
Moreover, by Lemma \ref{lem13}, we have $$|m(z_k)-\tilde{m}(z_k)| \;\asymp\; \sqrt{\ka+\eta_k} \;\geq\; \frac{N^{2 \epsilon} \delta_k}{\sqrt{\kappa + \eta_k + \delta_k}}\,,$$ where in the last step we used that $\delta_k \leq N^{-2\epsilon} (\kappa + \eta_k)$.
Therefore, on $\Omega_{k-1}$ we have
$$\min\hb{|s(z_k)-m(z_k)|,|s(z_k)-\tilde{m}(z_k)| } \;=\; |s(z_k)-m(z_k)|\,.$$
We have therefore proved that
\begin{equation} \label{boots1}
\p\pb{\Omega_{k-1}\cap \Xi_{k-1}\cap \Omega_k^c} \;\le\; N^{-D}\,.
\end{equation}

Moreover, using Lemma \ref{lem12} we find
\begin{equation*}
\indb{\Xi_{k-1} \cap \Omega_k} \Lambda(z_k) \;\prec\; \indb{\Xi_{k-1} \cap \Omega_k} \pb{\Psi_\Theta(z_k) + \Theta(z_k)} \;\leq\; C N^\epsilon \sqrt{\delta_k}\,,
\end{equation*}
where in the first step we used that $\phi(z_k) = 1$ on $\Xi_{k - 1}$ and the general estimate $\Lambda \leq \Lambda_* + \max_i \abs{G_{ii} - s} + \Theta$, and in the second step we used the definitions of $\Omega_k$ and $\delta_k$.
We conclude that
\begin{equation} \label{boots2}
\p(\Omega_{k-1}\cap\Xi_{k-1}\cap\Omega_k\cap \Xi_k^c) \;\le\; N^{-D}\,.
\end{equation}

Summarizing, defining $B_k\deq \Omega_k\cap\Xi_k$, we get from \eqref{boots1} and \eqref{boots2} that
\begin{equation} \label{k_leq_concl}
\p(B_0^c) \;\le\; N^{-D} \,, \qquad \p(B_{k-1}\cap B_{k}^c) \;\le\; 2N^{-D} \qquad (k < K)\,.
\end{equation}
Therefore we get for $k < K$ that $$\p(B_k^c) \;\le\; \p(B_0^c)+\sum_{l=1}^{k} \p(B_l^c\cap B_{l-1}) \;\le\; 2N^3N^{-D}\,.$$

\subsubsection*{Case 2: $ k\ge K$}
In this case the argument is similar but somewhat easier, since there is no need to use the assumption $\Omega_{k-1}$.
Then the argument is similar, but easier: there is no need to track $\Tta$ (so that we do not need the events $\Omega_k$). If $$\min\hb{|s(z_k)-m(z_k)|,|s(z_k)-\tilde{m}(z_k)| }\;\le\; \frac{N^\eps\del_k}{\sqrt{\ka+\eta_k+\del_k}}$$ then $$\Tta(z_k) \;\le\; N^\eps\sqrt{\del_k} +|m(z_k)-\tilde{m}(z_k)| \;\le\; CN^\eps\sqrt{\del_k}\,,$$
where we used Lemma \ref{lem13}. Using \eqref{P_Xi_bs} and Lemma \ref{lem12} we therefore deduce that
$$\p(\Xi_{k-1}\cap \Xi_k^c) \;\le\; N^{-D}\,.$$
From \eqref{k_leq_concl} we therefore conclude that $$\p(\Xi_k^c)\le CN^{3-D}$$ for all $k$. Since $\eps > 0$ and $D$ were arbitrary, the claim follows.
\epr

\section{Proof of Theorem \ref{Th3} (c): optimal error bounds} \label{sec:opt_bounds}

In this section we complete the proof of Theorem \ref{Th3} by improving the error bounds from Proposition \ref{prop1} to optimal ones. The main observation is that the equation for $s$ of the form \eqref{avg_s_eq} arises from an averaging over $i$. When deriving \eqref{avg_s_eq} (and its analogue for smaller $\eta$), we simply estimated the average $\frac{1}{N} \sum_i Y_i$ by $\max_i \abs{Y_i}$. (Recall the definition \eqref{def_Y_A_Z}.) In fact, the random variables $Y_i$ have a small expectation, so that their average should be smaller than the typical size of each individual variable. If they were independent, we would gain a factor $N^{-1/2}$ by a trivial concentration result. However, they are not independent. In fact, for small $\eta$ different $Y_i$ are strongly correlated, and one typically has
\begin{equation*}
\frac{1}{N} \var(Y_i) \;\ll\; \var \pbb{\frac{1}{N} \sum_j Y_j} \;\ll\; \var(Y_i)\,.
\end{equation*}
The upper bound corresponds to fully correlated variables, and the lower bound to uncorrelated variables. For general $\eta$, the variables $Y_i$ are in between these two extremes. The extreme cases are reached for $\eta \approx N^{-1}$ (almost fully correlated) and $\eta \approx 1$ (almost uncorrelated, in fact almost independent); see Remark \ref{rem: fa_scale} below for more details.

The key result that allows us to obtain optimal bounds on the average $\frac{1}{N} \sum_i Y_i$ is the following \emph{fluctuation averaging} result. Recall the definitions of $P_i$ and $Q_i$ from Definition \ref{definition: P Q}.

\beg{propo}[Fluctuation averaging]\la{prop16}
Suppose that for $z \in \f S$ we have
\be\la{04121411h45}\La_* \;\prec\; \Phi,\ee where
\begin{equation} \label{admissible Psi}
N^{-1/2} \;\le\; \Phi \;\le\; N^{-c}
\end{equation}
for some constant $c>0$. Then for $z \in \f S$ we have $$\ff{N}\sum_i Q_i\ff{G_{ii}}\;=\;O_\prec (\Phi^2)\,.$$
\en{propo}

The proof of Proposition \ref{prop16} is postponed to next section. In this section we use Proposition \ref{prop16} to complete the proof of Theorem \ref{Th3}. Throughout, we use that $1 - \phi \prec 0$ in $\f S$, as follows from Proposition \ref{prop1}. This allows us for instance to drop the indicator functions on the left-hand sides of \eqref{7} and \eqref{8}.

\blem\la{lem17} Under the assumptions of Proposition \ref{prop16}, the error terms $Y_i$ defined in Lemma \ref{lem11} satisfy $$\ff{N}\sum_{i=1}^NY_i \;=\; O_\prec(\Phi^2)\,.$$\elem

\bpr
From Schur's complement formula \eqref{SPSF} we get
$$Q_i\ff{G_{ii}} \;=\; H_{ii}-Q_i\sum_{k,l}^{(i)}H_{ik}G_{kl}^{(i)}H_{l i} \;=\; H_{ii}-Z_i \;=\; Y_i-A_i\,,$$
so that $Y_i = A_i+Q_i1/G_{ii}$. From the definition of $A_i$ in Lemma \ref{lem11} we obtain $A_i = O_\prec(\Phi^2)$, where we used Proposition \ref{prop1} and \eqref{mupperlowerbounded}. The claim now follows from Proposition \ref{prop16}.
\epr

\begin{rmk} \label{rem: fa_scale}
Since $A_i = O_\prec(\Phi^2)$ even without averaging, we find that the important part of $Y_i$ is $X_i \deq Q_i\ff{G_{ii}}$. Suppose for simplicity that we are in the bulk, i.e.\ $E \in [-2+c, 2 - c]$ for some constant $c > 0$. As evidenced by Theorem \ref{Th3}, $X_i$ is typically of size $(N \eta)^{-1/2}$ in the bulk. Moreover, again by Theorem \ref{Th3}, we have the bound $\Lambda_* \prec \Phi$ with $\Phi \deq (N \eta)^{-1/2}$. Thus, we find that the effect of the averaging on $X_i$ is to multiply it by a factor $(N \eta)^{-1/2}$. We conclude that for $\eta \approx 1$, the average $\frac{1}{N} \sum_i X_i$ is smaller than $X_i$ by a factor $N^{-1/2}$, which corresponds to the behaviour of independent random variables $X_1, \dots, X_N$. In contrast, for $\eta \approx N^{-1}$, the average $\frac{1}{N} \sum_i X_i$ is of the same size as $X_i$, which corresponds to the behaviour of fully correlated random variables $X_1 = \cdots = X_N$.  
\end{rmk}

Now, suppose that $\Tta\prec (N\eta)^{-\si}$ in $\f S$, where $\si \in [1/4,1]$. Using Proposition \ref{prop1} and Lemma \ref{lem12}, we get
\begin{equation} \label{def_Phi_s}
\La_*+|Y_i|+|G_{ii}-s| \;\prec\; \Phi_\si \;\deq\; \sqrt{\frac{\im m +(N\eta)^{-\si}}{N\eta}}\,.
\end{equation}
Using Lemma \ref{lem11}, we get $$-s-z+Y_i \;=\; \ff{G_{ii}-s+s} \;=\; \ff{s}-\ff{s^2}(G_{ii}-s)+\ff{s^2G_{ii}}(G_{ii}-s)^2\,.$$
Averaging over $i$ and using that $\abs{s^2G_{ii}}^{-1} \prec 1$, we get
$$-s-z+\ff{N}\sum_i Y_i \;=\; \ff{s}+O_\prec (\Phi_\si^2)\,.$$
Using Lemma \ref{lem17}, we get
\begin{equation} \label{flav_qeq}
s^2+sz+1 \;=\; O_\prec(\Phi_\si^2)\,.
\end{equation}
This is the desired perturbed quadratic equation for $s$.

Next, fix $\eps,D>0$ and define
\begin{equation*}
\cal E_\sigma \;\deq\; \frac{N^\eps\Phi_\si^2}{\sqrt{\ka+\eta+\Phi_\si^2}}\,.
\end{equation*}
From \eqref{flav_qeq} and Lemma \ref{lem13} combined with the fact that $s$, $m$, and $\cal E_\sigma$ are $N^2$-Lipschitz, we find that there exists an event $\Xi$ that does not depend on $z$ such that $1 - \ind{\Xi} \prec 0$ and
\be\la{512140017}
\ind{\Xi}\min\{|s-m|\,, |s-\tilde{m}|\} \;\le\; \cal E_\sigma
\ee
for all $z \in \f S$.

As in the proof of Lemma \ref{lem15}, we now have to show that $s$ is always close to $m$ and not $\tilde m$. The argument is a continuity argument very similar to that from the proof of Lemma \ref{lem15}. In fact, since we do not have to worry about the deterioration of error probabilities, it is even easier. We do not even have to use a lattice, since on the high-probability event $\Xi$ the estimate \eqref{512140017} holds simultaneously for all $z \in \f S$. The case $\eta \geq 1$ is again trivial by the simple observation that $\im \tilde m \leq c$ for some positive constant $c$. Let us therefore focus on the case $\eta \leq 1$. For the following we fix a realization $H \in \Xi$, so that the argument is entirely deterministic.

Let $z = E + \ii \in \f S$ and define $z_t \deq z + \ii \eta_t$ where $\eta_t \deq 1 - t$. The crossover time is defined by
\begin{equation*}
T \;\deq\; \min \hb{t \geq 0 \col \Phi_\sigma(z_t)^2 \geq N^{-2\epsilon} (\kappa + \eta_t)}\,.
\end{equation*}
Using the initial bound $\abs{s(z_0) - m(z_0)} \leq \cal E_\sigma(z_0)$ and the estimates
\begin{equation*}
\min\hb{|s(z_t)-m(z_t)| \,, |s(z_t)-\tilde{m}(z_t)|} \;\le\; \cal E_\sigma(z_t) \,, \qquad \abs{m(z_t)-\tilde{m}(z_t)} \;\geq\; c \sqrt{\kappa + \eta_t} \;\gg\; \cal E_\sigma(z_t)
\end{equation*}
valid for $t \in [0,T]$, we get using the continuity of $s$, $m$, and $\tilde m$, that $|s(z_t)-m(z_t)| \leq \cal E_\sigma(z_t)$ for $t \in [0,T]$.

Next, for $t > T$ we use the estimate
\begin{equation*}
|s(z_t)-m(z_t)| \;\leq\; \cal E_\sigma(z_t) + \abs{m(z_t)-\tilde{m}(z_t)} \;\leq\; 2 N^{\epsilon} \cal E_\sigma(z_t)\,.
\end{equation*}
We conclude that $\abs{s - m} \leq 2 N^\epsilon \cal E_\sigma$ on $\Xi$ and for all $z \in \f S$. Since $\epsilon$ and $D$ were arbitrary, we conclude that if $\Theta \prec (N \eta)^{-\sigma}$ then
\begin{equation} \label{Theta_self_impr}
\Theta \;\prec\; \frac{\Phi_\si^2}{\sqrt{\ka+\eta+\Phi_\si^2}} \;\leq\; \frac{\im m}{N\eta \sqrt{\ka+\eta}}+\sqrt{\frac{(N\eta)^{-\si}}{N\eta}} \;\leq\; \frac{C}{N\eta}+(N\eta)^{-1/2-\si/2}\,,
\end{equation}
where in the last step we used Lemma \ref{lem7}. Note that the second estimate is in general wasteful. It turns out to be optimal for $E \in [-2,2]$, but if $E$ is away from the limiting spectrum $[-2,2]$, it can be improved. We shall return to this point in Lemma \ref{lem:improved_ll} below.

In conclusion, we have shown for any $\sigma \in [1/4,1]$ that
\begin{equation} \label{sc_iter}
\Theta \;\prec\; (N \eta)^{-\sigma} \qquad \Longrightarrow \qquad \Theta \;\prec\; (N \eta)^{-1/2 -\sigma/2}\,.
\end{equation}
Starting from $\sigma = 1/4$ from Proposition \ref{prop1} and iterating \eqref{sc_iter} a bounded number of times, we get $\Theta \prec (N \eta)^{-1}$. (Note that the number of iterations is independent of $N$; it depends only on the constant $\epsilon$ in the definition of $\prec$, which is arbitrary but fixed.) This concludes the proof of \eqref{121414h53}. Finally, \eqref{121414h54} follows from \eqref{121414h53} and the estimates \eqref{7} and \eqref{8}. This concludes the proof of Theorem \ref{Th3}.

\section{Fluctuation averaging: proof of Proposition \ref{prop16}} \label{sec:averaging}

The first instance of the fluctuation averaging mechanism appeared in \cite{EYY2} for the Wigner case. A different proof (with a better bound on the constants) was given in \cite{EYYrigi}.
A conceptually streamlined version of the original proof was extended to
sparse matrices \cite{EKYY1} and to sample covariance matrices \cite{PY1}. 
Finally, an extensive analysis in \cite{EKYfluc} treated the fluctuation averaging of general polynomials
of Green function entries and identified the order of cancellations depending on the
algebraic structure of the polynomial. Moreover, in \cite{EKYfluc} an  additional cancellation effect was found
for the quantity $Q_i|G_{ij}|^2$. These improvements played a key role in obtaining 
the diffusion profile for the Green function of band matrices. The version we present here is based on the greatly simplified proof given in \cite{EKYY4}.

We start with a simple lemma which summarizes the key properties of $\prec$ when combined with expectation. Note that if $X$ and $Y$ are deterministic, $X\,\prec\, Y$ from Definition \ref{def:stochdom} simply means that for each $\eps>0$, we have for large enough $N$ and all $u$ that $X^{(N)}(u) \,\le \,  N^\epsilon Y^{(N)}(u)$.

\begin{lem} \label{lem: exp prec}
Suppose that the deterministic control parameter $\Phi$ satisfies $\Phi  \geq N^{-C}$ for some constant $C > 0$, and that for all $p$ there is a constant $C_p$ such that the nonnegative 
random variable $X$ satisfies $\E X^p \leq N^{C_p}$. 
Then we have the equivalence
\be\la{eq:stoch_domination}X \prec \Phi \qquad \Longleftrightarrow \qquad \E X^n \prec \Phi^n \txt{ for any fixed $n \in \N$} \,.\ee
Moreover, with $P_i$ and $Q_i$ as in Definition \ref{definition: P Q}, we have
\begin{equation} \label{stoch_dom_P}
X \prec \Phi \qquad \Longrightarrow \qquad P_i X^n \prec \Phi^n \,\txt{ and }\, Q_i X^n \prec \Phi^n \txt{ for any fixed $n \in \N$}\,.
\end{equation}
Finally, if $X \equiv X(u)$ and $\Phi \equiv \Phi(u)$ depend on some parameter $u$ and any of the above hypothesis is uniform in $u$, then so is the corresponding conclusion.
\end{lem}

\bpr
If $\E X^n \prec \Phi^n$ then for any $\eps>0$ we get from Markov's inequality
$$\P(X > N^\eps\Phi)\;\le \; \frac{\E X^n}{N^{n\eps}\Phi^n}\;\le \; N^{-\eps(n-1)}\,.$$
Choosing $n$ large enough (depending on $\epsilon$) proves the implication  ``$\Longleftarrow$'' of    \eqre{eq:stoch_domination}. Conversely, if  $X \prec \Phi$ then for any $D > 0$ we get
\begin{multline*} 
\E X \;=\; \E X \ind{X \leq N^{\epsilon} \Phi} + \E X \ind{X > N^{\epsilon} \Phi}
\\
\leq\; N^{\epsilon} \Phi + \sqrt{\E X^2} \sqrt{\P(X > N^\epsilon \Phi)} \;\leq\; N^\epsilon \Phi + N^{C_2/2 - D/2}\,,
\end{multline*}
Using $\Phi  \geq N^{-C}$ and choosing $D$ large enough, we obtain the implication ``$\Longrightarrow$'' of \eqref{eq:stoch_domination} for $n=1$. The same implicitation for arbitrary $n$ follows from the fact that $X \prec \Phi$ implies $X^n \prec \Phi^n$ for any fixed $n$. Moreover, \eqref{stoch_dom_P} follows using \eqre{eq:stoch_domination} and Jensen's inequality for conditional expectations. The final claim about uniformity is trivial.
\epr

We shall apply Lemma \ref{lem: exp prec} to the entries of $G$. In order to verify its assumptions, we record the following bounds.

\begin{lem}\label{lemma:B2}
Suppose that $\Lambda_* \prec \Phi$ for some deterministic control parameter $\Phi$ satisfying \eqref{admissible Psi}. Fix $\ell \in \N$. Then for any $i \neq j$ and  $T \subset \{1, \dots, N\}$ satisfying $\abs{T} \leq \ell$ and $i,j \notin T$ we have
\begin{equation} \label{basic bounds on GT}
G_{ij}^{(T)} \;=\; O_\prec(\Phi) \,, \qquad \frac{1}{G_{ii}^{(T)}} \;=\; O_\prec(1)\,.
\end{equation}
Moreover, we have the rough bounds $\absb{G_{ij}^{(T)}} \leq N$ and
\begin{equation} \label{rough bound on 1/G}
\E \absBB{\frac{1}{G_{ii}^{(T)}}}^n \;\leq\; N^\epsilon
\end{equation}
for any $\epsilon > 0$ and $N \geq N_0(n, \epsilon)$.
\end{lem}

\begin{proof}
The bounds \eqref{basic bounds on GT} follow easily by a repeated application of \eqref{RI1}, the estimate $\Lambda \prec N^{-c}$ from Proposition \ref{prop1}, and the lower bound in \eqref{mupperlowerbounded}. The deterministic bound
 $\absb{G_{ij}^{(T)}} \leq N$ follows immediately from $\eta \geq N^{-1}$ by definition of a spectral domain.

In order to prove \eqref{rough bound on 1/G}, we use Schur's complement formula \eqref{SPSF} applied to $1/ G_{ii}^{(T)}$, where the expectation is estimated using (iii) of Definition \ref{def:Wigner} and $\absb{G_{ij}^{(T)}} \leq N$. This gives
\begin{equation*}
\E \absBB{\frac{1}{G_{ii}^{(T)}}}^p \;\prec\; N^{C_p}
\end{equation*}
for all $p \in \N$. Since $1 / \abs{G_{ii}^{(T)}} \prec 1$, \eqref{rough bound on 1/G} therefore follows from \eqref{eq:stoch_domination}.
\end{proof}

\begin{proof}[Proof of Proposition \ref{prop16}]
First we claim that, for any fixed $\ell \in \N$, we have
\begin{equation} \label{Xi estimate}
\absbb{Q_k \frac{1}{G^{(T)}_{kk}}} \;\prec\; \Phi
\end{equation}
uniformly for $T \subset \{1, \dots, \N\}$, $\abs{T} \leq \ell$, and $k \notin T$. To simplify notation, for the proof we set $T = \emptyset$; the proof for nonempty $T$ is the same.
From Schur's complement formula \eqref{SPSF} we get $\abs{Q_k  (G_{kk})^{-1}} \leq \abs{H_{kk}} + \abs{Z_k}$.
The first term is estimated by $\abs{H_{kk}} \prec N^{-1/2} \leq \Phi$. The second term is estimated exactly as in \eqref{Z_i_est_split}:
\begin{equation*}
\abs{Z_k}^2 \;\prec\; \ff{N^2}\sum_{k}^{(i)}|G_{kk}^{(i)}|^2 + \sum_{k,l}^{(i)}\ff{N^2}|G_{kl}^{(i)}|^2 \;\prec\; N^{-1} + \Phi^2 \;\prec\; \Phi^2\,,
\end{equation*}
where in the second step we used Lemma \ref{lemma:B2}. This concludes the proof of \eqref{Xi estimate}.

Abbreviate $X_k \deq Q_k (G_{kk})^{-1}$. We shall estimate $N^{-1} \sum_{k} X_k$ in probability by estimating its $p$-th moment by $\Phi^{2p}$, from which the claim will easily follow using Markov's inequality. Before embarking on the estimate for arbitrary $p$, we illustrate its idea by estimating the variance 
\begin{equation} \label{variance of average of X}
\E \absbb{N^{-1} \sum_k X_k}^2 \;=\; N^{-2} \sum_{k,l} \E X_{k} \ol X_{l} \;=\; N^{-2} \sum_{k}  \E X_{k} \ol X_{k} + N^{-2} \sum_{k \neq l} \E X_{k} \ol X_{l}\,.
\end{equation}
Using Lemma \ref{lem: exp prec}, we find that the first term on the right-hand side of \eqref{variance of average of X} is $O_\prec(N^{-1} \Phi^2) = O_\prec(\Phi^4)$, where we used the estimate \eqref{admissible Psi}. Let us therefore focus on the second term of \eqref{variance of average of X}. Using the fact that $k \neq l$, we apply \eqref{RI1} to $X_{k}$ and $X_{l}$ to get
\begin{equation} \label{variance calculation}
\E X_{k} \ol X_{l} \;=\; \E Q_{k} \pbb{\frac{1}{G_{k k}}} Q_{l} \ol{\pbb{\frac{1}{G_{l l}}}} \;=\;
\E Q_{k} \pbb{\frac{1}{G_{kk}^{(l)}} - \frac{G_{kl} G_{lk}}{G_{kk} G_{kk}^{(l)} G_{ll}}} Q_{l} \ol{\pbb{\frac{1}{G_{ll}^{(k)}} - \frac{G_{lk} G_{kl}}{G_{ll} G_{ll}^{(k)} G_{kk}}}}\,.
\end{equation}
We multiply out the parentheses on the right-hand side. The crucial observation is that if the random variable $Y$ is $H^{(i)}$-measurable then \begin{equation}\label{eq:indepprojzero}\E Q_i(X) Y = \E Q_i(XY) = 0.\end{equation} Hence out of the four terms obtained from the right-hand side of \eqref{variance calculation}, the only nonvanishing one is
\begin{equation*}
\E Q_{k} \pbb{\frac{G_{kl} G_{lk}}{G_{kk} G_{kk}^{(l)} G_{ll}}} Q_{l} \ol{\pbb{\frac{G_{lk} G_{kl}}{G_{ll} G_{ll}^{(k)} G_{kk}}}} \;\prec\; \Phi^4\,.
\end{equation*}
This concludes the proof of $\E \absb{N^{-1} \sum_{k}X_k}^2 \prec \Phi^4$.

After this pedagogical interlude we move on to the full proof. Fix some even integer $p$ and write
\begin{equation*}
\E \absbb{\frac{1}{N} \sum_k X_k}^p \;=\; \frac{1}{N^p} \sum_{k_1, \dots, k_p}  \E X_{k_1} \cdots X_{k_{p/2}} \ol X_{k_{p/2 + 1}} \cdots \ol X_{k_p}\,.
\end{equation*}
Next, we regroup the terms in the sum over $\f k \deq (k_1, \dots, k_p)$ according to the partition of $\{1, \dots, p\}$ generated by the indices $\f k$. To that end, let $\fra P_p$ denote the set of partitions of $\{1, \dots, p\}$, and $\cal P(\f k)$ the element of $\fra P_p$ defined by the equivalence relation $r \sim s$ if and only if $k_r = k_s$.
In short, we reorganize the summation according to coincidences among the indices $\f k$. Then we write
\begin{equation} \label{Xp expanded}
\E \absbb{\frac{1}{N}\sum_k X_k}^p \;=\; \sum_{P \in \fra P_p} \frac{1}{N^p} \sum_{\f k} \ind{\cal P(\f k) = P} V(\f k) \,,
\end{equation}
where we defined
\begin{equation*}
V(\f k) \;\deq\; \E X_{k_1} \cdots X_{k_{p/2}} \ol X_{k_{p/2 + 1}} \cdots \ol X_{k_p}\,.
\end{equation*}
Fix $\f k$ and set $P \deq \cal P(\f k)$ to be the partition induced by the coincidences in $\f k$.
For any $r\in  \{1, \dots, p\}$, we denote by  $[r]$ the block of $r$ in $P$. 
Let $L \equiv L(P) \deq \h{r \col [r] = \{r\}} \subset \{1, \dots, p\}$ be the set of ``lone'' labels.
We denote by $\f k_L \deq (k_r)_{r \in L}$ the summation indices associated with lone labels.

The Green function entry $G_{kk}$ depends strongly on the randomness in the $k$-column of $H$, but only weakly on the randomness in the other columns.
We conclude that if $r$ is a lone label then all factors $X_{k_s}$ with $s \neq r$
in $V(\f k)$ depend weakly on the randomness in the $k_r$-th column of $H$. Thus, the idea is to make all Green function entries inside the expectation of $V(\f k)$ as independent as possible of the randomness in the rows of $H$ indexed by $\f k_L$ (see Definition \ref{definition: P Q}),
using the identity \eqref{RI1}. To that end, we say that a Green function entry $G_{xy}^{(T)}$ with $x,y \notin T$ is \emph{maximally expanded} if $\f k_L \subset T \cup \{x,y\}$. 
The motivation behind this definition is that using \eqref{RI1} we cannot add upper indices from the set $\f k_L$ to a maximally expanded Green function entry. We shall apply
 \eqref{RI1} to all Green function entries in $V(\f k)$.
In this manner we generate a sum of monomials consisting of off-diagonal Green function entries
and inverses of diagonal Green function entries. We can now repeatedly apply
 \eqref{RI1} to each factor
 until either they are all maximally expanded or a sufficiently large 
 number of off-diagonal Green function entries has been generated. 
 The cap on the number of off-diagonal entries is introduced to 
ensure that this procedure terminates after a finite number of steps.

In order to define the precise algorithm, let $\cal A$ denote the set of monomials in the off-diagonal entries $G_{xy}^{(T)}$, with $T \subset \f k_L$, $x \neq y$, and $x,y \in \f k \setminus T$, as well as the inverse diagonal entries $1 / G_{xx}^{(T)}$, with $T \subset \f k_L$ and $x \in \f k \setminus T$. Starting from $V(\f k)$, the algorithm will recursively generate sums of monomials in $\cal A$.
Let $d(A)$ denote the number of off-diagonal entries in $A \in \cal A$. For $A \in \cal A$ we shall define $w_0(A), w_1(A) \in \cal A$ satisfying
\begin{equation} \label{properties of w}
A \;=\; w_0(A) + w_1(A) \,, \qquad d(w_0(A)) \;=\; d(A) \,, \qquad d(w_1(A)) \;\geq\; \max \hb{2, d(A) + 1}\,.
\end{equation}
The idea behind this splitting is to use \eqref{RI1} on one entry of $A$; the first term on the right-hand side of \eqref{RI1} gives rise to $w_0(A)$ and the second to $w_1(A)$. The precise definition of the algorithm applied to $A \in \cal A$ is as follows.

\begin{itemize}
\item[(1)]
If all factors of $A$ are maximally expanded or $d(A) \geq p + 1$ then 
stop the expansion of $A$. In other words, the algorithm cannot be applied to $A$ in the future.
\item[(2)]
Otherwise choose some (arbitrary) factor of $A$ that is not maximally expanded. If this entry is off-diagonal, $G^{(T)}_{xy}$, write
\begin{equation} \label{splitting off-diag}
G^{(T)}_{xy} \;=\; G_{xy}^{(T u)} + \frac{G_{xu}^{(T)} G_{uy}^{(T)}}{G_{uu}^{(T)}}
\end{equation}
for the smallest $u \in \f k_L \setminus (T \cup \{x,y\})$. If the chosen entry is diagonal, $1/ G^{(T)}_{xx}$, write
\begin{equation} \label{splitting diag}
\frac{1}{G_{xx}^{(T)}} \;=\; \frac{1}{G_{xx}^{(Tu)}} - \frac{G_{xu}^{(T)} G_{ux}^{(T)}}{G_{xx}^{(T)} G_{xx}^{(Tu)} G_{uu}^{(T)}}
\end{equation}
for the smallest $u \in \f k_L \setminus (T \cup \{x\})$. Then the splitting $A = w_0(A) + w_1(A)$ is defined by the splitting induced by \eqref{splitting off-diag} or \eqref{splitting diag}, in the sense that we replace the factor $G^{(T)}_{xy}$ or $1/G_{xx}^{(T)}$ in the monomial $A$ by the right-hand sides of  \eqref{splitting off-diag} or \eqref{splitting diag}.
\end{itemize}
(This algorithm contains some arbitrariness in the choice of the factor of $A$ to be expanded. It may be removed for instance by first fixing some ordering  of all Green function entries $G_{ij}^{(T)}$. Then in (2) we choose the first factor of $A$ that is not maximally expanded.) Note that \eqref{splitting off-diag} and \eqref{splitting diag} follow from \eqref{RI1}. It is clear that \eqref{properties of w} holds with the algorithm just defined.

We now apply this algorithm recursively to each entry $A^{r} \deq 1 / G_{k_r k_r}$ 
in the definition of $V(\f k)$. More precisely, we start with $A^r$ and define $A_{0}^r \deq w_0(A^r)$ and $A_{1}^r \deq w_1(A^r)$. In the second step of the algorithm we define four monomials
$$
 A_{00}^r \;\deq\; w_0(A_0^r)\,, \qquad A_{01}^r \;\deq\; w_0(A_1^r)\,, \qquad  
A_{10}^r \;\deq\; w_1(A_0^r)\,, \qquad A_{11}^r \;\deq\; w_1(A_1^r)\,,
$$
and so on, at each iteration performing the steps (1) and (2) on each new monomial independently of the others. 
Note that the lower indices are binary sequences that describe the recursive application of the
operations $w_0$ and $w_1$.
In this manner we generate a binary tree whose vertices are given by finite binary strings $\sigma$. 
The associated monomials satisfy $A_{\sigma i}^r \deq w_i(A_\sigma^r)$ for $i = 0,1$,
where $\sigma i$ denotes the binary string obtained by appending $i$ to the right end of $\sigma$. See Figure \ref{fig: tree} for an illustration of the tree.

\begin{figure}[ht!]
\begin{center}
\includegraphics{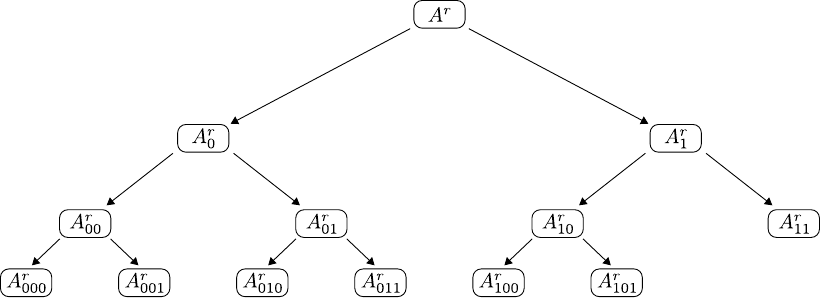}
\end{center}
\caption{The binary tree generated by applying the algorithm (1)--(2) to a monomial $A^r$. Each vertex of the tree is indexed by a binary string $\sigma$, and encodes a monomial $A^r_\sigma$. An arrow towards the left represents the action of $w_0$ and an arrow towards the right the action of $w_1$. The monomial $A_{11}^r$ satisfies the assumptions of step (1), and hence its expansion is stopped, so that the tree vertex $11$ has no children. \label{fig: tree}}
\end{figure}

We stop the recursion of a tree vertex whenever the associated monomial satisfies the stopping rule of step (1). In other words, the set of leaves of the tree is the set of binary strings $\sigma$ such that either all factors of $A^r_\sigma$ are maximally expanded or $d(A^r_\sigma) \geq p + 1$. We claim that the resulting binary tree is finite, i.e.\ that the algorithm always reaches step (1) after a finite number of iterations. Indeed, by the stopping rule in (1), we have $d(A^r_\sigma) \leq p + 1$ for any vertex $\sigma$ of the tree. Since each application of $w_1$ increases $d(\cdot)$ by at least one, and in the first step (i.e.\ when applied to $A^r$) by two,
we conclude that the number of ones in any $\sigma$ is at most $p$. Since each application of $w_1$ increases the number of Green function entries by at most four, and the application of $w_0$ does not change this number,
we find that the number of Green function entries in $A^r_\sigma$ is bounded by $4 p + 1$. Hence the maximal number of
 upper indices in $A^r_\sigma$ for any tree vertex $\sigma$ is $(4p+1)p$.
Since each application of $w_0$ increases the total number of upper indices by one, we find that $\sigma$ contains at most $(4p+1)p$  zeros. We conclude that the maximal length of the string $\sigma$ (i.e.\ the depth of the tree) is at most $(4p +1)p+ p = 4p^2+2p$. 
A string $\sigma$ encoding a tree vertex contains at most $p$ ones. Denoting by $k$ the number of ones in a string encoding a leaf of the tree, we find that the number of leaves is bounded by $\sum_{k = 0}^p \binom{4p^2 + 2p}{k} 
\leq (Cp^2)^{p}$.
Therefore, denoting by $\cal L_r$ the set of leaves of the binary tree generated from $A^r$, 
we have $\abs{\cal L_r} \leq (Cp^2)^p$.

By definition of the tree and $w_0$ and $w_1$, we have the decomposition
\begin{equation} \label{X split}
X_{k_r} \;=\; Q_{k_r} \sum_{\sigma \in \cal L_r} A_\sigma^r\,.
\end{equation}
Moreover, each monomial $A_\sigma^r$ for $\sigma \in \cal L_r$ either consists entirely of maximally expanded Green function entries or satisfies $d(A_\sigma^r) = p + 1$. (This is an immediate consequence of the stopping rule in (1)).

Next, we observe that for any string $\sigma$ we have
\begin{equation} \label{A sigma size}
A_\sigma^k \;=\; O_\prec \pb{\Phi^{b(\sigma) + 1}}\,,
\end{equation}
where $b(\sigma)$ is the number ones in the string $\sigma$. Indeed, if $b(\sigma) = 0$ then this follows from \eqref{Xi estimate}; if $b(\sigma) \geq 1$ this follows from the last statement in \eqref{properties of w} and \eqref{basic bounds on GT}.

Using \eqref{Xp expanded} and \eqref{X split} we have the representation
\begin{equation} \label{V in terms of leaves}
V(\f k) \;=\; \sum_{\sigma_1 \in \cal L_1} \cdots \sum_{\sigma_p \in \cal L_p} \E \pb{Q_{k_1} A_{\sigma_1}^1} \cdots \pb{Q_{k_p} \ol{A_{\sigma_p}^p}}\,.
\end{equation}

We now claim that any nonzero term on the right-hand side of \eqref{V in terms of leaves} satisfies
\begin{equation} \label{estimate of integrand}
\pb{Q_{k_1} A_{\sigma_1}^1} \cdots \pb{Q_{k_p} \ol{A_{\sigma_p}^p}} \;=\; O_\prec \pb{\Phi^{p + \abs{L}}}\,.
\end{equation}

\begin{proof}[Proof of \eqref{estimate of integrand}]
Before embarking on the proof, we explain its idea.
By \eqref{A sigma size}, the naive size of the left-hand side of \eqref{estimate of integrand} is $\Phi^p$. The key observation is that each
lone label $s\in L$ yields one extra factor $\Phi$ to the estimate. This is because by \eqre{eq:indepprojzero},  the expectation 
in \eqref{V in terms of leaves} would vanish if all other factors  $\pb{Q_{k_r} A_{\sigma_r}^r}$, $r\ne s$,
were $H^{(k_s)}$-measurable. The expansion of the binary tree makes this dependence explicit
by exhibiting $k_s$ as a lower index. But this requires performing an operation $w_1$
with the choice $u=k_s$ in \eqref{splitting off-diag} or \eqref{splitting diag}.
However, $w_1$ increases the number of off-diagonal element by at least one. 
In other words, every index associated with a lone label must have a ``partner'' index in a different Green function entry which arose by application of $w_1$.
Such a partner index may only be obtained through the creation of at least one off-diagonal Green function entry. The actual proof
below shows that this effect applies \emph{cumulatively} for all lone labels.

In order to prove \eqref{estimate of integrand}, we consider two cases. Consider first the case where for some $r = 1, \dots, p$ the monomial $A_{\sigma_r}^r$ on the left-hand side of \eqref{estimate of integrand} is not maximally expanded. Then $d(A_{\sigma_r}^r) = p + 1$, so that \eqref{basic bounds on GT} yields $A_{\sigma_r}^r \prec \Phi^{p + 1}$. Therefore the observation that $A_{\sigma_s}^s \prec \Phi$ for all $s \neq r$, together with \eqref{stoch_dom_P} implies that the left-hand side of \eqref{estimate of integrand} is $O_\prec\pb{\Phi^{2p}}$. Since $\abs{L} \leq p$, \eqref{estimate of integrand} follows.

Consider now the case where $A_{\sigma_r}^r$ on the left-hand side of \eqref{estimate of integrand} is maximally expanded for all $r = 1, \dots, p$. The key observation is the following claim about the left-hand side of \eqref{estimate of integrand} with a nonzero expectation.
\begin{itemize}
\item[$(*)$]
For each $s \in L$ there exists $r = \tau(s) \in \{1, \dots, p\} \setminus \{s\}$ such that the monomial $A_{\sigma_r}^r$ contains a Green function entry with lower index $k_s$.
\end{itemize}
In other words,  after expansion, the lone label $s$ has a ``partner'' label $r = \tau(s)$, such that the index
$k_s$ appears also in the expansion of $A^r$ (note that there may be several such partner labels $r$). 
To prove $(*)$, suppose by contradiction that there exists an $s \in L$ such that for all $r \in \{1, \dots, p\} \setminus \{s\}$ the lower index $k_s$ does not appear in the monomial $A_{\sigma_r}^r$. To simplify notation, we assume that $s = 1$. Then, for all $r = 2, \dots, p$, since $A_{\sigma_r}^r$ is maximally expanded, we find that $A_{\sigma_r}^r$ is $H^{(k_1)}$-measurable.
Therefore we have
\begin{equation*}
\E \pb{Q_{k_1} A_{\sigma_1}^1} \pb{Q_{k_2} A_{\sigma_2}^2} \cdots \pb{Q_{k_p} \ol{A_{\sigma_p}^p}} \;=\; \E Q_{k_1} \pB{A_{\sigma_1}^1 \pb{Q_{k_2} A_{\sigma_2}^2} \cdots \pb{Q_{k_p} \ol{A_{\sigma_p}^p}}} \;=\; 0\,,
\end{equation*}
where in the last step we used \eqre{eq:indepprojzero}.
This concludes the proof of $(*)$.

For $r \in \{1, \dots, p\}$ we define $l(r) \deq \sum_{s \in L} \ind{\tau(s) = r}$, the number of times that the label $r$ was chosen as a partner
to some lone label $s$. We now claim that 
\begin{equation} \label{key linking estimate}
A_{\sigma_r}^r \;=\; O_\prec \pb{\Phi^{1 + l(r)}}\,.
\end{equation}
To prove \eqref{key linking estimate}, fix $r \in \{1, \dots, p\}$. By definition, for each $s \in \tau^{-1}(\{r\})$ the index $k_s$ appears as a lower index in the monomial $A_{\sigma_r}^r$. Since $s \in L$ is by definition a lone label
and $s \neq r$, we know that $k_s$ does not appear as an index in $A^r$. By definition of the monomials associated with the tree vertex $\sigma_r$, it follows that $b(\sigma_r)$, the number of ones in $\sigma_r$, is at least $\absb{\tau^{-1}(\{r\})} = l(r)$
since each application of $w_1$ adds precisely one new (lower) index.
Note that in this step it is crucial that $s \in \tau^{-1}(\{r\})$ was a lone label. Recalling \eqref{A sigma size}, we therefore get \eqref{key linking estimate}.

Using \eqref{key linking estimate} and Lemma \ref{lem: exp prec} we find
\begin{equation*}
\absB{\pb{Q_{k_1} A_{\sigma_1}^1} \cdots \pb{Q_{k_p} \ol{A_{\sigma_p}^p}}} \;\prec\; \prod_{r = 1}^p \Phi^{1 + l(r)} \;=\; \Phi^{p + \abs{L}}\,.
\end{equation*}
This concludes the proof of \eqref{estimate of integrand}.
\end{proof}

Summing over the binary trees in \eqref{V in terms of leaves} and using Lemma \ref{lem: exp prec}, we get from \eqref{estimate of integrand}
\begin{equation} \label{bound for Vk}
V(\f k) \;=\; O_\prec\pb{\Phi^{p + \abs{L}}}\,.
\end{equation}
We now return to the sum \eqref{Xp expanded}. We perform the summation by first fixing $P \in \fra P_p$, with associated lone labels $L = L(P)$. We find
\begin{equation*}
\absbb{\frac{1}{N^p} \sum_{\f k} \ind{\cal P(\f k) = P}} \;\leq\;   (N^{-1})^{p-|P|} \;\leq\; (N^{-1/2})^{p - \abs{L}}\,;
\end{equation*}
in the first step we used that the summation is performed over $\abs{P}$ free indices, the remaining $p - \abs{P}$ being estimated by $N^{-1}$; in the second step we used that  
each block of $P$ that is not contained in $L$ consists of at least two labels, so that $p - \abs{P} \geq (p-\abs{L})/2$.
From \eqref{Xp expanded} and \eqref{bound for Vk} we get
\begin{equation*}
\E \absbb{\frac{1}{N} \sum_k X_k}^p \;\prec\; \sum_{P \in \fra P_p} (N^{-1/2})^{p - \abs{L(P)}} \, \Phi^{p + \abs{L(P)}}  \;\leq\; C_p \Phi^{2p}\,,
\end{equation*}
where in the last step we used the lower bound from \eqref{admissible Psi} and estimated the summation over $\fra P_p$ with a constant $C_p$ (which is bounded by $(Cp^2)^p$). Summarizing, we have proved that
\begin{equation} \label{final moment estimate}
\E \absbb{\frac{1}{N} \sum_k X_k}^p \;\prec\; \Phi^{2p}
\end{equation}
for any $p \in 2 \N$.

We conclude the proof of Proposition \ref{prop16} with a simple application of 
Markov's inequality. Fix $\epsilon > 0$ and $D > 0$. Using \eqref{final moment estimate} and Markov's inequality we find
\begin{equation*}
\P \pbb{\absbb{\frac{1}{N} \sum_k X_k} > N^\epsilon \Phi^2} \;\leq\; N \, N^{-\epsilon p}
\end{equation*}
for large enough $N \geq N_0(\epsilon, p)$. Choosing $p \geq \epsilon^{-1} (1 + D)$ concludes
 the proof of Proposition \ref{prop16}.
\end{proof}

We conclude this section with an alternative proof of Proposition \ref{prop16}. While the underlying argument remains similar, the following proof makes use of an additional decomposition of the space of random variables, which avoids the use of the stopping rule from Step (1) in the above proof of Proposition \ref{prop16}. This decomposition may be regarded as an abstract reformulation of the stopping rule.

\begin{proof}[Alternative proof of Proposition \ref{prop16}]
As before, we set $X_k \deq Q_k (G_{kk})^{-1}$. The decomposition is defined using the operations $P_i$ and $Q_i$, introduced in Definition \ref{definition: P Q}. It is immediate that $P_i$ and $Q_i$ are projections, that $P_i + Q_i = 1$, and that all of these projections commute with each other (by Fubini's theorem). For a set $A \subset \{1, \dots, N\}$ we use the notations $P_A \deq \prod_{i \in A} P_i$ and $Q_A \deq \prod_{i \in A} Q_i$.

Let $p$ be even and introduce the shorthand $\wt X_{k_s} \deq X_{k_s}$ for $s \leq p/2$ and $\wt X_{k_s} \deq \ol X_{k_s}$ for $s > p/2$. Then we get
\begin{equation*}
\E \absbb{\frac{1}{N} \sum_k X_k}^{p}
\;=\; \frac1{N^p}  \sum_{k_1, \dots, k_p}  \E \prod_{s = 1}^p \wt X_{k_s} \;=\;
\frac1{N^p}  \sum_{k_1, \dots, k_p}  \E \prod_{s = 1}^p \pBB{\prod_{r = 1}^p (P_{k_r} + Q_{k_r})\wt X_{k_s}}\,.
\end{equation*}
Introducing the notations $\f k = (k_1, \dots, k_p)$ and $[\f k] = \{k_1, \dots, k_p\}$, we therefore get by multiplying out the parentheses
\begin{equation} \label{Zp expanded}
\E \absbb{\frac{1}{N} \sum_k X_k}^{p} \;=\;
\frac1{N^p}  \sum_{\f k} \sum_{A_1, \dots, A_p \subset [\f k]}  \E \prod_{s = 1}^p \pb{P_{A_s^c} Q_{A_s} \wt X_{k_s}}\,.
\end{equation}

Next, by definition of $\wt X_{k_s}$, we have that $\wt X_{k_s} = Q_{k_s} \wt X_{k_s}$, which implies that $P_{A^c_s} \wt X_{k_s} = 0$ if $k_s \notin A_s$. Hence may restrict the summation to $A_s$ satisfying
\begin{equation} \label{C cond 1}
k_s \;\in\; A_s
\end{equation}
for all $s$. Moreover, we claim that the right-hand side of \eqref{Zp expanded} vanishes unless
\begin{equation} \label{C cond 2}
k_s \;\in\; \bigcup_{q \neq s} A_{q}
\end{equation}
for all $s$. Indeed, suppose that $k_s \in \bigcap_{q \neq s} A_{q}^c$ for some $s$, say $s = 1$. In this case, for each $s = 2, \dots, p$, the factor $P_{A_s^c} Q_{A_s} \wt X_{k_s}$ is $H^{(k_1)}$-measurable. Thus we get
\begin{multline*}
\E \prod_{s = 1}^p \pb{P_{A_s^c} Q_{A_s} \wt X_{k_s}} \;=\; \E \pb{P_{A_1^c} Q_{A_1} Q_{k_1} \wt X_{k_1}} \prod_{s = 2}^p \pb{P_{A_s^c} Q_{A_s} \wt X_{k_s}}
\\
=\; \E Q_{k_1} \pbb{\pb{P_{A_1^c} Q_{A_1} \wt X_{k_1}} \prod_{s = 2}^p \pb{P_{A_s^c} Q_{A_s} \wt X_{k_s}}} \;=\; 0\,,
\end{multline*}
where in the last step we used that $\E Q_i(X) = 0$ for any $i$ and any random variable $X$.

We conclude that the summation on the right-hand side of \eqref{Zp expanded} is restricted to indices satisfying \eqref{C cond 1} and \eqref{C cond 2}. Under these two conditions we have
\begin{equation} \label{sum size A}
\sum_{s = 1}^p \abs{A_s} \;\geq\; 2 \, \abs{[\f k]}\,,
\end{equation}
since each index $k_s$ must belong to at least two different sets $A_q$: to $A_s$ (by \eqref{C cond 1}) as well as to some $A_q$ with $q \neq s$ (by \eqref{C cond 2}).

Next, we claim that for $k \in A$ we have
\begin{equation} \label{claim on size of QA}
\abs{Q_A X_k } \;\prec\; \Phi^{\abs{A}}\,.
\end{equation}
Before proving \eqref{claim on size of QA}, we show it may be used to complete the proof. Using \eqref{Zp expanded}, \eqref{claim on size of QA}, and Lemma \ref{lem: exp prec}, we find
\begin{multline*}
\E \absbb{\frac{1}{N} \sum_k X_k}^{p} \;\prec\; C_p \frac{1}{N^p} \sum_{\f k} \Phi^{2 \abs{[k]}} \;=\; C_p \sum_{u = 1}^p \Phi^{2u} \frac{1}{N^p} \sum_{\f k} \ind{\abs{[\f k]} = u}
\\
\leq\; C_p \sum_{u = 1}^p \Phi^{2u} N^{u - p} \;\leq\; C_p (\Phi + N^{-1/2})^{2p} \;\leq\; C_p \Phi^{2p}\,,
\end{multline*}
where in the first step we estimated the summation over the sets $A_1, \dots, A_p$ by a combinatorial factor $C_p$ depending on $p$, in the fourth step we used the elementary inequality $a^n b^m \leq (a + b)^{n + m}$ for positive $a,b$, and in the last step we used \eqref{admissible Psi}. Thus we have proved \eqref{final moment estimate}, from which the claim follows exactly as in the first proof of Proposition \ref{prop16}.

What remains is the proof of \eqref{claim on size of QA}. The case $\abs{A} = 1$ (corresponding to $A = \{k\}$) follows from \eqref{Xi estimate}, exactly as in the first proof of Proposition \ref{prop16}. To simplify notation, for the case $\abs{A} \geq 2$ we assume that $k = 1$ and $A = \{1, 2, \dots, t\}$ with $t \geq 2$. It suffices to prove that
\begin{equation} \label{claim for QA 1/G}
\absbb{Q_t \cdots Q_2 \frac{1}{G_{11}}} \;\prec\; \Phi^{t}\,.
\end{equation}
We start by writing, using \eqref{RI1},
\begin{equation*}
Q_2 \frac{1}{G_{11}} \;=\; Q_2 \frac{1}{G_{11}^{(2)}} - Q_2 \frac{G_{12} G_{21}}{G_{11} G_{11}^{(2)} G_{22}} \;=\; Q_2 \frac{G_{12} G_{21}}{G_{11} G_{11}^{(2)} G_{22}}\,,
\end{equation*}
where the first term vanishes since $G_{11}^{(2)}$ is $H^{(2)}$-measurable. We now consider
\begin{equation*}
Q_3 Q_2 \frac{1}{G_{11}} \;=\; Q_2 Q_3 \frac{G_{12} G_{21}}{G_{11} G_{11}^{(2)} G_{22}}\,,
\end{equation*}
and apply \eqref{RI1} with $k = 3$ to each Green function entry on the right-hand side, and multiply everything out. The result is a sum of fractions of entries of $G$, whereby all entries in the numerator are off-diagonal and all entries in the denominator are diagonal. The leading order term vanishes,
\begin{equation*}
Q_2 Q_3 \frac{G_{12}^{(3)} G_{21}^{(3)}}{G_{11}^{(3)} G_{11}^{(23)} G_{22}^{(3)}} \;=\; 0\,,
\end{equation*}
so that the surviving terms have at least three (off-diagonal) Green function entries in the numerator.
We may now continue in this manner; at each step the number of (off-diagonal) Green function entries in the numerator increases by at least one.

More formally, we obtain a sequence $A_2, A_3, \dots, A_t$, where $A_2 \deq Q_2 \frac{G_{12} G_{21}}{G_{11} G_{11}^{(2)} G_{22}}$ and $A_{i}$ is obtained by applying \eqref{RI1} with $k = i$ to each entry of $Q_i A_{i - 1}$, and keeping only the nonvanishing terms. The following properties are easy to check by induction.
\begin{enumerate}
\item
$A_i = Q_i A_{i - 1}$.
\item
$A_i$ consists of the projection $Q_2 \cdots Q_i$ applied to a sum of fractions such that all entries in the numerator are off-diagonal and all entries in the denominator are diagonal.
\item
The number of (off-diagonal) entries in the numerator of each term of $A_i$ is at least $i$.
\end{enumerate}
By Lemma \ref{lem: exp prec} combined with (ii) and (iii) we conclude that $\abs{A_i} \prec \Phi^i$. From (i) we therefore get
\begin{equation*}
Q_t \cdots Q_2 \frac{1}{G_{11}} \;=\; A_t \;=\; O_\prec(\Phi^t)\,.
\end{equation*}
This is \eqref{claim for QA 1/G}. Hence the proof is complete.
\end{proof}

\section{Semicircle law on small scales: proof of Theorem \ref{thm:llsc}}\la{sec:local_law_small_scales}
We define the signed measure $\hat \mu$ and its Stieltjes transform $\hat s$ through
\begin{equation}
\hat \mu \;\deq\; \mu - \varrho \,, \qquad \hat s(z) \;\deq\; \int \frac{\hat \mu(\dd x)}{x - z} \;=\; s(z) - m(z)\,.
\end{equation}
The basic idea behind the proof of Theorem \ref{thm:llsc} is to estimate $\hat \mu(I)$ using the Helffer-Sj\"ostrand formula from Appendix \ref{sec:HS} in terms of its Stieltjes transfrom, $\hat s$, which is controlled by Theorem \ref{Th3}.

To that end, fix $\epsilon > 0$ and define $\eta \deq N^{-1 + \epsilon}$. Then for any interval $I \subset [-3,3]$ we choose a smoothed indicator function $f \equiv f_{I,\eta} \in \cal C^\infty_c(\R;[0,1])$ satisfying $f(x) = 1$ for $x \in I$, $f(x) = 0$ for $\dist(x,I) \geq \eta$, $\norm{f'}_\infty \leq C \eta^{-1}$, and $\norm{f''}_\infty \leq C \eta^{-2}$. Note that the supports of $f'$ and $f''$ have Lebesgue measure at most $2 \eta$, which we tacitly use from now on. Next, choose a smooth, even, cutoff function $\chi \in \cal C^\infty_c(\R;[0,1])$ satisfying $\chi(y) = 1$ for $\abs{y} \leq 1$, $\chi(y) = 0$ for $\abs{y} \geq 2$, and $\norm{\chi'}_\infty \leq C$.

Now using the Helffer-Sj\"ostrand formula from Proposition \ref{prop:HS} with $n = 1$, we get
\begin{equation*}
\int f(\lambda) \, \hat \mu(\dd \lambda) \;=\; \frac{1}{2\pi} \int \dd x \int \dd y \, (\partial_x + \ii \partial_y) \qb{\pb{f(x) + \ii y f'(x)} \chi(y)} \, \hat s(x + \ii y)\,.
\end{equation*}
Since the left-hand side is real, we obtain
\begin{align}
\int f(\lambda) \, \hat \mu(\dd \lambda) &\;=\; -\frac{1}{2\pi} \int \dd x \int_{\abs{y} \leq \eta} \dd y \, f''(x) \chi(y) y \im \hat s(x + \ii y)
\label{HS_error1} \\
&\qquad
-\frac{1}{2\pi} \int \dd x \int_{\abs{y} > \eta} \dd y \, f''(x) \chi(y) y \im \hat s(x + \ii y)
\label{HS_error2} \\ \label{HS_error3}
&\qquad
+ \frac{\ii}{2\pi} \int \dd x \int \dd y \, \pb{f(x) + \ii y f'(x)} \chi'(y) \hat s(x + \ii y)\,.
\end{align}
Note the crucial cancellation of the terms proportional to $f'(x) \chi(y)$.

We shall estimate the three error terms \eqref{HS_error1}--\eqref{HS_error3} by using the estimate \eqref{121414h53} from the local semicircle law, Theorem \ref{Th3}. In fact, \eqref{121414h53} only holds at each $z \in \f S$ individually, and in order to apply it to the integrals in \eqref{HS_error1}--\eqref{HS_error3} we need a \emph{simultaneous} estimate for all $z \in \f S$. This extension is the content of the following simple lemma, which follows from Theorem \ref{Th3} using the argument in Remark \ref{rem:Lipschitz} and the trivial identity $\hat s(x - \ii y) = \ol{\hat s(x + \ii y)}$.

\begin{lem} \label{lem:s_hat_est}
For any fixed $\epsilon > 0$, we have with high probability $\abs{\hat s(x + \ii y)} \leq N^\epsilon / (N \abs{y})$ for $\abs{x} \leq \epsilon^{-1}$ and $\abs{y} \in [\eta, \epsilon^{-1}]$.
\end{lem}

Using Lemma \ref{lem:s_hat_est}, we may now estimate \eqref{HS_error1}--\eqref{HS_error3}. First, using that $\chi'$ is supported in $[-2,2] \setminus (-1,1)$, we easily find
\begin{equation} \label{HS_error3f}
\abs{\text{\eqref{HS_error3}}} \;\leq\; C \eta
\end{equation}
with high probability.

Next, we estimate \eqref{HS_error1}. Since $y$ can be arbitrarily small, Lemma \ref{lem:s_hat_est} does not apply. However, this can be easily remedied by the following monotonicity argument. For simplicity, we only deal with $y > 0$; the estimate for $y < 0$ is analogous. It is easy to see that the map $y \mapsto y \im s(x + \ii y)$ is nondecreasing for all $x$ (this is a general property of the Stieltjes transform of any positive measure). Hence, for $y \in (0,\eta)$ we have
\begin{equation*}
y \im \hat s(x + \ii y) \;\leq\; y \im s(x + \ii y) \;\leq\; \eta \im s(x + \ii \eta) \;\leq\; \eta (N^\epsilon / (N \eta) + C) \;\leq\; C \eta
\end{equation*}
with high probability, where in the third step we used Lemma \ref{lem:s_hat_est}. We conclude that
\begin{equation} \label{HS_error1f}
\abs{\text{\eqref{HS_error1}}} \;\leq\; \eta^{-1} \int_{\abs{y} \leq \eta} \dd y \, C \eta \;=\; C \eta
\end{equation}
with high probability.

What remains is \eqref{HS_error2}. Unlike in \eqref{HS_error1}, the second derivative on $f$ is not affordable (since it results in a factor $\eta^{-1}$ after integration over $x$). The solution is to integrate by parts in $x$ and then in $y$. Combined with $\hat s(x - \ii y) = \ol{\hat s(x + \ii y)}$, we therefore get
\begin{multline}\label{HS_error12}
\abs{\text{\eqref{HS_error1}}} \;\leq\; \int \dd x \int_\eta^\infty \dd y \, \absb{f'(x) \chi(y) \hat s(x + \ii y)}
\\
+
\int \dd x \int_\eta^\infty \dd y \, \absb{f'(x) y \chi'(y) \hat s(x + \ii y)}
+
\int \dd x \, \absb{f'(x) \eta \hat s(x + \ii \eta)}\,.
\end{multline}
As above, using Lemma \ref{lem:s_hat_est} we easily find that the second line of \eqref{HS_error12} is bounded by $C \eta$ with high probability. Moreover, the first term on the right-hand side of \eqref{HS_error12} is bounded by
\begin{equation*}
C \int_\eta^2 \dd y \, \frac{N^\epsilon}{N y} \;\leq\; C \eta \log N\,.
\end{equation*}
Recalling \eqref{HS_error3f} and \eqref{HS_error1f}, we have therefore proved that
\begin{equation} \label{HS_est_main}
\absbb{\int f(\lambda) \, \hat \mu(\dd \lambda)} \;\leq\; C N^{-1 + 2 \epsilon}
\end{equation}
with high probability.

In order to conclude the proof of Theorem \ref{thm:llsc}, we have to return from the smoothed indicator function $f$ to the sharp indicator function of $I$. To that end, we note that if $I \subset [-3,3]$ we have the upper bound
\begin{equation*}
\mu(I) \;\leq\; \int f_{I,\eta}(\lambda) \, \mu(\dd \lambda) \;\leq\; \int f_{I,\eta}(\lambda) \, \varrho(\dd \lambda) + O(N^{-1 + 2 \epsilon}) \;\leq\;
\varrho(I) + O(N^{-1 + 2 \epsilon})
\end{equation*}
with high probability, where in the second step we used \eqref{HS_est_main}, and in the third that the density of $\varrho$ is bounded (see \eqref{defscl}). Conversely, denoting by $I' \deq \h{x \in \R \col \dist(x, I^c) \geq \eta}$, we get the lower bound
\begin{equation*}
\mu(I) \;\geq\; \int f_{I',\eta}(\lambda) \, \mu(\dd \lambda) \;\geq\; \int f_{I',\eta}(\lambda) \, \varrho(\dd \lambda) + O(N^{-1 + 2 \epsilon}) \;\geq\;
\varrho(I) + O(N^{-1 + 2 \epsilon})
\end{equation*}
with high probability.
Since $\epsilon > 0$ was arbitrary, we conclude that for any $I \subset [-3,3]$ we have $\hat \mu(I) = O_\prec(N^{-1})$.

Finally, in order to extend the result to arbitrary $I \subset \R$, we note that we have proved that $\hat \mu([-2,2]) = O_\prec(N^{-1})$. Since $\mu$ is a probability measure and $\varrho([-2,2]) = 1$, we therefore deduce that $\mu(\R \setminus [-2,2]) \prec N^{-1}$. The claim of Theorem \ref{thm:llsc} now easily follows for arbitrary $I \subset \R$ by splitting
\begin{equation*}
\mu(I) \;=\; \mu(I \cap [-2,2]) + \mu(I \cap (\R \setminus [-2,2])) \;=\; \mu(I \cap [-2,2]) + O_\prec(N^{-1}) \;=\; \varrho(I) + O_\prec(N^{-1})\,,
\end{equation*}
where in the last step we used $\hat \mu(I) = O_\prec(N^{-1})$ for $I \subset [-2,2]$.
(Note that this estimate is in fact simultaneous for all $I \subset \R$, i.e.\ for any fixed $\epsilon > 0$ we have with high probability $\hat \mu(I) = O(N^{-1 + \epsilon})$ for all $I \subset \R$.) This concludes the proof of Theorem \ref{thm:llsc}.

\section{Eigenvalue rigidity: proof of Theorem \ref{thm:rig}}\la{sec:rig}
The first key input of the proof is the following estimate of the norm of $H$. (Note that the exponent of $N^{-2/3}$ is optimal.)

\begin{proposition} \label{prop:bound_H}
For a Wigner matrix $H$ we have $\norm{H} \leq 2 + O_\prec(N^{-2/3})$.
\end{proposition}

The main tool in the proof of Proposition \ref{prop:bound_H} is the following improved version of the local semicircle law \eqref{121414h53} outside of the spectrum.

\begin{lem}[Improved local semicircle law outside of the spectrum] \label{lem:improved_ll}
For a Wigner matrix $H$ we have
\begin{equation*}
s(z) \;=\; m(z) + O_\prec\pBB{\frac{1}{\sqrt{\kappa + \eta}} \pbb{\frac{\im m}{N \eta} + \frac{1}{(N \eta)^2}}}
\end{equation*}
uniformly for $z \in \f S$.
\end{lem}
\begin{proof}
We use \eqref{Theta_self_impr} with the a priori bound $\Theta \prec (N \eta)^{-1}$ from \eqref{121414h53}, which yields
\begin{equation*}
\Theta \;\prec\; \frac{1}{\sqrt{\kappa + \eta}} \, \Phi_1^2\,,
\end{equation*}
where $\Phi_\sigma$ was defined in \eqref{def_Phi_s}. The claim follows.
\end{proof}
Note that the bound from Lemma \ref{lem:improved_ll} is better than \eqref{121414h53} when $E$ is sufficiently far outside of the spectrum, for large enough $\kappa$ and small enough $\eta$, since in that case $\im m$ is small by \eqref{Immmupperlowerbounded}.

\begin{proof}[Proof of Proposition \ref{prop:bound_H}]
For definiteness, we prove that the largest eigenvalue $\lambda_1$ of $H$ satisfies $\lambda_1 \leq 2 + O_\prec(N^{-2/3})$; the smallest eigenvalue $\lambda_N$ is handled similarly. From the F\"uredi-Koml\'os argument in Theorem \ref{Th:Furedi_Komlos}, we know that $\lambda_1 \leq 3$ with high probability. It therefore remains to show that, for any fixed $\epsilon > 0$, there is no eigenvalue of $H$ in
\begin{equation} \label{def_I_bad}
I \;\deq\; \qb{2 + N^{-2/3 + 4 \epsilon}, 3}
\end{equation}
with high probability. We do this using Lemma \ref{lem:improved_ll}. The idea of the proof is to choose, for each $E \in I$, a scale $\eta(E)$ such that $\im s(E + \ii \eta(E)) \leq \frac{N^{-\epsilon}}{N \eta(E)}$ with high probability; such an estimate will be possible thanks to the improved bound from Lemma \ref{lem:improved_ll}. (Note that the estimate \eqref{121414h53} can never produce such a bound.) We shall then conclude the argument by noting that such an upper bound is strong enough to rule out the existence of an eigenvalue at $E$.

By an argument analogous to Remark \ref{rem:Lipschitz}, we find from Lemma \ref{lem:improved_ll} and \eqref{Immmupperlowerbounded} that, with high probability,
\begin{equation} \label{s_m_outside}
\abs{s(z) - m(z)} \;\leq\; N^\epsilon \pbb{\frac{\eta}{\kappa} \frac{1}{N \eta} + \frac{1}{\sqrt{\kappa}} \frac{1}{(N \eta)^2}}
\end{equation}
for all $z \in \f S$ satisfying $E \geq 2$. For each $E \in I$, we define
\begin{equation}
z(E) \;\deq\; E + \ii \eta(E) \,, \qquad \eta(E) \;\deq\; N^{-1/2 - \epsilon} \kappa(E)^{1/4}\,.
\end{equation}
Using \eqref{Immmupperlowerbounded}, we find for $E \in I$ that
\begin{equation} \label{no_ev1}
\im m(z(E)) \;\leq\; \frac{\eta(E)}{\sqrt{\kappa(E)}} \;\leq\; \frac{N^{-\epsilon}}{N \eta(E)}\,.
\end{equation}
Moreover, from \eqref{s_m_outside} we find that, with high probability,
\begin{equation} \label{no_ev2}
\absb{s(z(E)) - m(z(E))} \;\leq\; \frac{2 N^{-\epsilon}}{N \eta(E)}
\end{equation}
for all $E \in I$. From \eqref{no_ev1} and \eqref{no_ev2} we conclude that, with high probability,
\begin{equation} \label{upper_bd_ims}
\im s(z(E)) \;\leq\; \frac{3 N^{-\epsilon}}{N \eta(E)}
\end{equation}
for all $E \in I$.

Now suppose that there is an eigenvalue, say $\lambda_i$, of $H$ in $I$. Then we find
\begin{equation} \label{lower_bd_ims}
\im s(z(\lambda_i)) \;=\; \frac{1}{N} \sum_j \frac{\eta}{(\lambda_j - \lambda_i)^2 + \eta(\lambda_i)^2} \;\geq\; \frac{1}{N \eta(\lambda_i)}\,.
\end{equation}
Since \eqref{upper_bd_ims} and \eqref{lower_bd_ims} with $E = \lambda_i \in I$ are mutually exclusive, we conclude that, with high probability, there is no eigenvalue of $H$ in $I$. Since $\epsilon > 0$ was arbitrary, the claim follows.
\end{proof}

Armed with Proposition \ref{prop:bound_H}, we may now complete the proof of Theorem \ref{thm:rig}. We only consider $i \leq N/2$; the indices $i > N/2$ are dealt with analogously. From Theorem \ref{thm:llsc} we get $\mu([-1,\infty)) \geq 1/2$ with high probability, which implies that $\lambda_i \geq -1$ for all $i \leq N/2$ with high probability; we shall use this fact tacitly in the following.

Next, we define the function
\begin{equation*}
f(E) \;\deq\; \varrho([E,\infty))\,.
\end{equation*}
Fix $\epsilon > 0$. Then from the definitions \eqref{def_mu} and \eqref{def_gamma}, we find
\begin{equation} \label{f_com_rig}
\frac{i}{N} \;=\; f(\gamma_i) + \frac{1}{2N} \;=\; \mu([\lambda_i, \infty)) \;=\; f(\lambda_i) + O \pbb{\frac{N^\epsilon}{N}}
\end{equation}
with high probability, where in the last step we used Theorem \ref{thm:llsc}. We consider two cases.

Suppose first that
\begin{equation} \label{rigi_cond}
\gamma_i, \lambda_i \;\in\; \bigl[2 - N^{-2/3 + 2 \epsilon}, \infty\bigr)\,.
\end{equation}
Then from \eqref{asymp_gamma} we find $i \leq C N^{3 \epsilon}$. Moreover, by Proposition \ref{prop:bound_H}, we have $\abs{\lambda_i - 2} \leq N^{-2/3 + 2 \epsilon}$ with high probability. Since $\gamma_i \in [2 - N^{-2/3 + 2 \epsilon}, 2]$, we therefore deduce that
\begin{equation} \label{rigi_est1}
\abs{\lambda_i - \gamma_i} \;\leq\; 2 N^{-2/3 + 2 \epsilon} \;\leq\; C N^{3 \epsilon} N^{-2/3} i^{-1/3}\,.
\end{equation}

Conversely, suppose that \eqref{rigi_cond} does not hold. Then, by definition of $f$, we have
\begin{equation*}
f(\gamma_i) \vee f(\lambda_i) \;\geq\; c (N^{-2/3 + 2 \epsilon})^{3/2} \;\geq\;  N^\epsilon \frac{N^\epsilon}{N}\,.
\end{equation*}
From \eqref{f_com_rig} we therefore get
\begin{equation*}
f(\gamma_i) \;=\; f(\lambda_i) (1 + O(N^{-\epsilon}))
\end{equation*}
with high probability.
Since $f(\lambda) \asymp (2 - \lambda)^{3/2}$, we deduce that $2 - \lambda_i \asymp 2 - \gamma_i$ with high probability. Moreover, since $f'(\lambda) \asymp (2 - \lambda)^{1/2}$, we deduce that $f'(\lambda_i) \asymp f'(\gamma_i)$ with high probability, and hence that $f'(\lambda) \asymp f'(\gamma_i)$ with high probability for any $\lambda$ between $\lambda_i$ and $\gamma_i$. Using the mean value theorem, \eqref{asymp_gamma}, and \eqref{f_com_rig}, we therefore find
\begin{equation} \label{rigi_est2}
\abs{\lambda_i - \gamma_i} \;\asymp\; \frac{\abs{f(\lambda_i) - f(\gamma_i)}}{\abs{f'(\gamma_i)}} \;\leq\; C \frac{N^\epsilon}{N} \pbb{\frac{i}{N}}^{-1/3} \;=\; C N^\epsilon N^{-2/3} i^{-1/3}
\end{equation}
with high probability.

From the conclusions \eqref{rigi_est1} and \eqref{rigi_est2} for both cases, we conclude that $\abs{\lambda_i - \gamma_i} \leq C N^{3 \epsilon} N^{-2/3} i^{-1/3}$ with high probability, for all $i \leq N/2$. Since $\epsilon > 0$ was arbitrary, the proof of Theorem \ref{thm:rig} is complete.

\section{Extension of the spectral domain in the local law} \la{sec:extension}
The restriction of $z$ to $\f S$ in \ref{Th3} is natural in the light of the requirement $\Psi \leq C$, but is in fact not necessary. Indeed, once Theorem \ref{Th3} has been established, it is not too hard to extend it to a larger subset of $\C_+$, with sharper (in fact optimal) error bounds outside of the spectrum. (In fact, Theorem \ref{Th3} may be extended to the whole upper-half plane $\C_+$ with optimal error bounds throughout. Since the region $\eta > \tau^{-1}$ is of limited practical interest, we shall not deal with it.)

\subsection{Extension to all $z$ down to the real axis}

We begin by noting that the lower bound on $\eta$ in \eqref{def_S} may be omitted.
\beg{Th}\la{thm:ext1}
Theorem \ref{Th3} remains valid for $\f S$ replaced with
\begin{equation*}
\hb{E+\mathrm{i}\eta\col |E|\le \tau^{-1}\,,\, 0 < \eta\le \tau^{-1}}\,.
\end{equation*}
\en{Th} 
For an application of Theorem \ref{thm:ext1}, see Section \ref{sec:comparison}, where it is used to derive a simple universality result for the local eigenvalue statistics of Wigner matrices.

The rest of this subsection is devoted to the proof of Theorem \ref{thm:ext1}. The key observation is the following simple deterministic monotonicity result. Define
\begin{equation} \label{e:Gammadef}
\Gamma(z) \;\deq\; \max_{i,j} \abs{G_{ij}(z)}\,.
\end{equation}

\begin{lem} \label{lem:Gbd}
  For any $M>1$ and $z \in \C_+$ we have
  $\Gamma(E+\ii\eta/M) \leq M \Gamma(E+\ii\eta)$.
\end{lem}

\begin{proof}
Fix $E \in \R$ and write $\Gamma(\eta) = \Gamma(E+\ii \eta)$. 
For sufficiently small $h$, using the resolvent identity,
the Cauchy-Schwarz inequality, and \eqref{Ward}, we get
\begin{align*}
  \abs{\Gamma(\eta+h)-\Gamma(\eta)}
  &\;\leq\; \max_{i,j} \abs{G_{ij}(E+\ii (\eta+h)) - G_{ij} (E+\ii \eta)}
  \nonumber\\
  &\;\leq\; |h| \max_{i,j} \sum_k \abs{G_{ik}(E+\ii (\eta+h))G_{kj} (E+\ii \eta)}
  \;\leq\; |h| \sqrt{\frac{\Gamma(\eta+h)\Gamma(\eta)}{(\eta+h)\eta}}
  \,.
\end{align*}
Thus, $\Gamma$ is locally Lipschitz continuous, and its almost everywhere defined derivative satisfies
\begin{equation*}
  \absbb{\frac{\dd \Gamma}{\dd \eta}}
  \;\leq\; \frac{\Gamma}{\eta}\,.
\end{equation*}
This implies $\frac{\dd }{\dd \eta} (\eta \Gamma(\eta)) \geq 0$ and therefore $\Gamma(\eta/M) \leq M\Gamma(\eta)$ as claimed.
\end{proof}

In order to prove Theorem \ref{thm:ext1}, we have to prove \eqref{121414h53} and \eqref{121414h54} for $z$ satisfying $\abs{E} \leq \tau^{-1}$ and $0 < \eta < N^{-1 + \tau}$. Using that $\abs{m(z)} \leq C$ by \eqref{mupperlowerbounded}, we find
\begin{equation*}
\abs{G_{ij}(z) - m(z) \delta_{ij}} \;\leq\; \Gamma(z) + C \;\leq\; \frac{N^\tau}{N \eta} \Gamma(E + \ii N^{-1 + \tau}) + C \;\prec\; \frac{N^\tau}{N \eta} \;\leq\; N^{2 \tau} \Psi(z)\,,
\end{equation*}
where in the second step we used Lemma \ref{lem:Gbd}, in the third step we used \eqref{121414h54} with spectral parameter $E + \ii N^{-1 + \tau} \in \f S$, and in the last step we used that $\eta < N^{-1 + \tau}$.  Since $\tau > 0$ was arbitrary, we obtain
\begin{equation*}
\abs{G_{ij}(z) - m(z) \delta_{ij}} \;\prec\; \Psi(z) \,,
\end{equation*}
which concludes the proof of \eqref{121414h54} for $\eta < N^{-1 + \tau}$. The proof of \eqref{121414h53} is similar. This concludes the proof of Theorem \ref{thm:ext1}.

\subsection{Local law outside of the spectrum}
Next, we extend Theorem \ref{Th3} to all $E$ outside of the spectrum, with optimal error bounds. Here, ``outside of the spectrum''  means that the distance from $E$ to the limiting spectrum $[-2,2]$ is more than $N^{-2/3}$, the scale on which the extreme eigenvalues of $H$ fluctuate. Recall the definition of $\kappa$ from \eqref{def_kappa}.

\begin{Th}\label{thm:ext2}
Let $H$ be a Wigner matrix. Fix $\tau > 0$ and define the domain
\begin{equation*}
\f S^o \;\equiv\; \f S^o_N(\tau) \;\deq\; \hb{E + \ii \eta \col \abs{E} \geq 2 + N^{-2/3 + \tau}, \eta > 0}\,.
\end{equation*}
Then we have
\begin{equation} \label{s-m_outside}
s(z) \;=\; m(z) + O_\prec \pbb{\frac{1}{N} \frac{1}{(\kappa + \eta) + (\kappa + \eta)^2}}
\end{equation}
and
\begin{equation} \label{G-m_outside}
G_{ij}(z) \;=\; m(z) \delta_{ij} + O_\prec \pbb{\frac{1}{\sqrt{N}} \frac{1}{(\kappa + \eta)^{1/4} + (\kappa + \eta)^2}}
\end{equation}
uniformly for $i,j = 1, \dots, N$ and $z\in \f S^o$.
\end{Th}

Results of this type (typically in the stronger guise of \emph{isotropic local laws}; see Section \ref{sec:isotropic}) are very useful for instance in the study of spiked random matrix models. See \cite{BEN2,KnowlesYinIso,KnowlesYinOutliers,BKYYPCA}   for more details.

The rest of this subsection is devoted to the proof of Theorem \ref{thm:ext2}. We begin with \eqref{s-m_outside}, which is an easy consequence of the rigidity result from Theorem \ref{thm:rig}. Let $\epsilon \in (0,\tau/2)$. From Theorem \ref{thm:rig} we get, with high probability,
\begin{equation} \label{rigi_tilde}
|\lam_i-\gamma_i| \;\leq\; N^\epsilon N^{-2/3} \pb{i\wedge (N+1-i)}^{-1/3}
\end{equation}
for all $i = 1, \dots, N$.
Next, it is convenient to define a slightly modified version $\tilde \gamma_i$ of the typical eigenvalue location $\gamma_i$ from \eqref{def_gamma} through $N \int_{\tilde \gamma_i}^2 \varrho(\dd x) = i$.  With the convention $\tilde \gamma_0 \deq 2$, we may write
\begin{equation*}
s(z) - m(z) \;=\; \sum_{i = 1}^N \int_{\tilde \gamma_i}^{\tilde \gamma_{i - 1}} \varrho(\dd x) \pbb{\frac{1}{\lambda_i - z} - \frac{1}{x - z}}\,.
\end{equation*}
Using \eqref{rigi_tilde}, we find for $x \in [\tilde \gamma_i, \tilde \gamma_{i - 1}]$ that
\begin{equation*}
\abs{\lambda_i - \gamma_i} + \abs{x - \gamma_i} \;\leq\; 2 N^\epsilon N^{-2/3}\pb{i\wedge (N+1-i)}^{-1/3}
\end{equation*}
with high probability. Since $\abs{z - \gamma_i} \geq N^{-2/3 + \tau}$ for all $i$ and $\epsilon > 0$ can be made arbitrarily small, we therefore get
\begin{equation} \label{s-m_series}
\abs{s(z) - m(z)} \;\prec\; \frac{1}{N} \sum_{i = 1}^N N^{-2/3}\pb{i\wedge (N+1-i)}^{-1/3} \, \frac{1}{\abs{\gamma_i - z}^2}\,.
\end{equation}
Now obtaining \eqref{s-m_outside} is an elementary exercise in estimating the right-hand side of \eqref{s-m_series}, using \eqref{asymp_gamma} (and its analogue for $i \geq N/2$); we omit the details.

What remains is the proof of \eqref{G-m_outside}. Unlike \eqref{s-m_outside}, the rigidity estimate from Theorem \ref{thm:rig} is clearly not sufficient, since we need to control individual entries of $G$. We note first that, by polarization, it suffices to prove
\begin{equation} \label{claim outside of spectrum}
\absb{\scalar{\f v}{G(z) \f v} - m(z)} \;\prec\; \frac{1}{\sqrt{N}} \frac{1}{(\kappa + \eta)^{1/4} + (\kappa + \eta)^2}
\end{equation}
for all $z \in \f S^o$ and for all unit vectors $\f v$ having at most two nonzero components\footnote{In fact, \eqref{claim outside of spectrum} is valid for arbitrary unit vectors $\f v$; see Section \ref{sec:isotropic}.}. In the rest of the proof, $\f v$ denotes such a vector.

We split the proof into two cases: $\kappa + \eta \leq 1$ and $\kappa + \eta > 1$.

\subsubsection*{Case 1: $\kappa + \eta \leq 1$} 

Define $\eta_0 \deq N^{-1/2} \kappa^{1/4}$. By definition of the domain $\f S^o$,
 we have $\eta_0 \leq \kappa/2$. 
Using \eqref{121414h54} in Theorem \ref{Th3} and \eqref{Immmupperlowerbounded}, we find that  \eqref{claim outside of spectrum}  holds if $\eta \geq \eta_0$.
For the following we therefore take
\begin{equation} \label{eta_kappa_est}
0 \;<\; \eta \;\leq\; \eta_0 \;\leq\; \kappa/2\,.
\end{equation}
We proceed by comparison using the two spectral parameters
\begin{equation*}
z \;\deq\; E + \ii \eta\,, \qquad z_0 \;\deq\; E + \ii \eta_0\,.
\end{equation*}
Since \eqref{claim outside of spectrum} holds at $z_0$ by Theorem \ref{Th3}, it is enough to prove the estimates
\begin{equation} \label{z z0 comparison 1}
\absb{m(z) - m(z_0)} \;\leq\; C N^{-1/2} \kappa^{-1/4}
\end{equation}
and
\begin{equation} \label{z z0 comparison 2}
\absb{\scalar{\f v}{G(z) \f v} - \scalar{\f v}{G(z_0) \f v}} \;\prec\; N^{-1/2} \kappa^{-1/4}\,.
\end{equation}
We start with \eqref{z z0 comparison 1}. From the definition \eqref{def_m} and the square root decay of the density of $\varrho$ near the spectral edges $\pm 2$, it is not hard to derive the bound $m'(z) \leq C \kappa^{-1/2}$ for $z \in \f S^o$ satisfying $\eta \leq \kappa \leq 2$.
Therefore we get
\begin{equation*}
\absb{m(z) - m(z_0)} \;\leq\; C \kappa^{-1/2} \eta_0 \;=\; C N^{-1/2} \kappa^{-1/4}\,,
\end{equation*}
which is \eqref{z z0 comparison 1}.

What remains is to prove \eqref{z z0 comparison 2} under \eqref{eta_kappa_est}. By Theorem \ref{thm:rig} and \eqref{eta_kappa_est} we have $\abs{E} \geq \norm{H} + \eta_0$ with high probability. Thus we get
\begin{equation} \label{strong 4}
\im \scalar{\f v}{G(z) \f v} \;=\; \sum_i \frac{\abs{\scalar{\f v}{\f u_i}}^2 \eta}{(E - \lambda_i)^2 + \eta^2} \;\leq\;
2 \sum_i \frac{\abs{\scalar{\f v}{\f u_i}}^2 \eta_0}{(E - \lambda_i)^2 + \eta_0^2} \;=\; 2 \im \scalar{\f v}{G(z_0) \f v} \;\prec\; N^{-1/2} \kappa^{-1/4}
\end{equation}
by Theorem \ref{Th3} and \eqref{Immmupperlowerbounded} at $z_0$.

Finally, we estimate the real part of the error in \eqref{z z0 comparison 2} using
\begin{multline} \label{strong 5}
\absb{\real \scalar{\f v}{G(z) \f v} - \real \scalar{\f v}{G(z_0) \f v}} \;=\; \sum_i \frac{(E - \lambda_i)(\eta_0^2 - \eta^2) \abs{\scalar{\f u_i}{\f v}}^2}{\pb{(E - \lambda_i)^2 + \eta^2} \pb{(E - \lambda_i)^2 + \eta^2_0}}
\\
\leq\; \frac{\eta_0}{E - \lambda_1} \sum_i \frac{\eta_0 \abs{\scalar{\f u_i}{\f v}}^2}{(E - \lambda_i)^2 + \eta^2_0} \;\leq\; \im \scalar{\f v}{G(z_0) \f v}
\end{multline}
with high probability,
where in the last step we used that $\abs{E} \geq \norm{H} + \eta_0$ with high probability. Combining \eqref{strong 4} and \eqref{strong 5} completes the proof of \eqref{z z0 comparison 2}, and hence of \eqref{claim outside of spectrum} for the case $\kappa + \eta \leq 1$.

\subsubsection*{Case 2: $\kappa + \eta > 1$}
Define the random signed measure
\begin{equation*}
\tilde \mu \;\deq\; \sum_{i = 1}^N \scalar{\f v}{\f u_i}\scalar{\f u_i}{\f v} \delta_{\lambda_i} - \varrho
\end{equation*}
with Stieltjes transfrom
\begin{equation*}
\tilde m(z) \;\deq\; \int \frac{\tilde \mu(\dd x)}{x - z} \;=\; \scalar{\f v}{G(z) \f v} - m(z)\,.
\end{equation*}
The basic idea of the proof is to apply the Helffer-Sj\"ostrand formula from Proposition \ref{prop:HS} to the function
\begin{equation*}
f_z(x) \;\deq\; \frac{1}{x - z} + \frac{1}{z}\,.
\end{equation*}
To that end, we choose a smooth compactly supported cutoff function $\chi \in \cal C_c^\infty(\C;[0,1])$ satisfying $\chi = 1$ in the $1/6$-neighbourhood of $[-2,2]$, $\chi = 0$ outside of the $1/3$-neighbourhood of $[-2,2]$, and $\abs{\partial_{\bar w} \chi(w)} \leq C$. By Theorem \ref{thm:rig} we have $\supp \tilde \mu \subset \{\chi = 1\}$ with high probability. By assumption of Case 2, $\dist(z, [-2,2]) \geq 1/2$. Since $f_z$ is holomorphic in $\supp \chi$, the Helffer-Sj\"ostrand formula yields, for $x \in \supp \tilde \mu$,
\begin{equation} \label{HS formula}
f_z(x) \;=\; \frac{1}{\pi} \int_{\C} \frac{\partial_{\bar w} (f_z(w) \chi(w))}{x - w} \, \dd^2 w
\end{equation}
with high probability. Noting that $\int \dd \tilde \mu = 0$, we may therefore write
\begin{equation} \label{HS formula 2}
\tilde m(z) \;=\; \int \tilde \mu(\dd x) \, f_z(x)
\;=\;
\frac{1}{\pi} \int_{\C}  f_z(w) \, \partial_{\bar w} \chi(w) \, \tilde m(w) \, \dd^2 w
\end{equation}
with high probability, where in second step we used \eqref{HS formula} and the fact that $f_z$ is holomorphic away from $z$. The integral is supported on the set $\supp \partial_{\bar w} \chi \subset \h{w \col \dist(w,[-2,2]) \in [1/6,1/3]}$, on which we have the estimates $\abs{f_z(w)} \leq C (\kappa(z) + \eta(z))^{-2}$ and $\abs{\tilde m(w)} \prec N^{-1/2}$, as follows from Theorem \ref{Th3} and Case 1. As in Remark \ref{rem:Lipschitz}, we may easily obtain the simultaneous bounds $\abs{\tilde m(w)} \leq N^{\epsilon - 1/2}$ for all $w \in \supp \partial_{\bar w}$, with high probability, for any fixed $\epsilon > 0$. Plugging these estimates into \eqref{HS formula 2}, we get
\begin{equation*}
\abs{\tilde m(z)} \;\prec\;  (\kappa + \eta)^{-2} N^{-1/2}\,,
\end{equation*}
which is \eqref{claim outside of spectrum}. This conclude the proof of Case 2, and hence also of Theorem \ref{thm:ext2}.

\section{Local law and comparison arguments} \label{sec:comparison}

\subsection{Overview of the Green function comparison method} \label{sec:GFC_overview}
In this section we explain how the local law can be used to compare the local eigenvalue distribution of two random matrix ensembles whose entries are close. Here, the closeness is quantified using \emph{moment matching}.

\begin{Def} \label{def:match}
Let $\E'$ and $\E''$ denote the expectations of two different Wigner ensembles. We say that $\E'$ and $\E''$ \emph{match to order $n$} if we have, for all $i,j = 1, \dots, N$ and $k,l = 0,1,\dots$ satisfying $k+l \leq n$,
\begin{equation} \label{moment_match}
\E' (H_{ij})^k (\ol H_{ij})^l \;=\; \E'' (H_{ij})^k (\ol H_{ij})^l\,.
\end{equation}
\end{Def}

Typically, we choose $\E'$ to be a Gaussian ensemble from Definition \ref{def:Gaussian_W}, for which the statistic that we are interested in can be explicitly computed, and $\E''$ to be a general Wigner ensemble. This allows us to obtain unversality results, expressing that if $\E'$ and $\E''$ match to a high enough order then they have the same asymptotic eigenvalue and eigenvector statistics.

Such comparison ideas go back to Lindeberg and his proof of the central limit theorem \cite{LINDEBERG}. They were introduced into random matrix theory in \cite{ChatterjeeLindeberg} and used to derive a so-called \emph{four-moment theorem} for the local eigenvalue statistics in \cite{Tao-Vu_CMP,Tao-Vu_ActaMath2011}. The use of Green functions greatly simplifies such comparison techniques in random matrix theory. This was first observed in \cite{MR2981427},
  where the \emph{Green function comparison method} was introduced. In this section we give a simple application of Green function comparison to the local eigenvalue statistics, given in Theorem \ref{thm:fourmoment} below.

Before explaining the Green function comparison method, it is instructive to recall Lindeberg's original proof of the central limit theorem, which forms the core idea of the Green function comparison method. (For its statement, we do not aim for optimal assumptions as the emphasis is on a clear and simple proof.)

\begin{Th}[CLT \`a la Lindeberg]  \label{thm:CLT}
Let $n \in \N$, and suppose that $X_1, \dots, X_n$ are independent real-valued random variables satisfying $\sum_{i = 1}^n \E X_i^2 = n$ as well as $\E X_i = 0$ and $\E \abs{X_i}^3 \leq K$ for all $1 \leq i \leq n$ with some constant $K > 0$. Then for any $f \in \cal C^3(\R)$ we have
\begin{equation*}
\E f\pbb{\frac{X_1 + \cdots + X_n}{\sqrt{n}}} - \E f(Y) \;=\; O\pb{K \norm{f'''}_\infty \, n^{-1/2}}\,,
\end{equation*}
where $Y$ is a standard normal random variable.
\end{Th}
By a standard approximation argument of continuous bounded functions by smooth functions with bounded derivatives, Theorem \ref{thm:CLT} implies the convergence in distribution of $n^{-1/2} (X_1^{(n)} + \cdots + X_n^{(n)})$ taken from a triangular array $(X_i^{(n)})_{1 \leq i \leq n}$ satisfying the assumptions of Theorem \ref{thm:CLT} for each $n \in \N$.

\begin{proof}[Proof of Theorem \ref{thm:CLT}]
Let $(Y_i)_{1 \leq i \leq n}$ be a family of independent Gaussian random variables, independent of $(X_i)_{1 \leq i \leq n}$, satisfying $\E Y_i = 0$ and $\E Y_i^2 = \E X_i^2$ for $1 \leq i \leq n$. For $0 \leq \gamma \leq n$ define the new family $(Z_i^\gamma)_{1 \leq i \leq n}$ through
\begin{equation*}
Z_i^\gamma \;\deq\;
\begin{cases}
X_i & \text{if } i \leq \gamma
\\
Y_i & \text{if } i > \gamma\,.
\end{cases}
\end{equation*}
By telescoping  we get
\begin{equation} \label{Lindeberg_telescope}
\E f \pbb{\frac{1}{\sqrt{n}} \sum_{i = 1}^n X_i} - \E f \pbb{\frac{1}{\sqrt{n}} \sum_{i = 1}^n Y_i} \;=\; \sum_{\gamma = 1}^n \E \qBB{f \pbb{\frac{1}{\sqrt{n}} \sum_{i = 1}^n Z_i^{\gamma}} - f \pbb{\frac{1}{\sqrt{n}} \sum_{i = 1}^n Z_i^{\gamma - 1}}}\,.
\end{equation}
Note that the two sums over $i$ on the right-hand side only differ in the summand $i = \gamma$, where we have $Z_\gamma^\gamma = X_\gamma$ and $Z^{\gamma - 1}_\gamma = Y_\gamma$.
We estimate the difference using a Taylor expansion of order three around the random variable
\begin{equation*}
W^\gamma_n \;\deq\; \frac{1}{\sqrt{n}} \sum_{i = 1}^n \ind{i \neq \gamma} Z_i^{\gamma}\,.
\end{equation*}
Then the expectation on the right-hand side of \eqref{Lindeberg_telescope} is equal to
\begin{equation} \label{Lind_oneterm}
\E \qb{f(W^\gamma_n + X_\gamma) - f(W^\gamma_n + Y_\gamma)} \;=\; \E \qBB{f'(W^\gamma_n) \frac{X_\gamma - Y_\gamma}{n^{1/2}} + \frac{1}{2} f''(W^\gamma_n) \frac{X_\gamma^2 - Y_\gamma^2}{n}} + O\pb{K \norm{f'''}_\infty \, n^{-3/2}}\,,
\end{equation}
where we used the assumption $\E \abs{X_i}^3 \leq K$ as well as $\E \abs{Y_i}^3 \leq 3 (\E Y_i^2)^{3/2} = 3 (\E X_i^2)^{3/2} \leq 3 \E \abs{X_i}^3 \leq 3 K$, by H\"older's inequality. Since $X_\gamma$ and $Y_\gamma$ are independent of $W_n^\gamma$, and using that the first two moments of $X_\gamma$ and $Y_\gamma$ match, we find that the right-hand side of \eqref{Lind_oneterm} is $O\pb{K \norm{f'''}_\infty \, n^{-3/2}}$. Plugging this back into \eqref{Lindeberg_telescope}, we find that \eqref{Lindeberg_telescope} is $O\pb{K \norm{f'''}_\infty \, n^{-1/2}}$. The claim then follows from the observation that 
$Y \deq \frac{1}{\sqrt{n}} \sum_{i = 1}^n Y_i$ is a standard normal random variable.
\end{proof}

Note that the central limit theorem can be regarded as a universality result: the asymptotic distribution of $n^{-1/2} \sum_{i = 1}^N X_i$ and $n^{-1/2} \sum_{i = 1}^N Y_i$ coincide if variables $X_i$ and $Y_i$ have matching first two moments.
Lindeberg's proof given above lays down the strategy for the Green function comparison method, which performs a similar comparison for the distribution of Green functions. Suppose that we have two random matrix ensembles $\E'$ and $\E''$ that match to order $n$ (see Definition \ref{def:match}). As in Lindeberg's proof, we introduce the trivial independent coupling of these ensembles, where our probability space consists of the independent matrices $H'$ and $H''$ whose laws have expectations $\E'$ and $\E''$ respectively.

In order to obtain a telescoping sum, we choose some (arbitrary) bijection
\begin{equation*}
\phi \;\col\; \h{(i,j) \col 1 \leq i \leq j \leq N} \;\to\; \hb{1,2, \dots, \gamma_N}\,,
\end{equation*}
where $\gamma_N \deq \frac{N(N+1)}{2}$ is the number of independent matrix entries. We then define the matrices $H^0, H^1, \dots, H^{\gamma_N}$ as the Wigner matrices satisfying
\begin{equation*}
H_{ij}^\gamma \;\deq\;
\begin{cases}
H_{ij}'' & \text{if } \phi(i,j) \leq \gamma
\\
H_{ij}' & \text{if } \phi(i,j) > \gamma
\end{cases}
\end{equation*}
for $i \leq j$.

Now suppose that $f$ is some statistic of Wigner matrices whose expectation we would like to understand. We telescope
\begin{equation} \label{GFC_telescope}
\E'' f(H) - \E' f(H) \;=\; \E (f(H'') - f(H')) \;=\; \sum_{\gamma = 1}^{\gamma_N} \E \pb{f(H^\gamma) - f(H^{\gamma - 1})}\,.
\end{equation}
As in Lindeberg's proof, we estimate each summand using Taylor's theorem. Note that there are $\gamma_N \asymp N^2$ terms, so that, if all derivatives of $f$ were of order one, we would require \emph{four moment matching}  instead of \emph{two moment matching} (as in Lindeberg's proof) for the sum to be $o(1)$.
The importance of this four moment matching condition for random matrices was first realized in \cite{ChatterjeeLindeberg}, and it was first applied to the question of universality of the local spectral statistics in \cite{Tao-Vu_ActaMath2011}.

It clear therefore that the key difficulty is to obtain good control on the derivatives of $f$. Especially for statistics that give information about the local eigenvalue distribution, $f$ can be a very unstable function and controlling its derivatives is a highly nontrivial task. It is here that the local law enter the game, by providing such estimates.

A good choice of statistic $f$ is some combination of Green functions, typically a well-behaved smooth function of some polynomial in the Green function entries.  As outlined in Section \ref{sec:Gfunct} and detailed in Sections \ref{sec:les}--\ref{sec:GFC} below for the example of local eigenvalue statistics, such combinations cover all eigenvalue and eigenvector statistics of interest. Performing a Taylor expansion of each summand in \eqref{GFC_telescope},  we find that we need to differentiate Green functions in the entries of $H$. By \eqref{res_exp}, this results in a sum over various polynomials in the Green function entries. The local law, for instance through \eqref{121414h54}, provides the required control of such polynomials. Note that in order to obtain information about the distribution of the eigenvalues on small scales, the imaginary part $\eta$ of the Green function spectral parameter has to be chosen very small, and such bounds on the Green function are a highly nontrivial input. See Lemma \ref{lem:GFC} below for a precise statement and proof of Green function comparison.

The Green function comparison method has been successfully applied to many different problem in random matrix theory: local eigenvalue distribution in the bulk \cite{MR2981427,EKYYERGII}, local eigenvalue distribution near the edge \cite{EYYrigi}, distribution of the eigenvectors \cite{KnowlesYinEig},  distribution of outliers in deformed matrix models \cite{KnowlesYinIso,KnowlesYinOutliers} and even deriving large deviation estimates for the Green function \cite{KnowlesYinIso,BourgadeCirc2,BaoErdosblockband}.
 In some of these applications, such as \cite{KnowlesYinIso,KnowlesYinOutliers},  the Green function comparison method is used not only to   obtain an upper bound on   the difference between the two ensembles $H'$ and $H''$ but to actually analyse this difference and hence understand the failure of universality. For instance, the distribution of outliers of deformed Wigner matrices in general depend on the third and fourth moments of the matrix entries, and the Green function comparison method can be used to compute the dependence of this distribution on the moments of $H$ in full generality \cite{KnowlesYinIso,KnowlesYinOutliers}.

In these notes we present the Green function comparison method as applied to the local eigenvalue statistics in the bulk spectrum.

\subsection{Local eigenvalue statistics} \label{sec:les}

A central problem in random matrix theory is to characterize the local eigenvalue statistics. It is best addressed using \emph{correlation functions} of the eigenvalue process. To simplify the notation somewhat, we assume that $H$ has an absolutely continuous law\footnote{The case of general $H$ may be easily recovered by considering $H^\epsilon \deq H + \epsilon V$ instead of $H$, where $V$ is a GOE or GUE matrix and $\epsilon > 0$. Then $H^\epsilon$ has an absolutely continuous law for all $\epsilon > 0$, and the eigenvalue distribution of $H$ coincides with that of $H^\epsilon$ after taking the limit $\epsilon \downarrow 0$.}. By the Weyl integration formula (see e.g.\ \cite[Sec. 4.1]{agz}, \cite[Sec. 2.6.1]{TAO2}, or \cite[Sec. 4.1]{PasturBook}),   we find that the eigenvalues of $H$ also have an absolutely continuous law.

For any $k = 1,2, \dots, N$ we define the \emph{$k$-point intensity measure} of the eigenvalue process $\sum_{i = 1}^N \delta_{\lambda_i}$ as the measure $\rho_k$ on $\R^k$ defined by
\begin{equation*}
\int f(x_1, \dots, x_k) \, \rho_k(\dd x_1, \dots, \dd x_k) \;\deq\;\E \sideset{}{^*}\sum_{i_1, \dots, i_k = 1}^N f(\lambda_{i_1}, \dots, \lambda_{i_k})
\end{equation*}
for bounded and measurable $f$,
where the star next to the sum means that the sum is restricted to $k$-tuples $i_1, \dots, i_k$ of distinct indices. Since the law of the eigenvalues of $H$ is absolutely continuous, we find that $\rho_k(\dd x_1, \dots, \dd x_k) = p_k(x_1, \dots, x_k) \, \dd x_1 \cdots \dd x_k$ has a density $p_k$, called the \emph{$k$-point correlation function} of the eigenvalue process. Formally, $p_k$  can be written as
\begin{equation} \label{def_p_k}
p_k(x_1, \dots, x_k) \;=\; \E \sideset{}{^*}\sum_{i_1, \dots, i_k = 1}^N \delta(x_1 - \lambda_{i_1}) \cdots \delta(x_k - \lambda_{i_k})\,,
\end{equation}
where $\delta$ is the Dirac delta function.  

The interpretation of $\rho_k(\dd x_1, \dots, \dd x_k)$ is the expected number of $k$-tuples of eigenvalues lying in the disjoint sets $\dd x_1, \dots, \dd x_k$, so that $p_k(x_1, \dots, x_k)$ is the expected density of $k$-tuples of eigenvalues near the distinct locations $x_1, \dots, x_k$. Moreover, it is immediate that $p_k$ is $\frac{N!}{(N - k)!}$ times the $k$-point marginal of the joint symmetrized probability density of the eigenvalues $\lambda_1, \dots, \lambda_N$.

In order to see individual eigenvalues, we need to zoom into the spectrum to the scale of the typical eigenvalue spacing.  We denote by
\begin{equation*}
\varrho_E \;\deq\; \frac{\varrho(\dd E)}{\dd E} \;=\; \frac{1}{2\pi} \sqrt{(4 - E^2)_+}
\end{equation*}
the density of the semicircle distribution at $E \in (-2,2)$.   By the semicircle law, the typical density of eigenvalues near the energy $E \in (-2,2)$ is $N \varrho_E$, so that the typical eigenvalue spacing is of order $(N \varrho_E)^{-1}$. In order to perform the zooming in, we define the \emph{local $k$-point intensity measure around the energy $E \in (-2,2)$}, denoted by $\rho_{k,E}$, as the $k$-point intensity measure of rescaled eigenvalue process
\begin{equation} \label{rescaled_ev_process}
\sum_{i = 1}^N \delta_{N \varrho_E (\lambda_i - E)}\,.
\end{equation}
Explicitly,
\begin{equation*}
\rho_{k,E}(\dd u_1, \dots, \dd u_k) \;=\; \rho_k \pbb{E + \frac{\dd u_1}{N \varrho_E}, \dots, E + \frac{\dd u_1}{N \varrho_E}}\,.
\end{equation*}
The measure $\rho_{k,E}$ has density
\begin{equation} \label{def_pE}
p_{k,E}(u_1, \dots, u_k) \;\deq\; \frac{1}{(N \varrho_E)^k} \, p_k \pbb{E + \frac{u_1}{N \varrho_E}, \dots, E + \frac{u_1}{N \varrho_E}}\,,
\end{equation}
the \emph{local $k$-point correlation function at energy $E$}.
Note that $p_{k,E}(u_1, \dots, u_k)$ has the interpretation of the expected density of $k$-tuples of eigenvalues near the distinct locations $u_1, \dots, u_k$ \emph{on the microscopic scale} where the typical separation of the eigenvalues is $1$. We therefore expect $p_{k,E}$ to have a limit as $N \to \infty$, for fixed $E \in (-2,2)$ and $k \in \N$. Indeed, in a series of pioneering works \cite{Gaudin, Dyson, Mehta}  Gaudin, Dyson, and Mehta proved that if $H$ is the GUE then
\begin{equation} \label{det_process}
p_{k,E}(u_1, \dots, u_k) \;\to\; \det \pb{K(u_i - u_j)}_{i,j = 1}^k
\end{equation}
weakly as $N \to \infty$, where
\begin{equation*}
K(u) \;\deq\; \frac{\sin (\pi u)}{\pi u}
\end{equation*}
is the \emph{sine kernel}. Thus, the correlation functions in the limit have a determinantal structure. Generally, a point process whose correlation functions have a determinantal structure as on the right-hand side of \eqref{det_process} is called \emph{determinantal}. We have therefore seen that the rescaled eigenvalue process \eqref{rescaled_ev_process} converges as $N \to \infty$ to a determinantal point process with kernel $K$. A similar formula holds for the GOE with a more complicated kernel $K$.

Returning to the analogy of the central limit theorem, the limiting behaviour of $p_{k,E}$ for the Gaussian ensembles GUE and GOE can be determined thanks to an explicit computation using the ``integrable'' nature of the Gaussian distribution. The question of universality of the local eigenvalue statistics is whether \eqref{det_process} is true for other Wigner matrices as well.

\subsection{Green function comparison for local eigenvalue statistics} \label{sec:GFC}
We now explain how the Green function comparison method can be used to prove that the asymptotic behaviour of $p_{k,E}$ only depends on the first four moments of the Wigner ensembles.

\begin{Th} \label{thm:fourmoment}
Let $\E'$ and $\E''$ be Wigner ensembles with local $k$-point correlation functions \eqref{def_pE} denoted by $p_{k,E}'$ and $p_{k,E}''$ respectively. Suppose that $\E'$ and $\E''$ match to order 4. Then for any fixed $k \in \N$, $E \in (-2,2)$, and $f \in \cal C^1_c(\R^k)$ there exists a $\xi > 0$ such that
\begin{equation} \label{def_vague_conv}
\int (p_{k,E}' - p_{k,E}'')(u_1, \dots, u_k) \, f(u_1, \dots, u_k) \, \dd u_1 \cdots \dd u_k \;=\; O(N^{-\xi})\,.
\end{equation}
\end{Th}

The rest of this subsection is devoted to the proof of Theorem \ref{thm:fourmoment} using Green function comparison. We have to express the local correlation function $p_{k,E}$ in terms of the Green function. To this end, we smear out   $p_{k,E}$ using the approximate delta function $\theta_\eta$ from \eqref{def_theta}, which can be written using the Green function as in \eqref{6121414h44}. Fix $\epsilon > 0$ and choose a scale $\eta \deq N^{-1 - \epsilon}$ slightly smaller than the typical eigenvalue spacing $(N \varrho_E)^{-1}$. Then we define the corresponding microscopic scale $\tilde \eta \deq N \varrho_E \eta = \varrho_E N^{-\epsilon}$ and the corresponding smeared-out observable
\begin{equation} \label{def_f_smeared}
f_{\tilde \eta}(u_1, \dots, u_k) \;\deq\; \int \dd v_1 \cdots \dd v_k \, f(v_1, \dots, v_k) \, \theta_{\tilde \eta}(u_1 - v_1) \cdots \theta_{\tilde \eta}(u_k - v_k)\,.
\end{equation}

The main work in the proof is to estimate \eqref{def_vague_conv} with $f$ replaced by $f_{\tilde \eta}$. Having done this, going back to $f$ will easily follow provided we have an upper bound on $p_{k,E}'$ and $p_{k,E}''$; such a bound will be obtained from the local semicircle law (see Lemma \ref{lem:approx_eta} below).

We transfer the smearing-out to the local $k$-point correlation function using
\begin{equation} \label{smearing_corr_f}
\int p_{k,E}(u_1, \dots, u_k) \, f_{\tilde \eta}(u_1, \dots, u_k) \, \dd u_1 \cdots \dd u_k
\;=\; \int p_{k,E, \tilde \eta}(u_1, \dots, u_k) \, f(u_1, \dots, u_k) \, \dd u_1 \cdots \dd u_k\,,
\end{equation}
where we defined the smeared-out local $k$-point correlation function
\begin{equation}
p_{k,E, \tilde \eta}(u_1, \dots, u_k) \;\deq\; \int p_{k,E}(v_1, \dots, v_k) \theta_{\tilde \eta}(u_1 - v_1) \cdots \theta_{\tilde \eta}(u_1 - v_1) \,\dd v_1 \cdots \dd v_k\,.
\end{equation}
Using the macroscopic variables
\begin{equation} \label{def_local_coordinates}
x_i \;\deq\; E + \frac{u_i}{N \varrho_E}
\end{equation}
associated with the microscopic variables $u_i$, we find
\begin{align*}
p_{k,E, \tilde \eta}(u_1, \dots, u_k)
&\;=\; \frac{1}{(N \varrho_E)^k} \int p_k(y_1, \dots, y_k) \, \theta_\eta(x_1 - y_1) \cdots \theta_\eta(x_k - y_k) \, \dd y_1 \cdots \dd y_k
\\
&\;=\; \frac{1}{(N \varrho_E)^k} \E \sideset{}{^*}\sum_{i_1, \dots, i_k} \theta_\eta(x_1 - \lambda_{i_1}) \cdots \theta_\eta(x_k - \lambda_{i_k})\,.
\end{align*}

From now on, to simplify notation, we only consider the case $k = 2$; other $k$ are handled in the same way with more complicated notation. We find
\begin{align*}
p_{2,E, \tilde \eta}(u_1, u_2) &\;=\;  \frac{1}{(N \varrho_E)^2} \E \sum_{i,j} \theta_\eta(x_1 - \lambda_i) \, \theta_\eta(x_2 - \lambda_j) -  \frac{1}{(N \varrho_E)^2} \E \sum_{i} \theta_\eta(x_1 - \lambda_i) \, \theta_\eta(x_2 - \lambda_i)
\\
&\;=\; \frac{1}{(N \varrho_E)^2} \E \tr \theta_\eta(x_1 - H) \tr \theta_\eta(x_2 - H) - \frac{1}{(N \varrho_E)^2}\E \tr \pb{\theta_\eta(x_1 - H) \theta_\eta(x_2 - H)}
\\
&\;=\; \frac{1}{\pi^2 (N \varrho_E)^2} \E \tr \pb{\im G(z_1)} \tr \pb{\im G(z_2)} - \frac{1}{\pi^2 (N \varrho_E)^2} \E \tr \pb{\im G(z_1) \im G(z_2)}\,,
\end{align*}
where $z_i \deq x_i + \ii \eta$ for $i = 1,2$.
Here in the last step we used \eqref{def_theta}. This formula is the desired representation of the smeared-out local two-point correlation function in terms of the Green function.

Recalling \eqref{smearing_corr_f} and the definition \eqref{def_local_coordinates}, we find that Theorem \ref{thm:fourmoment} follows immediately from the following two lemmas.

\begin{lem}[Green function comparison] \label{lem:GFC}
Suppose that $\E'$ and $\E''$ match to order $4$.
Fix $E \in (-2,2)$ and $\epsilon \in (0,1/20)$. Then there exists a $\xi > 0$ such that for $\eta \deq N^{-1-\epsilon}$ we have
\begin{equation} \label{GFC1}
\frac{1}{N^2}(\E' - \E'') \tr G(z) \tr G(w) \;=\; O(N^{-\xi})
\end{equation}
and
\begin{equation} \label{GFC2}
\frac{1}{N^2}(\E' - \E'') \tr \pb{G(z) G(w)} \;=\; O(N^{-\xi})
\end{equation}
for any $z,w \in \C$ satisfying $\abs{\im z} = \abs{\im w} = \eta$ and $\abs{\real z - E} + \abs{\real w - E} \leq C / N$.
\end{lem}

\begin{lem}[Removal of smearing-out] \label{lem:approx_eta}
Fix $E \in (-2,2)$ and $\epsilon > 0$. 
Let $f \in \cal C^1_c(\R^2)$ and denote by $f_{\tilde \eta}$ its smeared-out version from \eqref{def_f_smeared} with $\tilde \eta = \varrho_E N^{-\epsilon}$. Then there exists a $\xi > 0$ such that
\begin{equation*}
\int p_{k,E}(u_1,u_2) (f - f_{\tilde \eta})(u_1,u_2) \, \dd u_1 \, \dd u_2 \;=\; O(N^{-\xi})\,,
\end{equation*}
where $p_{k,E}$ is the local two-point correlation function at energy $E$ of a Wigner ensemble.
\end{lem}

\begin{proof}[Proof of Lemma \ref{lem:GFC}]
We write the left-hand sides of \eqref{GFC1} and \eqref{GFC2} as
\begin{equation}
(\E' - \E'') \frac{1}{N^2} \sum_{i,j} G_{jj}(z) G_{ii}(w) \,, \qquad (\E' - \E'') \frac{1}{N^2} \sum_{i,j} G_{ij}(z) G_{ji}(w)
\end{equation}
respectively. We only prove \eqref{GFC1}; the proof of \eqref{GFC2} is analogous. We use the independent coupling $(H',H'')$ from Section \ref{sec:GFC_overview} and the telescopic sum from \eqref{GFC_telescope}, with $f(H) \deq \frac{1}{N^2} \sum_{i,j} G_{ii}(z) G_{jj}(w)$. For each $\gamma = 1, \dots, \gamma_N$ we introduce the Hermitian matrix $W^\gamma = (W^\gamma_{ij})$ through \begin{equation*}
W_{ij}^\gamma \;\deq\;
\begin{cases}
H_{ij}'' & \text{if } \phi(i,j) < \gamma
\\
0 & \text{if } \phi(i,j) = \gamma
\\
H_{ij}' & \text{if } \phi(i,j) > \gamma
\end{cases}
\end{equation*}
for $i \leq j$. Thus, the matrices $H^\gamma - W^\gamma$ and $H^{\gamma - 1} - W^\gamma$ have rank at most two, and they are independent of $W^\gamma$.

For $\gamma = 1, \dots, \gamma_N$ we define the Green functions
\begin{equation*}
R^\gamma(z) \;\deq\; (W^\gamma - z)^{-1}\,, \qquad S^\gamma(z) \;\deq\; (H^{\gamma - 1} - z)^{-1}\,, \qquad T^\gamma(z) \;\deq\; (H^\gamma - z)^{-1}\,.
\end{equation*}
Since $\gamma_N = O(N^2)$, by \eqref{GFC_telescope} it suffices to prove that for each $\gamma = 1, \dots, \gamma_N$ we have
\begin{equation} \label{GFC_step_est}
\E \frac{1}{N^2} \sum_{i,j} T^\gamma_{ii}(z) T^\gamma_{jj}(w) - \E \frac{1}{N^2} \sum_{i,j} S^\gamma_{ii}(z) S^\gamma_{jj}(w) \;=\; O(N^{-2 - \xi})\,.
\end{equation}
To prove \eqref{GFC_step_est}, it suffices to prove
\begin{align}
\E \frac{1}{N^2} \sum_{i,j} T^\gamma_{ii}(z) T^\gamma_{jj}(w)  &\;=\; \cal A(z,w) + O(N^{-2 - \xi})\,, \label{GFC_comp1}
\\
\E \frac{1}{N^2} \sum_{i,j} S^\gamma_{ii}(z) S^\gamma_{jj}(w)  &\;=\; \cal A(z,w) + O(N^{-2 - \xi})\,, \label{GFC_comp2}
\end{align}
for some $\cal A(z,w) \in \C$.
We focus on \eqref{GFC_comp1}. The idea is to perform a resolvent expansion of $T^\gamma$ around $R^\gamma$. Denoting by $\Delta^\gamma \deq H^\gamma - W^\gamma$, we have
\begin{equation} \label{res_exp_T}
T^\gamma \;=\; \sum_{\ell = 0}^4 R^\gamma (-\Delta^\gamma R^\gamma)^\ell + T^\gamma (-\Delta^\gamma R^\gamma)^5\,.
\end{equation}

From now on, we fix $\gamma = 1, \dots, \gamma_N$ satisfying $\phi(a,b) = \gamma$ with $a \leq b$, and consistently omit $\gamma$ from our notation. In particular,
\begin{equation} \label{V_explicit}
\Delta_{ij} \;=\; \delta_{ia} \delta_{jb} H''_{ij} + (1 - \delta_{ab}) \delta_{ib} \delta_{ja} H''_{ij}\,,
\end{equation}
so that $\Delta$ is a bounded-rank Hermitian matrix consisting of entries of $H''$; its rank is two if $a < b$ and one if $a = b$.

Plugging \eqref{res_exp_T} into the left-hand side of \eqref{GFC_comp1}, we get
\begin{multline} \label{GFC_expansion}
\E \frac{1}{N^2} \sum_{i,j} T^\gamma_{ii}(z) T^\gamma_{jj}(w)
\;=\; \E \frac{1}{N^2} \sum_{i,j} \pBB{\sum_{\ell = 0}^4 \pb{R(z) (-\Delta R(z))^\ell}_{ii} + \pb{T(z) (-\Delta R(z))^5}_{ii}}
\\
\times 
\pBB{\sum_{\ell = 0}^4 \pb{R(w) (-\Delta R(w))^\ell}_{jj} + \pb{T(w) (-\Delta R(w))^5}_{jj}}\,.
\end{multline}
We multiply out the parentheses to get a series of terms. We split it them into the \emph{main terms}, which contain a total power of $\Delta$ at most four, and \emph{remainder terms}. The main terms are large but their contribution depends only on the first four moments of $\Delta$, i.e.\ of $H''$, and they will be put into $\cal A(z,w)$. The remainder terms are small thanks to the local semicircle laws for $T$ and $R$ and the estimate
\begin{equation} \label{Delta_bounds}
\norm{\Delta_{ij}}_\ell \;\leq\; C_\ell \, \indb{\{i,j\} = \{a,b\}} N^{-1/2}\,,
\end{equation}
which follows from Definition \ref{def:Wigner} (iii).
Explicitly, we define
\begin{equation*}
\cal A(z,w) \;\deq\; \frac{1}{N^2} \sum_{i,j} \sum_{\ell_1, \ell_2 \geq 0} \ind{\ell_1 + \ell_2 \leq 4} \, \E \qB{\pb{R(z) (-\Delta R(z))^{\ell_1}}_{ii} \pb{R(w) (-\Delta R(w))^{\ell_2}}_{jj}}
\end{equation*}
Since $\Delta$ is independent of $R$, we find that $\cal A(z,w)$ depends on $(H',H'')$ only through $W$ and the first four moments of $H''$.

Repeating the same argument for \eqref{GFC_comp2}, we find the \emph{same} $\cal A(z,w)$ because the first four moments of $H''$ and $H'$ coincide.

What remains, therefore, is to estimate
\begin{equation} \label{GFC_splitTS}
\E \frac{1}{N^2} \sum_{i,j} T^\gamma_{ii}(z) T^\gamma_{jj}(w) - \cal A(z,w) \,, \qquad
\E \frac{1}{N^2} \sum_{i,j} S^\gamma_{ii}(z) S^\gamma_{jj}(w) - \cal A(z,w)\,.
\end{equation}
We again only concentrate on the first expression. It is given by a sum
\begin{equation} \label{sum_as_monom}
\E \frac{1}{N^2} \sum_{i,j} T^\gamma_{ii}(z) T^\gamma_{jj}(w) - \cal A(z,w) \;=\; \frac{1}{N^2} \sum_{i,j} \sum_{\alpha = 1}^C \E X_{ij}(\alpha)
\end{equation}
for some constant $C$, where $X_{ij}(\alpha)$ lies in the set $\cal M_{\ell,p}$, defined as the set monomials in the entries of $T(z), R(z), T(w), R(w)$, and $\Delta$, such that the total degree of the entries of $\Delta$ is $\ell$ and total degree of the entries of $T(z), R(z), T(w), R(w)$ is $p$. Moreover, we immediately find that for every $X_{ij}(\alpha) \in \cal M_{\ell,p}$ we have $p \leq 2\ell$.

We estimate each such monomial by using the entrywise bounds \eqref{Delta_bounds} and
\begin{equation} \label{bound_RT}
\abs{R_{ij}(\zeta)} + \abs{T_{ij}(\zeta)} \;\prec\; N^{\epsilon}
\end{equation}
for all $i,j$ and $\zeta \in \{z,w\}$. The estimate \eqref{bound_RT} is the crucial estimate that makes the Green function comparison method work; it is proved using the local semicircle law. Before proving it, we explain how to use it to conclude the proof of Lemma \ref{lem:GFC}.

Take a monomial $X \in \cal M_{\ell,p}$ with $\ell, p \leq C$. Using H\"older's inequality, \eqref{Delta_bounds}, and \eqref{bound_RT}, we find
\begin{align}
\abs{\E X} &\;\leq\; \pB{\max_{i,j} \norm{\Delta_{ij}}_{\ell + p}^\ell} \max_{i,j} \sum_{\zeta \in \{z,w\}} \pB{\norm{R_{ij}(\zeta)}_{\ell + p}^p + \norm{T_{ij}(\zeta)}_{\ell + p}^p}
\notag \\ \label{bound_monom}
&\;\leq\; N^{\delta - \ell /2 + p \epsilon}
\end{align}
for any fixed $\delta > 0$ and large enough $N$. Here we estimated $\norm{R_{ij}(\zeta)}_{p+\ell} \prec N^\epsilon$ (and similarly for $\norm{T_{ij}(\zeta)}_{p+\ell}$), by using the estimate $\abs{R_{ij}(\zeta)} \leq N^{\epsilon + \delta_1}$ on an event of probability at least $1 - N^{-D}$ for any large $D > 0$ and small $\delta_1 > 0$, for large enough $N$ (as follows from \eqref{bound_RT}), and the trivial bound $\abs{R_{ij}(\zeta)} \leq \eta^{-1} = N^{1 + \epsilon}$ for the complementary low-probability event.

Plugging the estimate \eqref{bound_monom} into \eqref{sum_as_monom} and recalling that $p \leq 2 \ell$ and $\ell = 5, \dots, 10$ for all $X_{ij}(\alpha)$, we find
\begin{equation*}
\E \frac{1}{N^2} \sum_{i,j} T^\gamma_{ii}(z) T^\gamma_{jj}(w) - \cal A(z,w) \;=\; O\pb{N^{-5/2 + 10 \epsilon + \delta}}
\end{equation*}
for any $\delta > 0$ and large enough $N$.
A similar bound holds for the second term of \eqref{GFC_splitTS}. We conclude that \eqref{GFC_comp1} and \eqref{GFC_comp2} hold with $\xi \deq 1/2 - 10 \epsilon - \delta$.

All that remains to complete the proof of Lemma \ref{lem:GFC} is the estimate \eqref{bound_RT}. For definiteness, we suppose that $\zeta = z$ and $\im z = \eta > 0$. (The case $\im z = - \eta$ is obtained simply by complex conjugation.) Since $T(z)$ is the Green function of a Wigner matrix, we immediately get from Theorem \ref{thm:ext1} that $\abs{T_{ij}(z)} \prec (N \eta)^{-1} \leq C N^\epsilon$. For the analogous bound for $R(z)$ we cannot invoke Theorem \ref{thm:ext1} because $W$ is not a Wigner matrix. However, $W$ differs from the Wigner matrix $H^\gamma \equiv H$ by $\Delta$, so that we may again perform a resolvent expansion, except that we now expand around $T$ instead of $R$:
\begin{equation*}
R_{ij} \;=\; \sum_{\ell = 0}^2 \pb{T (\Delta T)^\ell}_{ij} + \pb{R (\Delta T)^{3}}_{ij}\,.
\end{equation*}
Using the above bound on the entries of $T_{ij}$ and the trivial bound $\abs{R_{ik}} \leq \eta^{-1} \leq C N^{1 + \epsilon}$, we therefore get
\begin{equation*}
\abs{R_{ij}} \;\prec\; \sum_{\ell = 0}^2 N^\epsilon (N^{-1/2 + \epsilon})^\ell + N^{1 + \epsilon} (N^{-1/2 + \epsilon})^3 \;\prec\; N^{\epsilon}\,,
\end{equation*}
as claimed. This concludes the proof of Lemma \ref{lem:GFC}.
\end{proof}

\begin{proof}[Proof of Lemma \ref{lem:approx_eta}]
Using \eqref{def_pE} and \eqref{def_p_k} we write
\begin{equation} \label{pkE_int_f}
I \;\deq\; \int p_{k,E}(u_1,u_2) (f - f_{\tilde \eta})(u_1,u_2) \, \dd u_1 \, \dd u_2 \;=\; \E \sum_{i \neq j} (f - f_{\tilde \eta}) \pb{N \varrho_E (\lambda_i - E), N \varrho_E (\lambda_j - E)}\,.
\end{equation}
We choose $\xi \in (0,\epsilon / 3)$ and set $\zeta \deq N^{\xi}$. We split
\begin{align}
&\mspace{-20mu}(f - f_{\tilde \eta})(u_1, u_2)
\notag \\
&\;=\; (f - f_{\tilde \eta})(u_1, u_2)  \, \ind{\abs{u_1} \leq \zeta} \, \ind{\abs{u_2} \leq \zeta}
\notag \\
&\qquad 
- f_{\tilde \eta}(u_1, u_2) \pB{\ind{\abs{u_1} \leq \zeta} \, \ind{\abs{u_2} > \zeta} + \ind{\abs{u_1} > \zeta} \, \ind{\abs{u_2} \leq \zeta} + \ind{\abs{u_1} > \zeta} \, \ind{\abs{u_2} > \zeta}}
\notag \\
&\;=\; O\pb{ \tilde \eta \log \tilde \eta^{-1} \, \ind{\abs{u_1} \leq \zeta} \, \ind{\abs{u_2} \leq \zeta}} + O\pbb{\frac{\tilde \eta \, \ind{\abs{u_1} \leq \zeta} \, \ind{\abs{u_2} > \zeta}}{u_2^2 + \zeta^2}}
\notag \\ \label{f-f_eta_decomp}
& \qquad
+ O\pbb{\frac{\tilde \eta \, \ind{\abs{u_1} > \zeta} \, \ind{\abs{u_2} \leq \zeta}}{u_1^2 + \zeta^2}}
+ O\pbb{\frac{\tilde \eta^2 \, \ind{\abs{u_1} > \zeta} \, \ind{\abs{u_2} > \zeta}}{(u_2^2 + \zeta^2)(u_1^2 + \zeta^2)}}
\end{align}
where in the second step we used that $\norm{f - f_{\tilde \eta}}_\infty = O(\tilde \eta \log \tilde \eta^{-1})$ since $f$ is bounded and Lipschitz, and that $f$ is bounded with compact support.

Plugging \eqref{f-f_eta_decomp} into \eqref{pkE_int_f} yields the decomposition $I = I_1 + I_2 + I_3 + I_4$ in self-explanatory notation. We estimate these terms using the eigenvalue rigidity estimates from Theorem \ref{thm:rig}, which imply that the event
\begin{equation} \label{rigi_corr}
\Xi \;\deq\; \hb{N \varrho_E \abs{\lambda_i - E} \;\geq\; c \, \abs{i - i_0} - N^{\xi/2} \text{ for all } i = 1, \dots, N}
\end{equation}
has probability at least $1 - N^{-3}$, for some constant $c > 0$ and $i_0 \in \{1, \dots, N\}$. We decompose the expectation on the right-hand side of \eqref{pkE_int_f} using $\ind{\Xi} + \ind{\Xi^c}$. The contribution of $\ind{\Xi^c}$ is easily bounded by $O(N^{-1})$ because $f$ is bounded.

In the following estimates it therefore suffices to estimate the expectation over the event $\Xi$. We find
\begin{equation*}
I_1 \;\leq\; O(N^{-1}) + O(\zeta^2 \tilde \eta  \log \tilde \eta^{-1}) \;=\; O(N^{-\xi})\,.
\end{equation*}
Moreover, we find
\begin{equation*}
I_2 \;=\; O(N^{-1}) + O \pbb{\sum_{i} \frac{\zeta \tilde \eta}{(i - i_0)^2 + \zeta^2}} \;=\; O(N^{-\xi})\,.
\end{equation*}
The terms $I_3$ and $I_4$ are estimated analogously. This concludes the proof.
\end{proof}

We conclude this subsection by remarking that the moment matching assumptions from Definition \ref{def:match} can be somewhat relaxed. Indeed, it is easy to check that Theorem \ref{thm:fourmoment} remains correct, with the same proof up to trivial adjustments, provided that, instead of the assumption that \eqref{moment_match} hold for $k + l \leq 4$, we assume that
\begin{equation*}
\E' (H_{ij})^k (\ol H_{ij})^l \;=\; \E'' (H_{ij})^k (\ol H_{ij})^l + O \pb{N^{(k+l)/2 - 2 + \delta_{ij} - \delta}}
\end{equation*}
for $k + l \leq 4$ for some constant $\delta > 0$. We leave the details of this extension to the interested reader.

\subsection{Green function comparison for the spectral edge and eigenvectors} \label{sec:GFC_edge_vec}

We conclude this section by sketching how the Green function comparison method can also be applied to the local eigenvalue statistics near the edge and to the distribution of the eigenvectors.

For the distribution of the extreme eigenvalues, we focus on the top edge $2$.
Based on Theorem \ref{thm:rig}, we expect the typical separation and scale of fluctuations of eigenvalues near the edge $2$ to be of order $N^{-2/3}$. For any fixed $\ell \in \N$, we define the random vector
\begin{equation*}
\f q_\ell \;\deq\; N^{2/3} (\lambda_1 - 2, \lambda_2 - 2, \dots, \lambda_\ell - 2)\,.
\end{equation*}
For the case of the GUE, the asymptotic behaviour of $\f q_\ell$ can be explicitly computed, and we have
\begin{equation} \label{edge_conv}
\f q_\ell \;\todist\; \f q_{\ell,\infty}
\end{equation}
as $N \to \infty$, where $\f q_{\ell,\infty}$ is the \emph{GUE-Tracy-Widom-Airy vector}. An analogous result holds for the GOE, with the \emph{GOE-Tracy-Widom-Airy} random vector as the limit.

The Tracy-Widom-Airy distributions appear in several contexts in probability theory, usually at the boundary between the two phases of weakly versus strongly coupled components in a system. For brevity, here we only focus on the GUE-Tracy-Widom-Airy vector for the scalar case $\ell = 1$. More general descriptions can be found in \cite{TW94,TW96,agz,RRV2011JAMS}. Let $F_2$ denote the cumulative distribution of the random variable $q_{1,\infty}$ from \eqref{edge_conv}. Then $F_2$ is given by the Fredholm determinant \bes\la{TWdef1} F_2(s) \;=\; \det(I-K_{\op{Ai}})_{L^2([s,+\infty))}\,,\ees
where  
$$K_{\op{Ai}}(x,y)\;\deq\; \frac{\op{Ai}(x)\op{Ai}'(y)-\op{Ai}'(x)\op{Ai}(y)}{x-y}\;=\; \int_0^{+\infty}\op{Ai}(x+t)\op{Ai}(y+t) \, \ud t$$is the  \emph{Airy kernel} and $\op{Ai}$ denotes the \emph{Airy function}. 
Alternatively, $F_2$ is also given by  \bes\la{TWdef2} F_2(s) \;=\;\exp\lf(-\int_s^{+\infty}(x-s)q^2(x)\ud x\ri),\ees where $q$ is the  solution of the \emph{Painlev\'e II} differential equation   $q''(x)=xq(x)+2q^3(x)$ \st  $q(x)/\op{Ai}(x)$ tends to $1$   as $x \to + \infty$.

The question of \emph{edge universality} is whether \eqref{edge_conv} holds for arbitrary Wigner matrices. It has seen much attention over the past 20 years, starting with the pioneering work \cite{SinaiSoshni,soshniCMP99,soshniJSP2002,SandrineSoshni},  where it was proved that \eqref{edge_conv} holds for any Wigner matrix whose entries have a symmetric distribution that matches the GUE to order two. 

The Green function comparison method can be used to establish \eqref{edge_conv} for arbitrary Wigner matrices that match the GUE to order two. This was proved in \cite{EYYrigi}, and a similar result holds for real symmetric Wigner matrices and the GOE. The basic idea of the proof is similar to that of Theorem \ref{thm:fourmoment}. Note, however, that instead of matching to order \emph{four} we only need to assume matching to order \emph{two}.

The matching to order four arose naturally in the proof of Theorem \ref{thm:fourmoment}, as the number $N^2$ of terms in the telescopic sum \eqref{GFC_telescope} was compensated by the size $N^{-5/2}$ of each term. The factor $N^{-5/2}$ arose from the four-moment matching assumption and the bound of order $1$ on the entries of the Green function. If we only assume matching to order two, a similar argument yields the bound $N^{-3/2}$ for each term of \eqref{GFC_telescope}, which is clearly far from enough. The key observation leading to the solution to this problem is that near the edge the entries of the Green function are in fact small. Indeed, in order to resolve the individual eigenvalues near the edge we take $\eta = N^{-2/3 - \epsilon}$ instead of $N^{-1 - \epsilon}$ as in Lemma \ref{lem:GFC}. For $\abs{E - 2} \leq N^{-2/3 + \epsilon}$ and $z = E + \ii \eta$, we find from the local semicircle law and the elementary bound \eqref{Immmupperlowerbounded} that
\begin{equation*}
\abs{G_{ij}(z)} \;\prec\; \delta_{ij} + N^{-1/3 + \epsilon}\,.
\end{equation*}
Thus, off-diagonal Green function entries provide additional smallness near the spectral edges. Exploiting this mechanism requires a careful analysis of the various monomials generated in expansions of the type \eqref{GFC_expansion}, classifying them in terms of the number of off-diagonal Green function entries. We refer to \cite{EYYrigi}  for the full details.

Finally, the Green function comparison method can also be used to analyse the distribution of eigenvectors: in \cite{KnowlesYinEig}, it was used to prove a universality result for the distribution of individual eigenvectors, stating that the distribution of eigenvectors of Wigner matrices is universal near the spectral edges, as well as in the bulk under an additional four-moment matching condition. Subsequently, this latter result was also obtained in \cite{Tao-Vu_vectors} using a different comparison method. Recently, in \cite{BourgadeYau}, the results of \cite{KnowlesYinEig} were extended to the bulk eigenvectors by combining the comparison results of \cite{KnowlesYinEig} with a novel eigenvector moment flow.

We now summarize the Green function comparison method for eigenvectors from \cite{KnowlesYinEig}.
In the case of the GUE, the distribution of eigenvectors is trivial: the law of the GUE is invariant under conjugation by unitary matrices, from which we deduce that the matrix of eigenvectors of the GUE is uniformly distributed on the group $\rr{SU}(N)$. This implies that any eigenvector $\f u_i = (u_i(k))_{k = 1}^N$ is uniformly distributed on the unit sphere of $\C^N$, and the random variables
\begin{equation} \label{evect_coord}
X_{ik} \;\deq\; \sqrt{N} u_i(k)
\end{equation}
are asymptotically independent standard complex normals.

Using Green function comparison, one can prove, just like for the eigenvalues, a four-moment theorem in the bulk and a two-moment theorem near the edge. Indeed, it was proved in \cite{KnowlesYinEig}
  that the distributions of a bounded number of eigenvector entries $(X_{i_1k_1}, \dots, X_{i_\ell k_\ell})$ for two different Wigner ensembles $E'$ and $E''$ coincide if $E'$ and $E''$ match to order four. Moreover, if $E'$ and $E''$ match to order two, then the distributions of $(X_{i_1k_1}, \dots, X_{i_\ell k_\ell})$ coincide near the edge, i.e.\ provided that the indices $i_1, \dots, i_\ell$ are bounded by $N^\epsilon$ for some small enough $\epsilon > 0$. In fact, using more refined techniques, it was proved in \cite[Section 7]{BKYYPCA}  that the distributions of $(X_{i_1k_1}, \dots, X_{i_\ell k_\ell})$ for two ensembles matching to order two coincide provided that $i_1, \dots, i_\ell \leq N^{1 - \epsilon}$ for any fixed $\epsilon > 0$.

The starting point of these proofs is to express the eigenvector components using the Green function. A helpful observation is that eigenvectors are only defined up to a global phase. To remove this ambiguity, it is therefore better to analyse the components of the rescaled \emph{spectral projections} $N \f u_i \f u_i^*$. They can be expressed, at least in an averaged form, in terms of the Green function using
\begin{equation} \label{GFC_evect}
\sum_{j} \frac{\eta / \pi}{(E - \lambda_j)^2 + \eta^2} \, N \f u_j \f u_j^* \;=\; \frac{N}{2 \pi \ii} \pb{G(E + \ii \eta) - G(E - \ii \eta)}\,.
\end{equation}
In order to extract $N \f u_i \f u_i$ from the identity \eqref{GFC_evect}, we have to choose $\eta$ to be much smaller than the typical eigenvalue separation near the eigenvalues $\lambda_i$, and integrate \eqref{GFC_evect} in $E$ over a random interval $I$ centred around $\lambda_i$, whose length is much larger than $\eta$, such that all eigenvalues $\lambda_j$, $j \neq i$, are far from $I$ on the scale $\eta$. Such an interval can be showed to exist with high probability \cite{KnowlesYinEig}. Then the Green function comparison may be applied to integrals of entries of the right-hand side of \eqref{GFC_evect}. We refer to \cite{KnowlesYinEig,BKYYPCA} for the details.

\section{Outlook: some further developments} \label{sec:outlook}

In this concluding section we survey further developments of local laws beyond the simple case of Wigner matrices presented in these notes.

\subsection{Isotropic laws} \label{sec:isotropic}
The local law from Theorem \ref{Th3} states roughly that the matrices $G$ and $m I$ are close with high probability for $z \in \f S$. As the dimension $N$ of these matrices is very large, which notion of closeness to use is a nontrivial question. Theorem \ref{Th3} establishes closeness in the sense of individual matrix entries. Other canonical (and stronger) notions are closeness in the \emph{weak operator sense}, \emph{strong operator sense}, and \emph{norm sense}, which entail controlling
\begin{equation*}
\scalar{\f w}{(G - mI) \f v} \,, \qquad
\abs{(G - mI) \f v}\,, \qquad \norm{G - mI}
\end{equation*}
respectively, for deterministic $\f v, \f w \in \C^N$. It is easy to see that already convergence in the strong operator sense must fail: $\abs{(G - mI)\f v}$ cannot be small for $z \in \f S$ with high probability. To see this, write
\begin{equation*}
\abs{(G - mI) \f v} \;=\; \sup_{\abs{\f w} \leq 1} \abs{\scalar{\f w}{(G - mI) \f v}} \;\geq\; \abs{\scalar{\f u_i}{\f v}} \absbb{\frac{1}{\lambda_i - z} - m}\,,
\end{equation*}
for any $i = 1, \dots, N$, where in the last step we chose $\f w = \f u_i$ to be an eigenvector of $H$. Suppose that $z = E + \ii \eta$ with $E$ in the bulk spectrum and $\eta \geq N^{-1 + \tau}$. Then from Theorem \ref{thm:rig} we get that with high probability there exists $i \equiv i(E) = 1, \dots, N$ such that $\abs{\lambda_i - z} \leq 2 \eta$. Thus,
\begin{equation*}
\abs{(G - mI) \f v} \;\geq\; \frac{c \abs{\scalar{\f u_i}{\f v}}}{\eta}\,.
\end{equation*}
Since $\f v$ is deterministic, the numerator is at least of order $N^{-1/2}$ with probability $1 - o(1)$; for instance in the case of GUE this is trivial because $\f u_i$ is uniformly distributed on the unit sphere. We conclude that $\abs{(G - mI) \f v}$ is with probability $1 - o(1)$ at least of order $N^{-1/2} \eta^{-1}$, which is much bigger than one for $\eta \ll N^{-1/2}$.

As it turns out, $G$ and $m I$ are close in the weak operator sense, and \eqref{121414h54} generalizes to
\begin{equation} \label{isotropc_ll}
\scalar{\f w}{G(z) \f v} - m(z) \scalar{\f w}{\f v} \;=\; O_\prec(\Psi(z))
\end{equation}
uniformly\footnote{We emphasize that the estimate is uniform in $\f v, \f w$ in the sense of Definition \ref{def:stochdom}, but of course not \emph{simultaneous} in $\f v, \f w$. Indeed, a simultaneous statement would imply closeness in the norm sense, which is wrong as remarked above.} for all deterministic unit vectors $\f v, \f w \in \C^N$ and $z \in \f S$. Such a result is called an \emph{isotropic local law}, and was first obtained for Wigner matrices in \cite{KnowlesYinIso} and subsequently generalized to other models in \cite{BEKYYPCA}.  The isotropic local law is a nontrivial extension of Theorem \ref{Th3}, and both proofs of \cite{KnowlesYinIso,BEKYYPCA}  in fact rely crucially on the entrywise estimate from \eqref{121414h54} as input.

Isotropic local laws have proved very useful in the analysis of deformed random matrix ensembles \cite{KnowlesYinIso,BKYYPCA}  and the distribution of eigenvectors \cite{BourgadeYau}.   Here we only mention an immediate consequence: \emph{isotropic delocalization} which extends the estimate from Theorem \ref{thm:deloc} to $\abs{\scalar{\f v}{\f u_i}}^2 \prec N^{-1}$ for all $i = 1, \dots, N$ and all deterministic unit vectors $\f v$. The proof is the same as that of Theorem \ref{thm:deloc}, using \eqref{isotropc_ll} instead of Theorem \ref{Th3}.

\subsection{Beyond Wigner matrices} \label{sec:beyondW}
In these notes we only consider Wigner matrices from Definition \ref{def:Wigner}. Although Wigner matrices occupy a central place in random matrix theory, there has recently been much interest in more general random matrix ensembles. One natural way to generalize Definition \ref{def:Wigner} is to admit matrix entries with general variances. Thus, we consider a matrix $H$ as in Definition \ref{def:Wigner}, except that we replace (ii) and (iii) with the following conditions.
\begin{itemize}
\item[(ii')]
For all $i,j$ we have $E H_{ij} = 0$ and $\E \abs{H_{ij}}^2 = S_{ij}$ for some matrix $S = (S_{ij})$.
\item[(iii')]
If $S_{ij} > 0$ then $S_{ij}^{-1/2} H_{ij}$ is bounded in any $L^p$ space, uniformly in $N,i,j$.
\end{itemize}
A common assumption on the matrix $S$ is that it be (doubly) stochastic:
\begin{equation} \label{S_stoch}
\sum_{j} S_{ij} \;=\; 1
\end{equation}
for all $i$. The assumption \eqref{S_stoch} guarantees that the limiting behaviour of $H$ is still governed by the semicircle distribution \cite{MPasturK,BMPastur}. Two important classes of random matrices satisfying \eqref{S_stoch} are the following.
\begin{itemize}
\item
\emph{Generalized Wigner matrix.} We require \eqref{S_stoch} and that $N S_{ij} \asymp 1$ uniformly in $N,i,j$.
\item
\emph{Band matrix.} Let $f$ be a symmetric probability density on $\R$, called the \emph{band profile function}, and $W \in [1,N]$ be the \emph{band width}. Then we set
\begin{equation*}
S_{ij} \;\deq\; \frac{1}{Z} f \pbb{\frac{[i - j]_N}{W}}\,,
\end{equation*}
where $[i]_N \deq (i + N \Z) \cap [-N/2,N/2)$ is the periodic representative of $i \in \Z$ in $\Z \cap [-N/2,N/2)$, and $Z > 0$ is  chosen so that \eqref{S_stoch} holds.
\end{itemize}

All of the results in these notes also hold for generalized Wigner matrices; see e.g. \cite{EYYrigi,EKYY4}
  for more details. The proofs for generalized Wigner matrices are similar to the ones presented in these notes. The major difference is that in the proof of Theorem \ref{Th3} one has to split the space $\C^N$ into the one-dimensional space spanned by $\f e \deq N^{-1/2} (1,1,\dots, 1)^*$ and its orthogonal complement. In the direction of $\f e$, the analysis is identical to that of Sections \ref{sec:weak_law}--\ref{sec:opt_bounds} thanks to the assumption \eqref{S_stoch}. On the orthogonal complement of $\f e$, the key extra difficulty is to obtain good bounds on the deterministic operator $(1 - m^2 S)^{-1}$. Near the spectral edge we have $m^2 \approx 1$ so that this operator is singular because, by \eqref{S_stoch}, $1$ is an eigenvalue of $S$ with eigenvector $\f e$. However, thanks to the assumptions $N S_{ij} \asymp 1$ it is not hard to see that $(1 - m^2 S)^{-1}$ is in fact regular on the orthogonal complement of $\f e$. We refer to \cite{EKYY4}  for the full details.

Band matrices are a much more challenging, and interesting, model. They arise from solid state physics as an alternative to the Anderson model to describe the physics of a disordered quantum Hamiltonian. So far an optimal local law and eigenvector delocalization are open problems, although partial progress has been made in \cite{EKQD,EKQDGeneral,EKYY4,EKYY3, BaoErdosblockband}.

If the condition \eqref{S_stoch} does not hold, then the asymptotic eigenvalue distribution of $H$ is no longer governed by the semicircle distribution, but by a much more complicated distribution that depends on $S$. This scenario was recently investigated in depth in \cite{AjankiErdosKruger,AjankiErdosKruger2},  where in particular a local law was established by a nontrivial extension of the methods presented in these notes, involving a stability analysis of a nonlinear quadratic vector equation.

Next, we say a few words about random matrices that are not of the above form but can nevertheless be studied using variants of the method presented in these notes. An important example is the \emph{sample covariance matrix} $H \deq X X^*$, where $X$ is an $M \times N$ matrix and $M,N \to \infty$ simultaneously. If the entries of $X$ are independent with mean zero and variance $N^{-1}$ then the asymptotic eigenvalue distribution of $H$ is governed by the Marchenko-Pastur distribution \cite{MP1967,bai-silver-book} instead of the semicircle distribution, and the local law was established in \cite{ESYY, PY1,SchleinCirc}. If the entries of $X$ have a nontrivial covariance structure $\Sigma \deq \E H$, the asymptotic eigenvalue distribution becomes much more complicated and depends on $\Sigma$, and the basic strategy from these notes cannot be applied. A local law for such matrices was established in great generality in \cite{KnowlesYinAnIso} using a new self-consistent comparison method.

We remark that a generalization of the methods of \cite{ESYY, PY1} for sample covariance matrices may also be used to obtain local laws for non-Hermitian matrices with independent entries, using Girko's Hermitization trick \cite{Girko}; see \cite{BourgadeCirc1, BourgadeCirc2, YinCirc} for more details.

Local laws have also been derived for matrix models describing the free additive convolution \cite{KarginPTRF12,KarginAOP13,KarginAOP13Sub,OrourkeVu,BG:SRTLL,BaoErdosSchnelli1,BaoErdosSchnelli2} and, using different techniques from the ones presented in these notes, for general $\beta$-ensembles (or one-dimensional log-gases) \cite{BEY1,BEY2,BEY3}.  We also mention that a local law was proved for Wigner matrices with heavy tailed entries in \cite{CharlesAlice}.

Finally, we note that local laws have been established for \emph{sparse} random matrices, where most entries of $H$ are with high probability zero. Such matrices typically arise as adjacency or Laplacian matrices of random graphs. A local law for the Erd\H{o}s-R\'enyi graph was derived in \cite{EKYY1}. Because the entries of the adjacency matrix of the Erd\H{o}s-R\'enyi graph are independent up to the symmetry condition, the basic structure of the proof presented in these notes remains applicable. This is stark contrast to the case of \emph{random regular graphs}, where the entries of the adjacency matrix exhibit strong nonlocal correlations and whole setup of the proof of Theorem \ref{Th3} breaks down. For this case, a local law was recently derived in \cite{LocalLawRegGraphs}  using a new local resampling method.

\appendix

\section{Schur's complement formula and proof of Lemma \ref{lem8}} \label{sec:schur}

\begin{lem}[Schur's complement formula] \label{lem:schur}
Provided all of the following inverse matrices exist, the inverse of a block matrix is given by
\begin{align*}
\begin{pmatrix}
A & B \\ C & D
\end{pmatrix}^{-1}
&\;=\;
\begin{pmatrix}
\cal A^{-1} & - \cal A^{-1} B D^{-1}
\\
-D^{-1} C \cal A^{-1} & D^{-1} C \cal A^{-1} B D^{-1} + D^{-1}
\end{pmatrix}
\\
&\;=\;
\begin{pmatrix}
A^{-1} B \cal D^{-1} C A^{-1} + A^{-1} & - A^{-1} B \cal D^{-1}
\\
- \cal D^{-1} C A^{-1} & \cal D^{-1}
\end{pmatrix}\,,
\end{align*}
where we defined
\begin{equation*}
\cal A \;\deq\; A - B D^{-1} C \,, \qquad \cal D \;\deq\; D - C A^{-1} B\,.
\end{equation*}
\end{lem}

\begin{proof}
This is routine verification. A good way of arriving at these formulas is to use Gaussian elimination to write
\begin{align*}
\begin{pmatrix}
A & B \\ C & D
\end{pmatrix}
&\;=\;
\begin{pmatrix}
1 & BD^{-1} \\ 0 & 1
\end{pmatrix}
\begin{pmatrix}
A - B D^{-1} C & 0
\\
0 & D
\end{pmatrix}
\begin{pmatrix}
1 & 0 \\ D^{-1} C & 1
\end{pmatrix}
\\
&\;=\;
\begin{pmatrix}
1 & 0 \\ C A^{-1} & 1
\end{pmatrix}
\begin{pmatrix}
A & 0 \\ 0 & D - C A^{-1} B
\end{pmatrix}
\begin{pmatrix}
1 & A^{-1} B \\ 0 & 1
\end{pmatrix}\,,
\end{align*}
from which the claim follows trivially.
\end{proof}

\begin{proof}[Proof of Lemma \ref{lem8}]
Without loss of generality, we take $T = \emptyset$. (Otherwise consider the matrix $H^{(T)}$ instead of $H$.) It suffices to prove the claim for $G = M^{-1}$ for a general matrix $M$. As in Definition \ref{def: minors}, we use the notation $M^{(k)} = (M_{ij})_{i,j \in \{1, \dots, N\} \setminus \{k\}}$.

First, from Schur's complement formula for the off-diagonal blocks we get for $i \neq j$
\begin{equation} \label{RI2M}
G_{ij} \;=\; - G_{ii} \sum_{k}^{(i)} M_{ik} G_{kj}^{(i)} \;=\; - G_{jj} \sum_{k}^{(j)} G_{ik}^{(j)} M_{kj}\,,
\end{equation}
which immediately yields \eqref{RI2}.

It remains to prove \eqref{RI1}; there are several approaches. Perhaps the most natural one is based on the Woodbury matrix identity
\begin{equation*}
(A + UBV)^{-1} = A^{-1} - A^{-1} U \pb{B^{-1} + V A^{-1} U}^{-1} V A^{-1}\,,
\end{equation*}
which can itself be deduced for instance using the resolvent identity $Y^{-1} = X^{-1} + X^{-1} (X - Y) Y^{-1}$. Let $\f e_k$ be the standard $k$-th column unit basis vector and define the matrices
\begin{equation*}
M^{[k]}_{ij} \deq M_{ij} \ind{i \neq k} \ind{j \neq k}\,, \qquad U \deq
\begin{pmatrix}
H \f e_k & \f e_k
\end{pmatrix}
\,, \qquad
V \deq
\begin{pmatrix}
\f e_k^* \\ \f e_k^* H
\end{pmatrix}\,,
\end{equation*}
so that we have
\begin{equation*}
M = M^{[k]} + UV - M_{kk} \f e_k \f e_k^*\,.
\end{equation*}
Then by the Woodbury matrix identity we have
\begin{equation*}
\pb{M^{[k]} - M_{kk} \f e_k \f e_k}^{-1} = G + G U (I - VGU)^{-1} VG\,.
\end{equation*}
A straightforward calculation yields
\begin{equation*}
I - VGU = -
\begin{pmatrix}
0 & G_{kk}
\\
M_{kk} & 0
\end{pmatrix}\,, \qquad
(I - VGU)^{-1} = - \frac{1}{M_{kk} G_{kk}}
\begin{pmatrix}
0 & G_{kk}
\\
M_{kk} & 0
\end{pmatrix}\,,
\end{equation*}
from which we deduce, after a short calculation,
\begin{equation*}
\pb{M^{[k]} - M_{kk} \f e_k \f e_k}^{-1} = G - \frac{1}{M_{kk}} \f e_k \f e_k^* - \frac{1}{G_{kk}} G \f e_k \f e_k^* G\,.
\end{equation*}
We conclude
\begin{equation*}
(M^{[k]})^{-1} = G - \frac{1}{G_{kk}} G \f e_k \f e_k^* G\,,
\end{equation*}
from which \eqref{RI1} follows.

For a second proof of \eqref{RI1}, write, for $k \neq j$,
\begin{equation*}
\sum_{l}^{(k)} M_{kl} \pbb{G_{lj} - 
\frac{G_{lk} G_{kj}}{G_{kk}}}
\;=\; - M_{kk} G_{kj} - \frac{G_{kj}}{G_{kk}} (1 - M_{kk} G_{kk}) \;=\; - \frac{G_{kj}}{G_{kk}}
\;=\;
\sum_{l}^{(k)} M_{kl} G_{lj}^{(k)}\,,
\end{equation*}
where in the last step we used \eqref{RI2M}. Taking the partial derivative $\frac{\partial}{\partial M_{ki}}$ of this identity (using the resolvent expansion to differentiate the entries of $G$) with $i \neq k$ yields
\begin{equation*}
G_{ij} -  \frac{G_{ik} G_{kj}}{G_{kk}} 
\;=\; G_{ij}^{(k)}\,,
\end{equation*}
which is \eqref{RI1}.

A third proof of \eqref{RI1} can be be obtained by explicit matrix inversion for $N = 3$, and then extended to arbitrary $N > 3$ using Schur's complement formula.
\end{proof}

\section{Basic properties of the semicircle distribution and proof of Lemma \ref{lem13}} \label{sec:stability}

For $k\ge 0$ define the \emph{$k$-th Catalan number} $\cal C_k\deq\ff{k+1}\binom{2k}{k}$.

\beg{lem}[Moments of the semicircle distribution and Catalan numbers]\la{lem:Catalan}For each $k\ge 0$, we have $$\ff{2\pi}\int_{-2}^2x^k\sqrt{4-x^2} \, \dd x \;=\; \beg{cases} \ds \cal C_{k/2}&\txt{ if $k$ is even}\\
0&\txt{ if $k$ is odd}\,.
\end{cases}$$
\en{lem}

\bpr This is an easy computation using the change of variables $x = 2 \cos \theta$ with $\theta \in [0,\pi]$.
\epr

\beg{lem}[Stieltjes transform of the semicircle distribution]\la{lem:stieltjessclformula}The Stieltjes transform $m(z)$  of the  semicircle distribution satisfies the relation \eqref{m_id}, and is explicitly given by the formula \eqref{m_qsolution}.
\end{lem}

\bpr
We prove \eqref{m_qsolution}, from which \eqref{m_id} follows immediately. Since both $m(z)$ and the right-hand side of \eqref{m_qsolution} are analytic for $z \in \C \setminus [-2,2]$, it suffices to prove \eqref{m_qsolution} for $z$ in some subset of $\C$ with a limit point. In fact, we establish \eqref{m_qsolution} for $\abs{z} > 2$.
In this case,  by Lemma \ref{lem:Catalan},  $$m(z) \;=\; -z^{-1}\int \ff{1-(x/z)} \, \varrho(\ud x) \;=\; -z^{-1}\sum_{n\ge 0}\cal C_nz^{-2n}\,,$$
where the power series is absolutely convergent. Now \eqref{m_qsolution} follows directly from the series expansion of the right-hand side of \eqref{m_qsolution}.\epr

\bpr[Proof of Lemma \ref{lem13}] We have $$\tilde{m}\;=\; \frac{-z-\sqrt{z^2-4}}{2}\,,$$ with the conventions for the square root introduced after \eqref{m_qsolution}. From this we get
\begin{equation*}
\abs{m - \tilde m} \;=\; \absb{\sqrt{z^2 - 4}} \;=\; \sqrt{\abs{z - 2} \abs{z + 2}} \;\asymp\; \sqrt{\kappa + \eta}\,,
\end{equation*}
where in the last step we used that $\abs{z} \leq C$ for $z \in \f S$.

Next, if $u$ is a solution of \eqref{u_eq} then 
$u\in \{m_r, \tilde{m}_r\}$, where we defined
$$m_r\;\deq\;\frac{-z + \sqrt{z^2-4(1+r)}}{2}\,, \qquad \tilde m_r\;\deq\;\frac{-z-\sqrt{z^2-4(1+r)}}{2}\,.$$
(Note that the choice of the branch of the square root is immaterial.) Note that for any complex square root $\sqrt{\cdot}$ and $w,\zeta \in \C$ we have
\begin{equation*}
\abs{\sqrt{w + \zeta} - \sqrt{w}} \wedge \abs{\sqrt{w + \zeta} + \sqrt{w}} \leq \frac{\abs{\zeta}}{\sqrt{\abs{w}}} \wedge \sqrt{\abs{\zeta}}\,.
\end{equation*}
Hence \eqref{u-m_eq} follows from $\abs{\sqrt{z^2 - 4}} \asymp \sqrt{\kappa + \eta}$.
\epr

\section{The Helffer-Sj\"ostrand formula} \label{sec:HS}

\begin{propo}[Helffer-Sj\"ostrand formula] \label{prop:HS}
Let $n \in \N$ and $f \in \cal C^{n+1}(\R)$. We define the \emph{almost analytic extension of $f$ of degree $n$} through
\begin{equation*}
\tilde f_n(x + \ii y) \;\deq\; \sum_{k = 0}^n \frac{1}{k !} (\ii y)^k f^{(k)}(x)\,.
\end{equation*}
Let $\chi \in \cal C^\infty_c(\C;[0,1])$ be a smooth cutoff function. Then for any $\lambda \in \R$ satisfying $\chi(\lambda) = 1$ we have
\begin{equation*}
f(\lambda) \;=\; \frac{1}{\pi} \int_\C \frac{\bar \partial (\tilde f_n(z) \chi(z))}{\lambda - z} \, \dd^2 z\,,
\end{equation*}
where $\dd^2 z$ denotes the Lebesgue measure on $\C$ and $\bar \partial \deq \frac{1}{2} (\partial_x + \ii \partial_y)$ is the antiholomorphic derivative.
\end{propo}

Applied to the matrix $H$, the Helffer-Sj\"ostrand formula gives
\begin{equation} \label{HS_rep}
f(H) \;=\; \frac{1}{\pi} \int_\C \bar \partial (\tilde f_n(z) \chi(z)) \, G(z) \, \dd^2 z\,,
\end{equation}
provided $\chi$ is chosen so that $\chi = 1$ on $\spec(H)$. Compared to the functional calculus from \eqref{cont_rep}, the Helffer-Sj\"ostrand representation \eqref{HS_rep} is more versatile and general. In particular, we only require that $f$ be differentiable instead of analytic. In most applications, however, we shall require some additional smoothness of $f$ (typically $f \in \cal C^2(\R)$), since the integral on the right-hand side of \eqref{HS_rep} is more stable for small $\eta$ if $n$ is large. This generality is of great use in applications, since it allows one to choose $f$ to be a (smoothed) indicator function, which is not possible with \eqref{cont_rep}.

\begin{proof}[Proof of Proposition \ref{prop:HS}]
Abbreviate $g(z) \deq \tilde f_n(z) \chi(z)$. By assumption, we have $g(\lambda) = f(\lambda)$ for $\lambda \in \R$ satisfying $\chi(\lambda) = 1$. Choose $r > 0$ large enough that $g = 0$ outside the ball $B_r(0)$. For $\epsilon > 0$ define the region $D_\epsilon \deq B_r(0) \setminus B_\epsilon(\lambda)$. Then we have
\begin{equation} \label{0_HS_split}
0 \;=\; \oint_{\partial B_r(0)} \frac{g(z)}{\lambda -z} \, \dd z \;=\; \int_{\partial D_\epsilon} \frac{g(z)}{\lambda -z} \, \dd z + \oint_{\partial B_\epsilon(\lambda)} \frac{g(z)}{\lambda -z}\,.
\end{equation}
By Green's formula, the first term of \eqref{0_HS_split} is equal to
\begin{equation*}
\int_{\partial D_\epsilon} \frac{g(z)}{\lambda -z} \, \dd z \;=\; \int_{D_\epsilon} \dd \pbb{\frac{g(z)}{\lambda -z}} \wedge \dd z \;=\; 
\int_{D_\epsilon} \bar \partial \pbb{\frac{g(z)}{\lambda -z}} \dd \bar z \wedge \dd z \;=\; \int_{D_\epsilon}  \frac{\bar \partial g(z)}{\lambda -z} \, 2 \ii \,\dd^2 z\,,
\end{equation*}
where in the last step we used that $(\lambda - z)^{-1}$ is holomorphic in $D_\epsilon$. Moreover, by Cauchy's theorem the second term of \eqref{0_HS_split} is equal to $-2 \pi \ii g(\lambda) + o(1)$ as $\epsilon \downarrow 0$, since $g$ is continuous. We conclude that
\begin{equation*}
g(\lambda) \;=\; \frac{1}{\pi} \int_{D_\epsilon}  \frac{\bar \partial g(z)}{\lambda -z} \,\dd^2 z + o(1)
\end{equation*}
as $\epsilon \downarrow 0$.
Letting $\epsilon \downarrow 0$ completes the proof.
\end{proof}

\section{Multilinear large deviation estimates: proof of Lemma \ref{lem9}} \label{sec:LDE}
We first recall the following  version of the Marcinkiewicz-Zygmund inequality.
\begin{lem}\label{lm:MZ}
Let $X_1, \dots, X_N$ be a family of independent random variables each satisfying \eqref{cond on X} and suppose that the family $(b_i)$ is deterministic. Then
\be\label{MZ}
\normbb{\sum_i b_i X_i}_p \;\leq\; (Cp)^{1/2} \mu_p \pbb{\sum_i \abs{b_i}^2}^{1/2}
\ee
\end{lem}

\begin{proof}
The proof is a simple application of Jensen's inequality. Writing $B^2 \deq \sum_{j} \abs{b_i}^2$,
we get, by the classical Marcinkiewicz-Zygmund inequality \cite{stroock} in the first line, that
\begin{align*}
\normbb{\sum_i b_i X_i}_p^p &\;\leq\; (C p)^{p/2} \, \normbb{\pbb{\sum_i \abs{b_i}^2 \abs{X_i}^2}^{1/2}}_p^{p}
\\
&\;=\; (Cp)^{p/2} B^p \, \E \qbb{\pbb{\sum_i \frac{\abs{b_i}^2}{B^2} \abs{X_i}^2}^{p/2}}
\\
&\;\leq\; (C p)^{p/2} B^p \, \E \qbb{\sum_i \frac{\abs{b_i}^2}{B^2} \abs{X_i}^p}
\\
&\;\leq\; (Cp)^{p/2} B^p \mu_p^p\,. \qedhere
\end{align*}
\end{proof}

Next, we prove the following intermediate result.

\begin{lem}\label{lm:aXY}
Let $X_1, \dots, X_N, Y_1, \dots, Y_N$ be independent random variables each satisfying \eqref{cond on X}, and suppose that the family $(a_{ij})$ is deterministic. Then for all $p \geq 2$ we have
\begin{equation*}
\normbb{\sum_{i,j} a_{ij} X_i Y_j}_p \;\leq\; C p \, \mu_p^2 \pbb{\sum_{i,j} \abs{a_{ij}}^2}^{1/2}\,.
\end{equation*}
\end{lem}
\begin{proof}
Write
\begin{equation*}
\sum_{i,j} a_{ij} X_i Y_j \;=\; \sum_j b_j Y_j \,, \qquad b_j \;\deq\; \sum_i a_{ij} X_i\,.
\end{equation*}
Note that $(b_j)$ and $(Y_j)$ are independent families. By conditioning on the family $(b_j)$, we therefore get from Lemma \ref{lm:MZ}
and the triangle inequality that
\begin{equation*}
\normbb{\sum_j b_j Y_j}_p \;\leq\; (C p)^{1/2} \, \mu_p \normbb{\sum_j \abs{b_j}^2}_{p/2}^{1/2} \;\leq\; (C p)^{1/2} \, \mu_p \pbb{\sum_j \norm{b_j}_p^2}^{1/2}\,.
\end{equation*}
Using Lemma \ref{lm:MZ} again, we have
\begin{equation*}
\norm{b_j}_p \;\leq\;  (C p)^{1/2} \, \mu_p \pbb{\sum_i \abs{a_{ij}}^2}^{1/2}\,.
\end{equation*}
This concludes the proof.
\end{proof}

\begin{lem}\label{lm:aXX}
Let $X_1, \dots, X_N$ be independent random variables each satisfying \eqref{cond on X}, and suppose that the family $(a_{ij})$ is deterministic.
Then we have
\begin{equation*}
\normbb{\sum_{i \neq j} a_{ij} X_i X_j}_p \;\leq\; C p \, \mu_p^2  \pbb{\sum_{i \neq j} \abs{a_{ij}}^2}^{1/2}\,.
\end{equation*}
\end{lem}

\begin{proof}
The proof relies on the identity (valid for $i \neq j$)
\begin{equation} \label{identity for LDE}
1 \;=\; \frac{1}{Z_N} \sum_{I \sqcup J = \qq{N}} \ind{i \in I} \ind{j \in J}\,,
\end{equation}
where the sum ranges over all partitions of $\qq{N} = \{1, \dots, N\}$ into two sets $I$ and $J$, and $Z_N \deq 2^{N - 2}$ is independent of $i$ and $j$. Moreover, we have
\begin{equation} \label{comb bound}
\sum_{I \sqcup J = \qq{N}} 1 \;=\; 2^{N} - 2\,,
\end{equation}
where the sum ranges over nonempty subsets $I$ and $J$.
Now we may estimate
\begin{equation*}
\normbb{\sum_{i \neq j} a_{ij} X_i X_j}_p \;\leq\; \frac{1}{Z_N} \sum_{I \sqcup J = \qq{N}} \normbb{\sum_{i \in I} \sum_{j \in J} a_{ij} X_i X_j}_p \;\leq\; \frac{1}{Z_N} \sum_{I \sqcup J = \qq{N}} C p \, \mu_p^2  \pbb{\sum_{i \neq j} \abs{a_{ij}}^2}^{1/2}\,,
\end{equation*}
where we used that, for any partition $I \sqcup J = \qq{N}$, the families $(X_i)_{i \in I}$ and $(X_j)_{j \in J}$ are independent, and hence the Lemma \ref{lm:aXY} is applicable. The claim now follows from \eqref{comb bound}.
\end{proof}

Note that the proof of Lemma \ref{lm:aXX} may be easily extended to multilinear expressions of the form $\sum_{i_1, \dots, i_k}^* a_{i_1 \dots i_k} X_{i_1} \cdots X_{i_k}$. We shall not pursue such extensions here.
 
We may now complete the proof of Lemma \ref{lem9}.

\begin{proof}[Proof of Lemma \ref{lem9}]
The proof is a simple application of Markov's inequality. Part (i) follows from Lemma \ref{lm:MZ}, part (ii) from Lemma \ref{lm:aXX}, and part (iii) from Lemma \ref{lm:aXY}. We give the details for part (iii).

For $\epsilon >0$ and $D>0$ we have
\begin{align*}
 \P \qBB{ \absbb{\sum_{i \neq j} a_{ij} X_i X_j} \geq N^\epsilon\Psi}
& \;\leq\; \P\qBB{ \absbb{\sum_{i \neq j} a_{ij} X_i X_j} \geq N^\epsilon\Psi \,,\, 
 \pbb{\sum_{i \neq j} \abs{a_{ij}}^2}^{1/2}\leq N^{\epsilon/2}\Psi}
\\
&\qquad + \P \qBB{\pbb{\sum_{i \neq j} \abs{a_{ij}}^2}^{1/2} \geq N^{\epsilon/2}\Psi} \\
&\;\leq\;
  \P \qBB{\absbb{ \sum_{i \neq j} a_{ij} X_i X_j} \geq N^{\epsilon/2} 
\pbb{\sum_{i \neq j} \abs{a_{ij}}^2}^{1/2}} + N^{-D-1}
\\
&\;\leq\; \pbb{\frac{Cp \mu_p^2}{N^{\epsilon/2}}}^p + N^{-D-1}
\end{align*}
for arbitrary $D$. In the second step we used the definition of $\pb{\sum_{i\ne j} |a_{ij}|^2}^{1/2} \prec \Psi$ with parameters $\epsilon/2$ and $D+1$.
In the last step we used Lemma \ref{lm:aXX} by conditioning on $(a_{ij})$. Given $\epsilon$ and $D$, there is a large enough $p$ such that the first term on the last line is bounded by $N^{-D - 1}$. Since $\epsilon$ and $D$ were arbitrary, the proof is complete.

The claimed uniformity in $u$ in the case that $a_{ij}$ and $X_i$ depend on an index $u$ also follows from the above estimate.
\end{proof}

\section{Proof of Lemma \ref{thm:global_law}} \label{sec:PGL}

 \bpr[Proof of Lemma \ref{thm:global_law}] 
 A preliminary remark  is that by \eqref{6121414h44} (and its natural extension to the set $\mathcal{M}_{\le 1}(\R)$ of nonnegative measures on $\R$ with total mass at most  $1$), any element of $\mathcal{M}_{\le 1}(\R)$ is determined by its Stieltjes transform.

 Note  that for any fixed $z\in \C_+$, the function which maps a probability measure on $\R$ to its Stieltjes transform evaluated at $z$ is continuous in the topology of convergence in distribution. As a consequence, if $\mu \todist \varrho$ in probability, then $s(z)\to m(z)$ in probability. 
 
 To prove the converse, let us first treat the deterministic case. Suppose $\mu\equiv \mu_N$ to be deterministic. Then by Helly's theorem, from any of its subsequences, one can extract a subsequence having a limit $\tta$ in   $\mathcal{M}_{\le 1}(\R)$ for the vague topology (i.e.\ the topology generated by   integration of continuous functions with null limits at infinity). By hypothesis, $\tta$ must have Stieltjes transform $m$, hence be equal to $\varrho$. This proves that $\mu$ converges to $\varrho$ for the vague topology, hence for the convergence in distribution because $\mu$ and $\varrho$ are both \pro measures. 
 
 Let us now prove the converse in the general case where $\mu\equiv \mu_N$ is random. Let $(z_n)$ be a sequence of pairwise distinct elements of $\C_+$ having a limit in $\C_+$. By what precedes, the distance $\op{d}_S$ on the set $\mc{M}_1(\R)$ of \pro measures on $\R$ defined by $$ \op{d}_S(\nu,\nu')\; \deq\; \sum_{n} 2^{-n}S_{\nu-\nu'}(z_n)\qquad\qquad \trm{with }S_{\nu-\nu'}(z)\deq\int\frac{\ud(\nu-\nu')(t)}{t-z}$$ defines the topology of convergence in distribution. By hypothesis,  $\op{d}_S(\mu,\varrho)$ tends to zero in probability, which closes the proof of the theorem.

 Let us also give a more direct and  more conceptual proof of the converse. 
 For the vague topology defined above, the set $\mathcal{M}_{\le 1}(\R)$ is a compact Polish space, hence so is the set  $\mathcal{M}_1(\mathcal{M}_{\le 1}(\R))$ of \pro measures on $\mathcal{M}_{\le 1}(\R)$ endowed with the topology of the convergence in distribution. Denoting by $P\equiv P_N\in \mathcal{M}_1(\mathcal{M}_{\le 1}(\R))$ the distribution of $\mu$, the statement we have to prove is that $P\,\to\, \del_\varrho$. Let $P_\infty\in \mathcal{M}_1(\mathcal{M}_{\le 1}(\R))$ be an accumulation point of $P$. By hypothesis, a $P_\infty$-distributed random element of $\mathcal{M}_{\le 1}(\R)$ has Stieltjes transform equal to $m$ with \pro one, hence is  equal to $\varrho$ with \pro one. As the topology induced by the vague topology on $\mc{M}_1(\R)$ is the one of the convergence in distribution, this concludes the proof. \ 
 \epr

\section{Global semicircle law and proof of Theorem \ref{Th:Global_Law}}
Theorem \ref{Th:Global_Law} follows from the following more general result. 
\beg{Th}\la{Th:Global_Law_appendix} Suppose that the Hermitian matrix $H$ satisfies the following conditions. 
\begin{enumerate}
\item
The upper-triangular entries $( H_{ij} \col 1\le i\le j\le N)$ are independent.
\item
For all $i,j$ we have $\E  H_{ij}=0$ and for all $i\ne j$ we have $\E | H_{ij}|^2= N^{-1} $.   
\item  The sequence  $\del_N:=N\sup_{i,j}\E|H_{ij}|^3$ tends to zero as $N\to\infty$.  
\end{enumerate}  Then  for any fixed $z\in \C_+$  and $\eps>0$ there exist constants $C,c > 0$ \st 
$$\P(|s(z)-m(z)|\ge \eps)\;\le \; C\mathrm{e}^{-cN}\,.$$   \en{Th}

To prove this theorem, we shall use several preliminary lemmas. The first one is a decorrelation result, generalizing the well-known Stein lemma (obtained by a straightforward integration by parts) to non-Gaussian variables. 
\beg{lem}\la{lem:Stein} Let $f \in \cal C^2(\R^n;\R)$ be bounded with bounded first and second derivatives. Let $X$ be a centred random vector with values in $\R^n$ with finite second moment. Then there is $C\equiv C(n)$ \st $$\absb{\E f(X)X-\E \lan \nabla f(X),Y\rang Y} \; \le \; C\|\nabla^2 f\|_\infty \, \E |X|^3\,, $$ where $Y$ is an independent copy of $X$ and $\nabla^2 f$ is the Hessian of $f$.
\en{lem}

\bpr Note first that by the Taylor-Lagrange formula, for any $x$, $$|f(x)-(f(0)+\lan \nabla f(0),x\rang)|\;\le\; \|\nabla^2f\|_\infty |x|^2\,,$$ so it suffices to prove that $$
\absb{\E (f(0)+\lan \nabla f(0),Y\rang )Y-\E \lan \nabla f(X),Y\rang Y} \; \le \; C\|\nabla^2f\|_\infty \E |X|^3\, . $$
Since $X$ is centred, we can omit the term $f(0)Y$, and it suffices to prove that $$   | \nabla f(0) -  \nabla f(X)| \; \le \; C\|\nabla^2f\|_\infty\,  |X|\, ,$$ which follows from the mean-value theorem.
\epr

This lemma implies easily that for $H$ a centred complex random variable ($n=2$)  with finite third moment and $g$ a 
   $\mc{C}^2$ function on $\C$ which is bounded and has bounded first and second derivatives, we have \be\la{eq:SteinC}
 \lf|\E g(H)H\; -\; \E\lf( \frac{\partial g(H)}{\partial \real(H)} \real(H')H'+ \frac{\partial g(H)}{\partial \im(H)} \im(H')H'\ri)\ri|\;\le\; C\|\nabla^2g\|_\infty \E |H|^3
 \ee for $H'$ distributed as $H$, independent of $H$.
 
 The following well known lemma can be found in \cite{McDiarmid} or \cite[Th. 6.2]{BLM}.  It states that functions of many independent random variables $X_i$ have sub-Gaussian tails, with variance bounded by a quantity which is additive in the variables $X_i$. 
\beg{lem}[McDiarmid] Let $X_1, \ld, X_n$ be independent random variables taking values in some spaces denoted by   $E_1, \ld, E_n$, let $$f:E_1\ti\cd\ti E_n\to \R$$ be a measurable function and set $Y=f(X_1, \ld, X_n).$ Define, for each $k=1,\ld, n$, \be\la{wbstss}c_k\deq \sup |f(x_1, \ld, x_{k-1}, y, x_{k+1}, \ld, x_n)-f(x_1, \ld, x_{k-1}, z, x_{k+1}, \ld, x_n)|,\ee where the supremum is taken over $x_i\in E_i$ for all $i\ne k$ and $y,z\in E_k$. Then for each $r\ge 0$,  we have \be\la{eq:i7}\p(|Y-\E Y|\ge r)\;\le \; 2\exp\lf(-\frac{2r^2}{c_1^2+\cd+c_n^2}\ri).\ee\en{lem}

 The following elementary consequence of McDiarmid's lemma is pointed out in \cite{BCC2} (see also \cite{ShcherbinaTirozzi2010,PasturBook}). It implies for example that the fluctuations of $s(z)$ around its mean have order at most $N^{-1/2}$, as if the eigenvalues had been some i.i.d.\ $L^2$ random variables. It is well known (see e.g.\ \cite{bai-silver-book}) that in fact,   for Wigner matrices, as a consequence of the eigenvalue repulsion phenomenon, these fluctuations have order $N^{-1}$. However, for heavy-tailed or Erd\H{o}s-R\'enyi random matrices, the fluctuations can have any order between $N^{-1}$ and $N^{-1/2}$ (see \cite{ACFTCL,HHTMF}). 
\beg{lem}Let $H$ be a random  $N\ti N$ Hermitian matrix whose semi-columns $(H_{i,1})_{1\le i\le N}$, $(H_{i,2})_{2\le i\le N}$, \ld, $(H_{i,N})_{N\le i\le N}$ are independent random vectors. Then for any $z\in \C_+$ and any $r\ge 0$, \be\la{eq:i8}\p(|s(z)-\E s(z)|\ge r)\;\le \; 2\mathrm{e}^{-\frac{N\eta^2r^2}{8}}.\ee
\en{lem}

\bpr It suffices to notice that, by  \eqre{res_exp}, a variation of one of the semi-columns of $H$ affects $G(z)$ by an additive perturbation by a matrix with rank at most $2$ and operator norm at most $2/\eta$. 
\epr

\bpr[Proof of Theorem \ref{Th:Global_Law_appendix}]
Note that we have $ G(z)=z^{-1}(-1+HG(z))$, so that \beq \E s(z)&=&-\ff{z}+\ff{Nz}\E \Tr HG \\
&=&-\ff{z}+\ff{Nz}\sum_{k} \E H_{kk}G_{kk}  + \ff{Nz}\sum_{ k\ne  l} \E  H_{kl}G_{lk} \,.
\eeq
 Set $E_{kl\pm lk}$ to be the matrix with all entries equal to zero, except the $(k,l)$-th and $(l,k)$-th ones, which are respectively equal to $1$ and $\pm 1$. Set also $E_{kk}$ to be the  matrix with all entries equal to zero, except the $(k,k)$-th one, which is equal to $1$. 
Finally,     define $$\al_{kl}\deq N\E \real (H_{kl})H_{kl}\qquad \trm{ and }\qquad \bet_{kl}\deq N\E \im (H_{kl})H_{kl}\,.$$
Using \eqre{res_exp} to differentiate, it follows  from   \eqre{eq:SteinC}  that 
\beq  \E s(z)&=&-\ff{z}-   \ff{N^2z} \sum_kN\E H_{kk}^2\E (GE_{kk}G)_{kk} 
- \ff{N^2z}\sum_{k\ne l} \al_{kl} \E (GE_{kl+lk}G)_{lk}\\ &&
- \frac{\mathrm{i}}{N^2z}\sum_{k\ne l } \bet_{kl} \E (GE_{kl-lk}G)_{lk} +O(\cal E(z,N))
 \eeq 
where we use the shorthand $\cal E(z,N) \deq \max\{\eta^{-2}, \eta^{-3}\} (N^{-1/2}+\delta_N)$.
 Hence 
\beq  \E s(z) \;=\; -\ff{z} 
- \ff{N^2z}\sum_{k, l} \al_{kl} \E[G_{lk}^2+G_{kk}G_{ll}]
- \frac{\ii}{N^2z}\sum_{k, l } \bet_{kl} \E[G_{lk}^2-G_{kk}G_{ll}]  +O(\cal E(z,N))\, .
 \eeq 
Next, note that $$\Big|\sum_{k,l} G_{lk}^2\Big|=\lf|\Tr GG^T\ri|\le N\eta^{-2}\,.$$
 Using the fact that $\al_{kl}-\mathrm{i}\bet_{kl}=1$, it follows that $$  \E s(z)\;=\;-\ff{z} 
- \ff{N^2z}\sum_{k, l} \E G_{kk}G_{ll}
   +O(\cal E(z,N))\;=\;-\ff{z} 
- \ff{z}  \E s(z)^2 
   +O(\cal E(z,N))\,.
$$
But by \eqre{eq:i8}, $$|\E s(z)^2-(\E s(z))^2|\;\le \;\frac{16}{\eta^2N}\,,$$ so 
$$  \E s(z)\;=\; -\ff{z} 
- \ff{z}  (\E s(z))^2 
   +O(\cal E(z,N))\,.$$
We then conclude the proof using Lemma \ref{lem13} and \eqre{eq:i8}.\epr

We conclude this appendix with a remark about Theorem \ref{Th:Global_Law_appendix} applied to sparse random matrices. Sparse random matrices arise typically as adjacency matrices of random graphs. For instance, let $G(N,p)$ denote the \emph{Erd\H{o}s-R\'enyi graph}, defined as the (simple) random graph on $N$ vertices where each edge $\{i,j\}$ is open with probability $p$ and closed with probability $1 - p$, independently of the other edges. The graph is encoded by its adjacency matrix $A = A^* = \{0,1\}^{N \times N}$, whereby $A_{ij} = A_{ji}$ is equal to one if and only if the edge $\{i,j\}$ is open in $G(N,p)$.
The diagonal entries of $A$ are zero, and the strict upper-triangular entries $(A_{ij} \col 1 \leq i < j \leq N)$ are i.i.d.\ Bernoulli random variables with law $(1-p)\del_0+p\del_1$. We define the normalized adjacency matrix $B\deq(Np(1-p))^{-1/2}A$, and denote by $s_B(z)$ the Stieltjes transform of its empirical eigenvalue distribution (see \eqref{def_s}). Then Theorem \ref{Th:Global_Law_appendix} implies that for any fixed $z \in \C_+$ we have $s_B(z) \to m(z)$ in probability provided that $Np(1-p) \to \infty$ as $N \to \infty$.

For the proof, we define $H \deq B - \E B$. Then it is an easy exercise to check that $H$ satisfies the assumptions (i)--(iii) of Theorem \ref{Th:Global_Law_appendix} provided that $Np(1-p) \to \infty$. Thus, Theorem \ref{Th:Global_Law_appendix} implies that $s_H(z) \to m(z)$ in probability. To return to $B$, we note that
\begin{equation*}
B \;=\; H - \frac{p}{\sqrt{N p (1 - p)}} I + \sqrt{\frac{Np}{1 - p}} \f e \f e^*\,,
\end{equation*}
where we defined the unit vector $\f e \deq N^{-1/2} (1,1, \dots, 1)^*$. Thus, $B$ differs from $H$ by a small shift proportional to the identity plus a matrix of rank one, and the convergence $s_B(z) \to m(z)$ follows from the convergence $s_H(z) \to m(z)$ combined with the following general deterministic result.

\begin{lem}[Small perturbations do not affect the global regime]
Let $H$ and $\tilde H$ be Hermitian $N \times N$ matrices, and denote by $s(z)$ and $\tilde s(z)$ the Stieltjes transforms of their empirical eigenvalue distributions respectively (see \eqref{def_s}). Then we have
\begin{equation*}
\abs{s(z) - \tilde s(z)} \;\leq\; \frac{\operatorname{rank}(H - \tilde H)}{N} \min \hbb{\frac{2}{\eta}, \frac{\norm{H - \tilde H}}{\eta^2}}\,.
\end{equation*}
\end{lem}

\begin{proof}
We drop the arguments $z$ and denote by $G$ and $\tilde G$ the Green functions of $H$ and $\tilde H$ respectively. Then the resolvent identity yields
\begin{equation*}
s - \tilde s \;=\; \frac{1}{N} \tr \pb{G - \tilde G} \;=\; \frac{1}{N} \tr \pb{ G (\tilde H - H) \tilde G}\,.
\end{equation*}
Moreover, by the same resolvent identity we have $\operatorname{rank}(G - \tilde G) \leq \operatorname{rank}(H - \tilde H)$. The claim now follows from the trivial bounds $\norm{G(z)} \leq \eta^{-1}$ and $\norm{\tilde G(z)} \leq \eta^{-1}$, combined with the general estimate $\abs{\tr X} \leq \operatorname{rank}(X) \norm{X}$ for any matrix $X$, which can be proved for instance by singular value decomposition of $X$. This concludes the proof.
%
\end{proof}


We note that also a local law has been established for the Erd\H{o}s-R\'enyi graph; see Section \ref{sec:beyondW} for references.

\section{Extreme eigenvalues: the F\"uredi-Koml\'os argument} \label{sec:FK}

We state here a theorem, essentially due to F\"uredi and Koml\'os in \cite{KF} (see also \cite{agz,bai-silver-book,VuCombinatorica,TAO2}), that we use in the proof of Proposition \ref{prop:bound_H} to ensure that $\norm{H} \leq C$ with high probability. As extreme eigenvalues are not the main subject of this text, in the proof given here a key combinatorial result, Vu's Lemma from \ref{lemFK}, is admitted without proof.

 \beg{Th}\la{Th:Furedi_Komlos}
 Let $H$ be a Wigner matrix.   Then for any fixed $\epsilon>0$, we have 
 $$ \norm{H}\;\le\; 2+\eps$$ with high probability.
 \en{Th}
 
The rest of this appendix is devoted to the proof of Theorem \ref{Th:Furedi_Komlos}. We remark that if, instead of Definition \ref{def:Wigner} (iii), we make the stronger assumption that $\sqrt{N} H_{ij}$ are uniformly sub-Gaussian (see \eqref{eq:subGaussian} below), then a simpler concentration argument, given in Appendix \ref{sec:sibG_ext}, yields the bound $\norm{H} \leq C$ with high probability, which is also sufficient to complete the proof of Proposition \ref{prop:bound_H} and may therefore be used to replace the arguments of this section.

Before proving Theorem  \ref{Th:Furedi_Komlos}, let us state the following lemma, whose proof is postponed after the proof of the theorem. 

\beg{lem}\la{Th:Furedi_Komlos0}
Let $\wH$ be a Hermitian matrix $\wH = \wH^*$ whose entries $\wH_{ij}$ satisfy the following conditions.
\begin{enumerate}
\item
The upper-triangular entries $(\wH_{ij} \col 1\le i\le j\le N)$ are independent.
\item
For all $i,j$ we have $\E \wH_{ij}=0$, and for $i\ne j$, we have $\E |\wH_{ij}|^2\le N^{-1}$.  
\item We have $\ds (\log N)^2\, \max_{i,j} \|\wH_{ij}\|_\infty \; \lto  \;0\, .$  \end{enumerate}  Then for any fixed $\epsilon>0$, we have 
$$\norm{\wH}\; \le\; 2+\eps$$
with high probability.
\en{lem}

\beg{rmk}\la{Rmk:Furedi_Komlos}  It follows obviously that if $\si=\si_N$ is a deterministic sequence tending to $1$, then the conclusion of Lemma \ref{Th:Furedi_Komlos0} also holds for the matrix $\si\wH$.  We then  deduce, by standard perturbation bounds (e.g.\ \cite[Cor. A.6]{agz}),   that for $M$ a deterministic Hermitian  matrix \st $\Tr (M^2) = o(1)$, the conclusion of the theorem also holds for the matrix $\si \wH+M$. 
\end{rmk}

\bpr[Proof of Theorem  \ref{Th:Furedi_Komlos}]Let us introduce several auxiliary matrices. We choose $\ka\in (0,1/2)$ and  define \bgt\ite $X = (X_{ij}) =\sqrt{N}H$,
\ite $Y = (Y_{ij})$ with $Y_{ij}\deq X_{ij}\ind{|X_{ij}|\le N^{\ka}}$,
\ite $M = (M_{ij})$ with $M_{ij}\deq \E Y_{ij}$,
\ite $\wH\deq \ff{\sqrt{N}\si}(Y-M)$ with $\si\deq \max_{ij}\si_{ij}$ and $\si_{ij}$ is the standard deviation of $Y_{ij}$.
\ent
Note first that by the Chebyshev inequality, for each $p\in \N$, if $C_p$ is a uniform bound on $\norm{\sqrt{N}H_{ij}}_p^p$, then for all $x> 0$, \be\la{3001150}\P(|X_{ij}|\ge x)\; \le \; \frac{C_p}{x^p}\,.\ee
The first consequence of \eqre{3001150} is that 
for  any $L>0$, there is $C=C(L)$ \st \be\la{30011501}\P(X\ne Y) \;\le\; CN^{-L}\,.\ee
Next, as \beq |M_{ij}|&=&|\E X_{ij}-\E X_{ij}\ind{|X_{ij}|> N^{\ka}}|\\ &=&|\E X_{ij}\ind{|X_{ij}|> N^{\ka}}|\\ &\le& \E |X_{ij}|\ind{|X_{ij}|> N^{\ka}}\\ &=&N^{\ka}\P(|X_{ij}|> N^{\ka})+\int_{N^{\ka}}^{+\infty}\P(|X_{ij}|> x)\ud x,\eeq
we know that for  any $L>0$, there is $C=C(L)$ \st \be\la{3001151}|M_{ij}|\le CN^{-L}.\ee
In the same way, as $\E Y_{ij}^2=1-\E |X_{ij}|^2\ind{|X_{ij}|>N^{\ka}}$, we have $$1-\E Y_{ij}^2= N^{2\ka}\P(|X_{ij}|> N^{\ka})+\int_{N^{\ka}}^{+\infty}2x\P(|X_{ij}|> x)\ud x\,,$$ so that   for  any $L>0$, there is $C=C( L)$ \st \be\la{3001152}\max_{i,j}|\si_{ij}-1|\le CN^{-L}.\ee
The first consequence of \eqre{3001151} and \eqre{3001152} is that for any $\eta\in (0, 1/2-\ka)$, for  $N$ large enough, for all $i,j$, we have \be\la{30011501M}\norm{ \wH_{ij}}_\infty \, \le  \,N^{-\eta}.\ee
As the entries of $\wH$ obviously satisfy the other hypotheses of Lemma \ref{Th:Furedi_Komlos0}, we deduce that the conclusion of this lemma holds for the extreme eigenvalues of $\wH$. By Remark \ref{Rmk:Furedi_Komlos} and the estimates   \eqre{3001152} and  \eqre{3001151}, we deduce    that  the conclusion of Lemma \ref{Th:Furedi_Komlos0}  also holds for the extreme eigenvalues of   $\ff{\sqrt{N}}Y$. At last, \eqre{30011501} and the union bound allow to conclude.
\epr

\bpr[Proof of Lemma \ref{Th:Furedi_Komlos0}]
We will prove the lemma thanks to the following equation, true for any $\eps>0$ and $k$ even   (the choice of $k$ will be specified later): \be\la{Eq:FKP1}\p(\norm{\wH}>2+\eps) \ \le \ \p(\Tr \wH^{k}\ge (2+\eps)^{k}) \ \le \ (2+\eps)^{-k} \E\Tr \wH^{k}.\ee

We have, for any $k$ even,   $$\E\Tr \wH^{k}\ = \ \sum_{\f i=(i_1, \ld, i_k)} \E \wH_{i_1i_2}\cd \wH_{i_{k}i_{1}} \ = \ \sum_{p=1}^k\sum_{\substack{\f i=(i_1,\ld, i_k)\\ |\{i_1, \ld, i_k\}|=p}}\E \wH_{i_1i_2}\cd \wH_{i_{k}i_{1}} \,.$$

For each $\f i=(i_1,\ld, i_k)$, let $G_{\f i}$ be the simple, non oriented graph with vertex set $V_{\f i}\deq \{i_1, \ld, i_k\}$ and edge set $E_{\f i}\deq \{\{i_t,i_{t+1}\}\ste t=1,\ld, k\}$ for $i_{k+1}\deq i_1$. For $p\deq |V_{\f i}|$, denote also by $W_{\f i}$ the $p$-tuple deduced from ${\f i}$ by removing any entry which has already appeared when running ${\f i}$ from $i_1$ to $i_k$ (for example, if ${\f i}=(1,2,5,1,3,2,4)$, then $W_{\f i}=(1,2,5,3,4)$).
For $\E \wH_{i_1i_2}\cd \wH_{i_{k}i_{1}}$ to be non zero, each edge of $G_{{\f i}}$ has to be visited at least twice by the path $\ga_{\f i}\deq (i_1, \ld, i_k,i_1)$.  It implies that $|E_{\f i}|\le k/2$, so that, as $G_{\f i}$ is connected, $|V_{\f i}|\le k/2+1$. 

For any $p=1, \ld, k/2+1$ and any   $W\in \{1, \ld, N\}^p$ with pairwise distinct entries, let $I_{k}(W)$ denote the set $k$-tuples  ${\f i}\in \{1, \ld, N\}^k$ \st \bgt\ite $W_{\f i}=W$,
\ite  each edge of $G_{{\f i}}$ has to be visited at least twice by the path $\ga_{\f i}$.
\ent
Then we have 
\beq \E\Tr \wH^{k}
&=&\sum_{p=1}^{k/2+1}\sum_{\substack{W\in \{1, \ld, N\}^p\\ \trm{with distinct entries}} }\sum_{\substack{{\f i}\in I_{k}(W)}}\E \wH_{i_1i_2}\cd \wH_{i_{k}i_{1}}\,.
\eeq
Note that Hypothesis (ii) implies that for \be\la{20615200}K \;\equiv\; K_N \;\deq\;   \sup_{i,j}  \|\wH_{ij}\|_\infty  \,,\ee    for any $m\ge 2$, \be\la{206152} \E |\wH_{ij}|^m \; \le \; N^{-1} K^{m-2}\,.\ee
Hence  that for any $p=1, \ld, k/2+1$ and any ${\f i}$ \st $|V_{\f i}|=p$, by \eqre{206152}, we have 
$$ \E \wH_{i_1i_2}\cd \wH_{i_{k}i_{1}} \; \le \; N^{-(p-1)}K^{k-2(p-1)}\,.$$ Indeed, there are at least $p-1$ independent factors in the product (because each $\ell\in \{2, \ld, k\}$ \st $i_\ell\notin\{i_1, \ld, i_{\ell-1}\}$ gives rise to a new independent factor). It follows that 
\beq  \E\Tr \wH^{k} 
&\le &\sum_{p=1}^{k/2+1}N^{-(p-1)}K^{k-2(p-1)}\sum_{\substack{W\in \{1, \ld, N\}^p\\ \trm{with distinct entries}} }|  I_{k}(W)| \eeq
Then, the key result we shall rely on is the following one, due to Vu in \cite[p. 735]{VuCombinatorica}.
\beg{lem}\la{lemFK}For any $p=1, \ld, k/2+1$ and any $W\in \{1, \ld, N\}^p$ with distinct entries, we have \be\la{206151}|  I_{k}(W)| \le \binom{k}{2p-2}2^{k+s+1}p^s(s+2)^s,\ee where $s\deq k-2(p-1)$.
\en{lem}

It follows that, using the notation $s=k-2(p-1)$, $p=\frac{k-s}{2}+1$,
\beq  \E\Tr \wH^{k} 
&\le &\sum_{p=1}^{k/2+1}N^{-(p-1)}K^{k-2(p-1)}N^p\binom{k}{2p-2}2^{k+s+1}p^s(s+2)^s\\
&=&N\sum_{s\in \{0,2,4, \ld, k\}} K^{s} \binom{k}{s}2^{k+s+1}\pbb{\frac{k-s}{2}+1}^s(s+2)^s
\eeq
Note that  for any $k\ge 16$ (with the notation $x\deq s/k$), $$
2k^{-2}\pbb{\frac{k-s}{2}+1}(s+2)= (1-x+2k^{-1})(x+4k^{-1})= (1-x)x+2k^{-1}(2-x)+8k^{-2}  \le  1\,.$$
Hence for  $k\ge 16$, \be\la{66151}  \E\Tr \wH^{k}  
\,\le\, 2^{k+1}N\sum_{s=0}^k\binom{k}{s}\lf(Kk^2 \ri)^s\,
=\, 2^{k+1}N\lf(1+Kk^2 \ri)^k\ee
Let us now  choose $k\equiv k_N$ even \st for $K$ as in \eqre{20615200}, we have  $$(\log N)/k\;\lto\; 0\qquad\trm{ and } \qquad   k^2K\;\lto\; 0$$(such a $k$ exists by Hypothesis (iii)). 
Then by \eqre{66151}, for any $\eps,D>0$, we have $$N^D(2+\eps)^{-k}\E\Tr \wH^{k}\;\lto\; 0 \,,$$
which, by \eqre{Eq:FKP1}, concludes the proof of   Lemma \ref{Th:Furedi_Komlos0}. \epr

\section{Extreme eigenvalues: sub-Gaussian entries} \label{sec:sibG_ext}

In this appendix we provide a simple alternative proof that the extreme eigenvalues of $H$ are bounded, following \cite{ESY2}. This argument requires sub-Gaussian entries and gives a weaker bound than Theorem \ref{Th:Furedi_Komlos}, establishing that $\norm{H} \leq C$ for some constant $C$ with high probability. On the other hand, it is very simple and gives strong bounds on the error probabilities. Note that any bound of the form $\norm{H} \leq C$ is sufficient for the proof of Proposition \ref{prop:bound_H}, so that the argument given in this appendix may be used to replace entirely that of Appendix \ref{sec:FK} if one strengthens the decay assumption in (iii) of Definition \ref{def:Wigner} to sub-Gaussian decay as in  \eqre{eq:subGaussian}.

\beg{Th} Let $H$ be a Hermitian $N\times N$ matrix with independent centred  upper-triangular entries such that for some positive constant $\tau$ we have \be\la{eq:subGaussian}\E \me^{\tau N|H_{ij}|^2}\;\le\; \tau^{-1}
\ee for all $i,j$. Then there are positive constants $c,C$ depending only on $\tau$ such that for all $t>0$ we have $$\P(\|H\|\ge t)\;\le\; \me^{-c(t^2-C)N}\,.$$
\en{Th}

\bpr It is elementary (see e.g.\ \cite[Lem. 1.4.2]{CGLP}) that for any $N$ and any   $\eps>0$, one can find a (deterministic) $\epsilon$-net $B$ of the unit sphere of $\C^N$ with cardinality $\abs{B} \leq ((2+\eps)/\eps)^{2N}$. This means that for any $\f v$ in the unit sphere there exists $\f w \in B$ such that $\abs{\f v - \f w} \leq \epsilon$.  Taking $\epsilon < 1$,   we deduce that
$$\|H\|^2\;\le \;   \frac{\max_{\f v \in B} \abs{H \f v}^2}{1-\eps}\,.$$
Fixing  $\eps \in (0,1)$   and using a union bound, we therefore have to prove that for some positive constants $c,C$ and for any deterministic unit vector $\f v$ we have $$\P(|H \f v|^2\ge u)\;\le\; \me^{-c(u-C)N}$$
for all $u > 0$. By Markov's inequality, it suffices to prove that for some positive constants $c,C$ and for any deterministic unit vector $\f v$ we have
$$\E \me^{c N |H \f v|^2}\;\le\; \me^{c N C}\,.$$ 
For a unit vector $\f v=(v_1, \ld, v_N)$, using the independence of the upper-triangular entries of $H$ and the Cauchy-Schwarz inequality we obtain, for any $c>0$, 
\begin{align*}
\E \me^{c N |H\f v|^2} &\;\le\; \E \me^{2 c N \sum_i \abs{\sum_{j < i} H_{ij} v_j}^2 + 2 c N \sum_i \abs{\sum_{j \geq i} H_{ij} v_j}^2}
\\
&\;\leq\; \pB{\E \me^{4 c N \sum_i \abs{\sum_{j < i} H_{ij} v_j}^2} \, \E \me^{4 c N \sum_i \abs{\sum_{j \geq i} H_{ij} v_j}^2}}^{1/2}
\\
&\;=\; \prod_{i = 1}^N \pB{\E \me^{4c N|\sum_{j<i} H_{ij}v_j|^2} \, \E \me^{4c N|\sum_{j\ge i} H_{ij}v_i|^2}}^{1/2}\,,
\end{align*}
so that the claim follows from  Lemma \ref{lem:subGaussianscalarproduct} below.\epr

\beg{lem}\la{lem:subGaussianscalarproduct}Let $\f Y \in \C^N$ be a random vector with independent centred entries $Y_i$ satisfying
$$\E \me^{\delta |Y_i|^2}\;\le\; \delta^{-1}$$
for all $i$, where $\delta$ is a positive constant. Then there 
is a $\tau>0$ depending only on $\delta$  such that for any unit vector $\f v\in \C^N$ we have  $$\E\me^{\tau |\scalar{\f Y}{\f v}|^2}\;\le\; \tau^{-1}\,.$$
\en{lem}

\bpr
By decomposing both $\f Y$ and $\f v$ into real and imaginary parts and using H\"older's inequality, we can assume without loss of generality that both $\f Y$ and $\f v$ have real entries. 
Let then  $g$ be a standard real Gaussian variable, independent of the other variables, let $\E_g$ denote the expectation with respect to $g$ and let $\E$ denote the expectation with respect to all other variables than $g$.  We have, for any $\lambda>0$,   $$\E\me^{\lambda|\scalar{\f Y}{\f v}|^2}=\E \E_g \me^{\sqrt{2\lambda}\scalar{\f Y}{\f v} g} \;=\; \E_g \E \me^{\sqrt{2\lambda}\scalar{\f Y}{\f v} g} 
   \;=\;  \E_g\prod_{i} \E \me^{\sqrt{2\lambda} Y_i v_ig},$$ so that, by Lemma \re{lem:fromX2toX} below,  $$
  \E\me^{\lambda|\scalar{\f Y}{\f v}|^2} \;\le\;  \E_g\prod_{i} \me^{6\lambda g^2v_i^2 / \delta^3} \;\le\; \E_g \me^{6\lambda g^2 / \delta^3}\,,
$$
which  is finite  as soon as 
 $12\lambda< \delta^3 $.\epr

 \beg{lem}\la{lem:fromX2toX}For any real centred random variable $X$, $r\in \R$, and $\del>0$, we have 
 $$\E \me^{rX}\;\le \;  1+3(\me^{r^2/\del^{2}}-1)\E \me^{\del X^2}\
 \;\le \; \exp \pb{3r^2\del^{-2}\E \me^{\del X^2}}\,.$$
  \en{lem}
  
  \bpr 
   The second inequality follows from the fact that for any $y\ge 1$ we have   $$1+3(\me^{r^2/\del^{2}}-1)y\;=\;1+\sum_{n\ge 1}\frac{3y(r^{2}/\del^{2})^n}{n!}\;\leq\; 1+\sum_{n\ge 1}\frac{(3yr^2/\del^{2})^n}{n!}\;=\;\exp \pb{3r^2\del^{-2}y}\,.$$
To prove the first inequality, it suffices to consider the case $r=1$ (otherwise consider $\tilde{X} \deq rX$ and $\tilde{\del}\deq\del/r$). 
Using that $\me^x\le 1+x+3\frac{\me^x+\me^{-x}-2}{2}$ for all $x\in \R$, we get $$\E \me^X\;\leq\; 1+3\sum_{n\ge 1}\frac{\E X^{2n}}{(2n)!}\;\leq\; 1+3\sum_{n\ge 1}\frac{n!\del^{-2n}\E \me^{\del X^2}}{(2n)!}\;\leq\; 1+3 (\me^{1/\del^{2}}-1)\E \me^{\del X^2}\,,
 $$
as claimed.
  \epr

\section*{Acknowledgements}

We gratefully acknowledge the hospitality and support of the Institut Henri Poincar\'e and the Soci\'et\'e Math\'ematique de France during the conference \emph{\'Etats de la recherche en matrices al\'eatoires} in December 2014. Antti Knowles was partly supported by Swiss National Science Foundation grant 144662 and the SwissMAP NCCR grant.

\begin{thebibliography}{10}

\bibitem{AjankiErdosKruger} O. Ajanki, L. Erd{\H o}s, T. Kr\"uger,  \emph{Quadratic vector equations on complex upper half-plane}. Preprint arXiv:1506.05095.

\bibitem{AjankiErdosKruger2} O. Ajanki, L. Erd{\H o}s, T. Kr\"uger,  \emph{Universality for general Wigner-type matrices}. Preprint arXiv:1506.05098.

\bibitem{agz} G. Anderson, A. Guionnet, O. Zeitouni, \emph{An Introduction to Random Matrices}. Cambridge Studies in Advanced Mathematics 118, Cambridge, 2009.
 
\bibitem{bai-silver-book} Z.D. Bai, J.W. Silverstein, \emph{Spectral analysis of large dimensional random matrices}. Second Edition, Springer, New York, 2009.

\bibitem{BaoErdosblockband} Z. Bao, L. Erd{\H o}s,  \emph{Delocalization for a class of random block band matrices}.  Preprint arXiv:1503.07510. 

\bibitem{BaoErdosSchnelli1} Z. Bao, L. Erd{\H o}s, K. Schnelli, \emph{Local Stability of the Free Additive Convolution}. Preprint arXiv:1508.05905. 

\bibitem{BaoErdosSchnelli2} Z. Bao, L. Erd{\H o}s, K. Schnelli, \emph{Local law of addition of random matrices on optimal scale}. Preprint arXiv:1509.07080.

\bibitem{LocalLawRegGraphs} R. Bauerschmidt, A.  Knowles, H.-T. Yau, \emph{Local semicircle law for random regular graphs}. Preprint arXiv:1503.08702.

\bibitem{BG:SRTLL} F. Benaych-Georges, \emph{Local Single Ring Theorem}. 	Preprint arXiv:1501.07840.

\bibitem{BEN2} F. Benaych-Georges, A. Guionnet, M. Ma\"{\i}da, \emph{Fluctuations of the extreme eigenvalues of finite rank deformations of random matrices}. Electron. J. Probab. Vol. 16 (2011),  no. 60, 1621--1662.

   \bibitem{ACFTCL}  F. Benaych-Georges,  A. Guionnet, C. Male, \emph{Central limit theorems for linear statistics of heavy tailed random matrices}.  Comm. Math. Phys. 329 (2014), no. 2, 641--686.
   
      \bibitem{HHTMF}  F. Benaych-Georges,  A. Maltsev, \emph{Fluctuations of linear statistics of half-heavy-tailed random matrices}.  Preprint arXiv:1410.5624.

\bibitem{BEKYYPCA} A. Bloemendal, L. Erd{\H o}s, A. Knowles, H.-T. Yau,  J. Yin, \emph{Isotropic local laws for sample covariance and generalized Wigner matrices}. Electron. J. Probab. 19 (2014), no. 33, 1--53.

\bibitem{BKYYPCA} A. Bloemendal, A. Knowles, H.-T. Yau,  J. Yin, \emph{On the principal components of sample covariance matrices}. To appear in {Probab. Theory Related Fields}.

\bibitem{BMPastur} L. Bogachev, S.A. Molchanov, L.A.  Pastur, \emph{On the density of states of random band matrices}. (Russian) Mat. Zametki 50 (1991), no. 6, 31--42; translation in 
Math. Notes 50 (1992), no. 5--6, 1232--1242.

     \bibitem{BCC2} C.  Bordenave, P.  Caputo,    D.  Chafa{\"{\i}},
    \emph{Spectrum of non-Hermitian heavy tailed random matrices}.
  Comm. Math. Phys.
 307 (2011), no. 2,  513--560.
 
     \bibitem{CharlesAlice} C.  Bordenave, A. Guionnet, \emph{Localization and delocalization of eigenvectors for heavy-tailed random matrices}. Probab. Theory Related Fields 157 (2013), no. 3--4, 885--953.

   \bibitem{BLM} S. Boucheron, G. Lugosi, P. Massart, \emph{Concentration inequalities}. Oxford  University Press, Oxford, 2013.

\bibitem{BEY1} P. Bourgade, L. Erd\H{o}s, H.-T. Yau, \emph{Bulk universality of general $\beta$-ensembles with non-convex potential}. J. of Math. Phys. 53, special issue in honor of E. Lieb's 80th birthday,  (2012), no. 9.

\bibitem{BEY2} P. Bourgade, L. Erd\H{o}s, H.-T. Yau, \emph{Universality of general $\beta$-ensembles}. Duke Math. J. 163 (2014), no. 6, 1127--1190.

\bibitem{BEY3} P. Bourgade, L. Erd\H{o}s, H.-T. Yau, \emph{Edge Universality of $\bet$-ensembles}.  Comm. Math. Phys. 332 (2014), no. 1, 261--353.

      \bibitem{BourgadeYau} P. Bourgade, H.-T. Yau, \emph{The eigenvector moment flow and local quantum unique ergodicity}.  Preprint arXiv:1312.1301.

      \bibitem{BourgadeCirc1}      
      P. Bourgade, H.-T. Yau, J. Yin, \emph{Local circular law for random matrices}, Probab. Theory Related Fields 159 (2014), no. 3, 545--595.

      \bibitem{BourgadeCirc2} P. Bourgade, H.-T. Yau, J. Yin,  \emph{The local circular law II: the edge case}.    Probab. Theory Related Fields  159  (2014), no. 3--4, 619--660.

      \bibitem{CMS}
      C. Cacciapuoti , A. Maltsev, B. Schlein, \emph{Bounds for the Stieltjes transform and the density of states of Wigner matrices}. Probab. Theory Related Fields 163  (2014), no. 1, 1--29.

      \bibitem{SchleinCirc}
      C. Cacciapuoti , A. Maltsev, B. Schlein, \emph{Local Marchenko-Pastur law at the hard edge of sample covariance matrices}. J. Math. Phys. 54 (2013).
 
   \bibitem{CGLP}  D. Chafa\"\i, O. Gu\'edon, G. Lecu\'e, A. Pajor,  \emph{Interactions between compressed sensing random matrices and high dimensional geometry}. Panoramas et Synth\`eses 37. Soci\'et\'e Math\'ematique de France, Paris,  2012.

 \bibitem{ChatterjeeLindeberg} S. Chatterjee, \emph{A generalization of the Lindeberg principle}. Ann. Probab. 34 (2006), no. 6, 2061--2076.

   \bibitem{Dyson}  F.J. Dyson, \emph{A Brownian-motion model for the eigenvalues of a random matrix}. J. Mathematical Phys. 3 (1962) 119--1198.

\bibitem{Erd1} L. Erd{\H o}s, \emph{Universality of Wigner random matrices: a survey of recent results}. Russian Mathematical Surveys 66 (2011), no. 3, 507--626.

\bibitem{EKQD}  L. Erd{\H o}s, A.  Knowles, \emph{Quantum diffusion and eigenfunction delocalization in a random band matrix model}. 
Comm. Math. Phys. 303 (2011), no. 2, 509--554. 

\bibitem{EKQDGeneral}  L. Erd{\H o}s, A.  Knowles, \emph{Quantum diffusion and delocalization for band matrices with general distribution}. Ann. Henri Poincar\'e 12 (2011), no. 7, 1227--1319.

\bibitem{EKYfluc}  L. Erd{\H o}s, A.  Knowles,   H.-T. Yau,
\emph{Averaging Fluctuations in Resolvents of Random Band Matrices}. Ann. Henri Poincar\'e 14 (2013), no. 8, 1837--1926.

\bibitem{EKYY1} L. Erd{\H o}s, A.  Knowles,   H.-T. Yau, J. Yin,
 \emph{Spectral Statistics of Erd\H{o}s-R\'enyi Graphs I: Local Semicircle Law}.
Ann. Probab. 41 (2013), no. 3B, 2279--2375.

\bibitem{EKYYERGII}
L. Erd{\H{o}}s, A. Knowles, H.-T. Yau,  J. Yin,  \emph{Spectral statistics of Erd\H{o}s-R\'enyi Graphs II: Eigenvalue spacing and the extreme eigenvalues}. Comm. Math. Phys. 314 (2012), no. 3, 587--640.

\bibitem{EKYY3}  L. Erd{\H o}s, A.  Knowles,   H.-T. Yau, J. Yin,
 \emph{Delocalization and Diffusion Profile for Random Band Matrices}.  Comm. Math. Phys. 323 (2013), no. 1, 367--416.

\bibitem{EKYY4}
L. Erd{\H{o}}s, A. Knowles, H.-T. Yau,  J. Yin,  \emph{The local semicircle law for a general class of random
  matrices}. Electron. J. Probab 18 (2013), 1--58.

    \bibitem{ESRY} L. Erd\H{o}s, J. Ram\'{\i}rez, B. Schlein, H.-T. Yau, \emph{Universality of sine-kernel for Wigner matrices with a small Gaussian perturbation}. 
Electron. J. Probab. 15 (2010), no. 18, 526--603. 

\bibitem{ESY2} L. Erd\H{o}s, B. Schlein, H.T. Yau, \emph{Semicircle law on short scales and delocalization of eigenvectors for Wigner random matrices}. Ann. Probab. 37 (2009), no. 3, 815--852.

\bibitem{ESY3} L. Erd\H{o}s, B. Schlein, H.T. Yau, \emph{Local semicircle law and complete delocalization for Wigner random matrices}. Comm. Math. Phys. 287 (2009), 641--655.

\bibitem{ESY4} L. Erd\H{o}s, B. Schlein, H.T. Yau, \emph{Wegner estimate and level repulsion for Wigner random matrices}. Int. Math. Res. Not. 2010 (2009), 436--479.

\bibitem{ESYY} L. Erd\H{o}s, B. Schlein, H.T. Yau,   J. Yin, \emph{The local relaxation flow approach to universality of the local statistics for random matrices}. 
Ann. Inst. Henri Poincar\'e Probab. Stat. 48 (2012), no. 1, 1--46. 

\bibitem{MR2981427}
L. Erd{\H{o}}s, H.-T. Yau,   J. Yin,
\emph{Bulk universality for generalized Wigner matrices}.
 Probab. Theor. Rel. Fields 154 (2012), no.  1--2, 341--407.

\bibitem{EYY2}  L. Erd{\H{o}}s,   H.-T. Yau,  J. Yin, \emph{Universality for generalized Wigner matrices with Bernoulli
distribution}.  J. of Combinatorics 1 (2011), no. 2, 15--85

\bibitem{EYYrigi}  L. Erd{\H{o}}s,   H.-T. Yau,  J. Yin, \emph{Rigidity of Eigenvalues of Generalized Wigner Matrices}.
Adv. Math.
 229 (2012), no. 3, 1435--1515.
 
 \bibitem{Girko} V.L. Girko, \emph{The circular law}. Teor. Veroyatnost. i Primenen. (Russian) 29 (1984), 669--679.

 \bibitem{KF} Z. F\"uredi, J. Koml\'os,  \emph{The eigenvalues of random symmetric matrices}.
Combinatorica 1 (1981), no. 3, 233--241.

\bibitem{Gaudin} M. Gaudin, \emph{Sur la loi limite de l'espacement des valeurs propres d'une matrice al\'eatoire}. Nuclear Physics
  25 (1961),   447--458.

\bibitem{KarginPTRF12} V. Kargin, \emph{A concentration inequality and a local law for the sum of two random matrices.}  Probab. Theory Related Fields 154 (2012), no. 3--4, 677--702.

\bibitem{KarginAOP13} V. Kargin, \emph{An inequality for the distance between densities of free convolutions.} Ann. Probab. 41 (2013), no. 5, 3241--3260.

\bibitem{KarginAOP13Sub} V. Kargin, \emph{Subordination of the resolvent for a sum of random matrices}. Ann. Probab. 43 (2015), no. 4, 2119--2150. 

\bibitem{Khor} A. Khorunzhy, \emph{On smoothed density of states for Wigner random matrices}. Rand. Op. Stoch. Eq. 4 (1997), no. 2, 147--162. 

 \bibitem{KnowlesYinEig} A. Knowles, J. Yin,  \emph{Eigenvector Distribution of Wigner Matrices}.  Probab. Theory Related Fields 155  (2013), 543--582.

 \bibitem{KnowlesYinIso} A. Knowles, J. Yin, \emph{The isotropic semicircle law and deformation of Wigner matrices}. 
 {Comm. Pure Appl. Math.} 66 (2013), no. 11, 1663--1750. 
 
  \bibitem{KnowlesYinOutliers} A. Knowles, J. Yin, \emph{The outliers of a deformed Wigner matrix}. Ann. Probab. 42 (2014), no. 5, 1980--2031.
 
  \bibitem{KnowlesYinAnIso} A. Knowles, J. Yin, \emph{Anisotropic local laws for random matrices}. Preprint arXiv:1410.3516.
 
 \bibitem{LINDEBERG} J.W. Lindeberg,  \emph{Eine neue Herleitung des Exponentialgesetzes in der Wahrscheinlichkeitsrechnung}. Math. Z. 15 (1922), 211--225.

\bibitem{MP1967} V.A. Marchenko, L.A. Pastur, \emph{Distribution of eigenvalues in certain sets of random matrices}. {Mat. Sb.} 72 (1967), no. 114, 507--536.

\bibitem{McDiarmid} C. McDiarmid, \emph{On the method of bounded differences}. Surveys in combinatorics,  London Math. Soc. Lecture Note Ser. 141, Cambridge Univ. Press, Cambridge (1989), pp. 148--188.

\bibitem{Mehta} M.L. Mehta, \emph{Random Matrices}. Second Edition, Academic Press, Boston, 1991.

\bibitem{MPasturK} S.A. Molchanov, L.A.  Pastur, A.M. Khorunzhii,  \emph{Distribution of the eigenvalues of random band matrices in the limit of their infinite order}. (Russian. English, Russian summary) Teoret. Mat. Fiz. 90 (1992), no. 2, 163--178; translation in 
Theoret. and Math. Phys. 90 (1992), no. 2, 108--118.

\bibitem{OrourkeVu} S. O'Rourke, V. Vu, 
\emph{Universality of local eigenvalue statistics in random matrices with external source}.
Random Matrices Theory Appl. 3 (2014), no. 2,   37 pp. 

\bibitem{PasturBook} L.A.  Pastur, M. Shcherbina, \emph{Eigenvalue distribution of large random matrices}. Mathematical Surveys and Monographs, 171. American Mathematical Society, Providence, RI, 2011.

\bibitem{SandrineSoshni} S. P\'ech\'e, A. Soshnikov, \emph{Wigner random matrices with non-symmetrically distributed entries}. J. Stat. Phys. 129 (2007), no. 5--6, 857--884.

\bibitem{PY1} N.S. Pillai, J. Yin, \emph{Universality of covariance matrices}. Ann. Appl. Probab. 24 (2014), no. 3, 935--1001. 

\bibitem{RRV2011JAMS} J.A. Ram\'\i rez, B. Rider, B. Vir\'ag, \emph{Beta ensembles, stochastic Airy spectrum, and a diffusion}. J. Amer. Math. Soc. 24 (2011), no. 4, 919--944.

  \bibitem{ShcherbinaTirozzi2010} M.~Shcherbina, B.~Tirozzi,  \emph{Central limit theorem for fluctuations of linear eigenvalue  statistics of large random graphs}. J. Math. Phys. 51 (2010), no. 2,  20 pp.

\bibitem{SinaiSoshni} Y.G. Sinai, A. Soshnikov,  \emph{A refinement of Wigner's semicircle law in a neighborhood of the spectrum edge for random symmetric matrices}. (Russian) Funktsional. Anal. i Prilozhen. 32 (1998), no. 2, 56--79, 96; translation in Funct. Anal. Appl. 32 (1998), no. 2, 114--131.

\bibitem{soshniCMP99} A. Soshnikov, \emph{Universality at the edge of the spectrum in Wigner random matrices}. Comm. Math. Phys. 207 (1999), no. 3, 697--733.

\bibitem{soshniJSP2002} A. Soshnikov, \emph{A note on universality of the distribution of the largest eigenvalues in certain sample covariance matrices}. Dedicated to David Ruelle and Yasha Sinai on the occasion of their 65th birthdays. J. Statist. Phys. 108 (2002), no. 5--6, 1033--1056.

\bibitem{stroock}
D. Stroock, \emph{Probability theory, and analytic view}. Second edition. Cambridge University
  Press, Cambridge, 2011.
  
  \bibitem{TAO2} T. Tao, 	\emph{Topics in random matrix theory}. Graduate Studies in Mathematics, American Mathematical Society, Providence, RI, 2012.

    \bibitem{Tao-Vu_CMP} T. Tao, V. Vu, \emph{Random matrices: universality of local eigenvalue statistics up to the edge}. Comm. Math. Phys. 298 (2010), no. 2, 549--572.
      
  \bibitem{Tao-Vu_ActaMath2011} T. Tao, V. Vu, \emph{Random matrices: universality of local eigenvalue statistics}. Acta Math. 206 (2011), no. 1, 127--204.
  
\bibitem{Tao-Vu_vectors} T. Tao, V. Vu, \emph{Random matrices: universal properties of eigenvectors}. Rand. Mat. Theor. Appl. 1 (2012), 1150001.
    
    \bibitem{TW94} C. Tracy, H. Widom, \emph{Level spacing distribution and Airy kernel}. Commun. Math. Phys. 159 
(1994), no.1, 151--174.

\bibitem{TW96}  C. Tracy, H. Widom,
\emph{On orthogonal and symplectic matrix ensembles}.
Commun. Math. Phys. 177 (1996), no. 3, 727--754.

\bibitem{VuCombinatorica} V. Vu, \emph{Spectral norm of random matrices}. Combinatorica 27  (2007), no. 6,  721--736.

\bibitem{Wig} E.P. Wigner, \emph{On the distribution of the roots of certain symmetric matrices}. Ann. of Math.  67 (1958),    325--327.

\bibitem{YinCirc} J. Yin, \emph{The local circular law III: general case}. Probab. Theory Related Fields  160  (2014), no. 3, 679--732.

\en{thebibliography}
 
  \end{document}